\documentclass[12pt,times]{amsart}
\usepackage{amssymb,amsmath,amsfonts,latexsym,color,amsthm,tikz}
\usepackage{mathtools}
\usepackage{times}
\usepackage{calc}
\usepackage[T1]{fontenc}
\usepackage[latin2]{inputenc}
\usepackage{graphicx}
\allowdisplaybreaks
\usepackage[all]{xy}
\usepackage{enumitem}
\def\id{\mathop{\rm Id}\nolimits}

\def\0D{\Delta^{(0)}}
\def\1D{\Delta^{(1)}}

\newcommand{\longto}{\longrightarrow}

\newtheorem{theorem}{Theorem}[section]
\newtheorem{remark}[theorem]{Remark}
\newtheorem{proposition}[theorem]{Proposition}
\newtheorem{lemma}[theorem]{Lemma}
\newtheorem{corollary}[theorem]{Corollary}
\newtheorem{conjecture}[theorem]{Conjecture}

\newtheorem{definition}[theorem]{Definition}

\def\build#1_#2^#3{\mathrel{
\mathop{\kern 0pt#1}\limits_{#2}^{#3}}}

\numberwithin{equation}{section}

\def\part{\partial}

\def\text{\hbox}

\def\build#1_#2^#3{\mathrel{
\mathop{\kern 0pt#1}\limits_{#2}^{#3}}}

\numberwithin{equation}{section}

\newcommand{\comment}[1]{\relax}

\title{\vspace*{-15mm}Graph algebras}
\author[P. M.~Hajac]{Piotr M.~Hajac}
\address[P. M.~Hajac]{Instytut Matematyczny, Polska Akademia Nauk, ul. \'Sniadeckich 8, Warszawa, 00-656 Poland;
and
Department of Mathematics,
University of Colorado Boulder,
2300 Colorado Avenue,
Boulder, CO 80309-0395,
USA}
\email{pmh@impan.pl }

\author[M.~Tobolski]{Mariusz Tobolski}
\address[M.~Tobolski]{Instytut Matematyczny, Polska Akademia Nauk, ul. \'Sniadeckich 8, Warszawa, 00-656 Poland}
\email{mtobolski@impan.pl }

\begin{document}
\baselineskip15pt
\begin{abstract}\normalsize
This introduction to graphs and graph algebras provides the optimal bound for the number of all
paths of length $k$ in a graph with $N\geq k$ edges and no loops. Our proof relies on a construction
of a number of terminating algorithms that reshape such graphs without ever decreasing the number
of paths of length~$k$. The key two algorithms work in turns each of them ending with a graph to which
the other algorithm can be applied. Finally, one arrives at a specific graph realizing the optimal bound.
Herein graph algebras mean path algebras and Leavitt path algebras. For the ground field $\mathbb{C}$
of complex numbers, the latter are viewed as dense subalgebras in their universal C*-completions called
graph C*-algebras. 
\end{abstract}
\maketitle
\vspace*{-10mm}
{\tableofcontents}

\vspace*{-15mm}
\section*{Acknowledgements}
\noindent
This  work is part of the  project ``New Geometry of Quantum Dynamics''
supported by EU grant H2020-MSCA-RISE-2015-691246 and co-financed by 
Polish Government grants 328941/PnH/2016 and  W2/H2020/2016/317281.
It is based on the advanced part of PMH's lecture course delivered at the University of Colorado Boulder
in the Spring Semester 2019. PMH is very grateful to his students whose curiosity-driven
attitude allow him to teach a lecture course based in ca.\ 20\% on ongoing research. Special
thanks go to Carla Farsi for hosting PMH as an Ulam Professor, and to research collaborators Alexandru
Chirvasitu (Theorem~1.10) and Sarah Reznikoff (Theorem~2.9).

\section{Graphs (quivers)}

\begin{definition}
A \underline{graph} is a quadruple $E:=(E^0,E^1,s,t)$, where:
\begin{itemize}
\item $E^0$ is the set of \underline{vertices},
\item $E^1$ is the set of \underline{edges} (arrows),
\item $E^1\overset{s}{\to}E^0$ is the \underline{source} map assigning to each edge its beginning,
\item $E^1\overset{t}{\to}E^0$ is the \underline{target} (range) map assigning to each edge its end.
\end{itemize}
\end{definition}
\noindent
For instance, consider the following graph
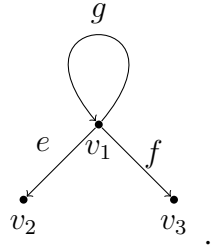
\begin{center}
\begin{tikzpicture}[auto,swap]
\tikzstyle{vertex}=[circle,fill=black,minimum size=3pt,inner sep=0pt]
\tikzstyle{edge}=[draw,->]
\tikzset{every loop/.style={min distance=20mm,in=130,out=50,looseness=50}}
    \node[vertex,label=below:$v_1$] (1) at (0,0) {};
    \node[vertex,label=below:$v_2$] (2) at (-1,-1) {};
    \node[vertex,label=below:$v_3$] (3) at (1,-1) {};
    \path (1) edge [edge, loop above] node {$g$} (1);
    \path (1) edge [edge] node {$e$} (2);
    \path (1) edge [edge] node[near end,above] {$f$} (3);
\end{tikzpicture}\,.
\end{center}
Here
\begin{align*}
s(e)=v_1,\qquad t(e)=v_2,\\
s(f)=v_1,\qquad t(f)=v_3,\\
s(g)=v_1,\qquad t(g)=v_1.
\end{align*}

\noindent
\underline{Elementary remarks}:
\begin{enumerate}
\item
The maps $s$ and $t$ need not be injective nor surjective.
\item
 If both $E^0$ and $E^1$ are empty, we call $E$ the \underline{empty} graph. The set $E^1$ might 
always be empty, but $E^0$ must not be empty if $E^1$ is not empty: every edge must have its
beginning and its end.
\item
 $E^0$ and $E^1$ might be infinite (usually, at most countable).
\end{enumerate}

\subsection{Paths}

\begin{definition}
Let $E$ be a graph. A \underline{finite path} in $E$ is a finite tuple $p_n:=(e_1,\ldots,e_n)$ of edges
satisfying 
\[
t(e_1)=s(e_2),\quad t(e_2)=s(e_3),\quad \ldots,\quad t(e_{n-1})=s(e_n).
\] 
The beginning $s(p_n)$ of $p_n$
is $s(e_1)$ and the end $t(p_n)$ of $p_n$ is $t(e_n)$. If $s(p_n)=t(p_n)$, we call $p_n$ 
a~\underline{loop}. An \underline{infinite path} is a sequence $(e_i)_{i\in\mathbb{N}}$ of edges 
satisfying 
\[
\forall\,i\in\mathbb{N}:\quad t(e_i)=s(e_{i+1}).
\]
\end{definition}
\begin{definition}
The \underline{length} of a path is the size of the tuple. Every edge is a path of length~$1$.
Vertices are considered as finite paths of length~$0$. The length of an infinite path is infinity.
\end{definition}

\noindent
\underline{Elementary remarks}:
\begin{enumerate}
\item 
The space $\text{FP}(E)$ of all finite paths in $E$ (vertices included) might be infinite even if $E$
is a finite graph (both $E^0$ and $E^1$ are finite):
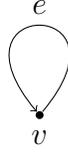
\begin{center}
\begin{tikzpicture}[auto,swap]
\tikzstyle{vertex}=[circle,fill=black,minimum size=3pt,inner sep=0pt]
\tikzstyle{edge}=[draw,->]
\tikzset{every loop/.style={min distance=20mm,in=130,out=50,looseness=50}}
    \node[vertex,label=below:$v$] (1) at (0,0) {};
    \path (1) edge [edge, loop above] node {$e$} (1);
\end{tikzpicture}
\end{center}
\[
E^0=\{v\},\quad E^1=\{e\},\quad \text{FP}(E)=\{v, e, (e,e), (e,e,e),\ldots\}.
\]
\item 
Examples of infinite paths:
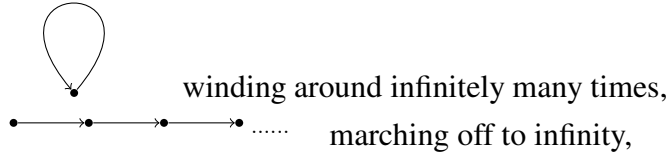
\begin{center}
\begin{tikzpicture}[auto,swap]
\tikzstyle{vertex}=[circle,fill=black,minimum size=3pt,inner sep=0pt]
\tikzstyle{edge}=[draw,->]
\tikzset{every loop/.style={min distance=20mm,in=130,out=50,looseness=50}}
    \node[vertex] (1) at (0,0) {};
    \path (1) edge [edge, loop above] node {} (1);
\end{tikzpicture}
winding around infinitely many times,
\end{center}
\begin{center}
\begin{tikzpicture}[auto,swap]
\tikzstyle{vertex}=[circle,fill=black,minimum size=3pt,inner sep=0pt]
\tikzstyle{edge}=[draw,->]
\tikzset{every loop/.style={min distance=20mm,in=130,out=50,looseness=50}}
    \node[vertex] (1) at (-1,0) {};
   \node[vertex] (2) at (0,0) {};
    \node[vertex] (3) at (1,0) {};
    \node[vertex] (4) at (2,0) {};
    \path (1) edge [edge] node {} (2);
    \path (2) edge [edge] node {} (3);
    \path (3) edge [edge] node {} (4);
\end{tikzpicture}
$^{......}\quad$ marching off to infinity,
\end{center}
or a combination of the above cases.
\end{enumerate}
\begin{theorem}
Let $E$ be a finite graph. Then $\mathrm{FP}(E)$ is finite if and only if there are no loops in $E$.
\end{theorem}
\begin{proof}
If there is a loop in $E$, then we have paths of arbitrary length, so there are infinitely many of them:
\[
e_1, \quad(e_1,e_2),\quad \ldots,\quad (e_1,\ldots, e_n),\quad (e_1,\ldots, e_n, e_1), 
\quad \text{etc.}
\]
Vice versa, if there are no loops, then edges in any path $(e_1,e_2,\ldots,e_n)$ cannot repeat
themselves: 
\[
e_i=e_j \Rightarrow i=j.
\]
Indeed, suppose the contrary: $e_i=e_j$ for $i<j$. Then 
\[
s(e_i)=s(e_j)=t(e_{j-1}),
\]
so the path $p_{ij}:=(e_i,\ldots,e_{j-1})$ is a loop:
\[
s(p_{ij})=s(e_{i})=t(e_{j-1})=t(p_{ij}),
\]
which contradicts our assumption of not having loops.

Therefore, the length of the longest possible path in $E$ is at most the number $N$ of all edges.
This yields the finite decomposition
\[
\text{FP}(E)=\text{FP}_0(E)\cup\text{FP}_1(E)\cup\ldots\cup \text{FP}_N(E),
\]
where $\text{FP}_k(E)$ is the space of all paths in $E$ of length $k$.
Furthermore, the sets $\text{FP}_0(E)=E^0$ and $\text{FP}_1(E)=E^1$
are finite by assumption. To construct a path of length $k$, first we must choose $k$ different edges
from the set of $N$ edges. We can do it in $\binom{N}{k}$ many ways. Then we can order these
$k$ edges into a path in at most $k!$ different ways, so there are at most
\[
k!\binom{N}{k}=\frac{N!}{(N-k)!}
\]
many paths of length $k$. 

Summarizing, $\text{FP}(E)$ is a finite union of finite sets, so it is finite.
\end{proof}

The estimate of the number of paths of length $k$ used in the above proof is far from optimal. Our goal now is to find the optimal estimate,
i.e.\ the estimate for which there exists a graph having exactly as many paths as allowed by the estimate.

\begin{definition}
Let $E$ be a graph, and let $p_n:=(e_1,\ldots, e_n)$ be a finite path of length at least one.
A \underline{subpath} $q_k$ of $p_n$ is a path $(e_i,e_{i+1},\ldots, e_{i+k})$,
where $i\in\{1,\ldots,n\}$ and $k\in\{0,\ldots,n-i\}$.
If $(e_n)_{n\in\mathbb{N}}$ is an infinite path, then any $(k+1)$-tuple $(e_i,\ldots,e_{i+k})$,
for any $i\in\mathbb{N}$, $k\in\mathbb{N}\cup\{\infty\}$, 
is a~subpath of $(e_n)_{n\in\mathbb{N}}$\,. Every source and every
target of each edge of a path (finite or infinite) is viewed as a subpath of length zero.
\end{definition}

\begin{theorem}
Let $E$ be any graph. If there exists a path $p$ (finite or infinite) whose edges can be rearranged
(permuted) into a path, then there exists a loop in $E$.
\end{theorem}
\begin{proof}
Let $S$ be a subset of $\mathbb{N}$ containing at least two elements, and let $\sigma: S\to S$ be a bijection that is not the idenity.
Since $\sigma\neq\id$, there exist the smallest $j\in S$ such that $\sigma(j)\neq j$. As $\sigma$ is bijective, $\sigma(j)>j$. 
Indeed, if $j$ is the smallest
element of $S$, we are done. If there is $i<j$, then $\sigma(j)\neq\sigma(i)=i$, so $\sigma(j)>j$. 
Furthermore,   $\sigma^{-1}(j)\neq j$. If $\sigma^{-1}(j)< j$, then
we get a contradiction: $j=\sigma(\sigma^{-1}(j))=\sigma^{-1}(j)< j$. Therefore, also  $\sigma^{-1}(j)> j$.
 
 Next, let $p:=(e_1,...,e_n,...)$ or $p:=(e_1,...,e_n)$. Then let $S:=\mathbb{N}$ or $S:=\{1,...,n\}$, respectively. Suppose now that 
 $p_\sigma:=(e_{\sigma(1)},...,e_{\sigma(n)},...)$ or $p_\sigma:=(e_{\sigma(1)},...,e_{\sigma(n)})$ is again a path for a bijection $\sigma$ as above.
 Then $(e_j,\ldots,e_{\sigma^{-1}(j)})$ is a subpath of~$p$, so $(e_{\sigma(j)},...,e_j)$ is a subpath of $p_\sigma$. Combining the latter path
 with the path $(e_j,e_{j+1},\ldots,e_{\sigma(j)-1},e_{\sigma(j)})$, we obtain a loop: 
\[
(e_{\sigma(j)},...,e_j,e_{j+1},\ldots,e_{\sigma(j)-1}).
\]
 Note that, if $\sigma(j)=j+1$, then the path $(e_{\sigma(j)},...,e_j)=(e_{j+1},...,e_j)$ is already a loop.
\end{proof}
\begin{corollary}
If $E$ is a graph with $N$ edges and no loops, then there are at most $\binom{N}{k}$
different paths of length $1\leq k\leq N$.
\end{corollary}
\begin{proof}
No loops in $E$ implies that edges cannot repeat themselves in any path, so one needs to choose
$k$ different edges from $N$ edges. By the above proposition, there is at most one way these $k$ different edges can form a path of length $k$.
\end{proof}

In any graph with $N\geq 1$ edges, there are exactly $\binom{N}{1}=N$ paths of length one, i.e. edges.
There is a graph
\begin{center}
\begin{tikzpicture}[auto,swap]
\tikzstyle{vertex}=[circle,fill=black,minimum size=3pt,inner sep=0pt]
\tikzstyle{edge}=[draw,->]
\tikzset{every loop/.style={min distance=20mm,in=130,out=50,looseness=50}}
    \node[vertex] (1) at (-1,0) {};
   \node[vertex] (2) at (0,0) {};
    \node (3) at (0.5,0) {\ldots};
    \node[vertex] (4) at (1,0) {};
    \node[vertex] (5) at (2,0) {};
   
    \path (1) edge [edge] node {$e_1$} (2) ;
    \path (4) edge [edge] node {$e_N$} (5);
\end{tikzpicture}
\end{center}
with $N$ edges and no loops with exactly $\binom{N}{N}=1$ path of length $N$. However,
there is \underline{no} graph with $3$ edges and no loops and $\binom{3}{2}=3$ different paths
of length $2$:
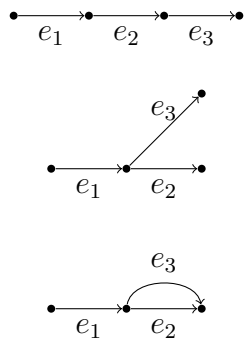
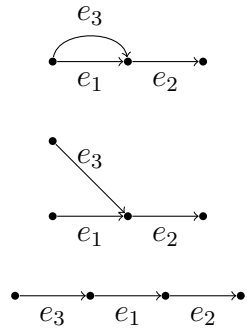
\begin{center}
\begin{tikzpicture}[auto,swap]
\tikzstyle{vertex}=[circle,fill=black,minimum size=3pt,inner sep=0pt]
\tikzstyle{edge}=[draw,->]
\tikzset{every loop/.style={min distance=20mm,in=130,out=50,looseness=50}}
    \node[vertex] (1) at (-1,0) {};
   \node[vertex] (2) at (0,0) {};
    \node[vertex] (3) at (1,0) {};
    \node[vertex] (4) at (2,0) {};
   
    \path (1) edge [edge] node {$e_1$} (2) ;
    \path (2) edge [edge] node {$e_2$} (3);
    \path (3) edge [edge] node {$e_3$} (4);
\end{tikzpicture}
\end{center}

\begin{center}
\begin{tikzpicture}[auto,swap]
\tikzstyle{vertex}=[circle,fill=black,minimum size=3pt,inner sep=0pt]
\tikzstyle{edge}=[draw,->]
\tikzset{every loop/.style={min distance=20mm,in=130,out=50,looseness=50}}
    \node[vertex] (1) at (-1,0) {};
   \node[vertex] (2) at (0,0) {};
    \node[vertex] (3) at (1,0) {};
    \node[vertex] (4) at (1,1) {};
    \path (1) edge [edge] node {$e_1$} (2);
    \path (2) edge [edge] node {$e_2$} (3);
    \path (2) edge [edge] node[above] {$e_3$} (4);
\end{tikzpicture}
\end{center}

\begin{center}
\begin{tikzpicture}[auto,swap]
\tikzstyle{vertex}=[circle,fill=black,minimum size=3pt,inner sep=0pt]
\tikzstyle{edge}=[draw,->]
\tikzset{every loop/.style={min distance=20mm,in=130,out=50,looseness=50}}
    \node[vertex] (1) at (-1,0) {};
   \node[vertex] (2) at (0,0) {};
    \node[vertex] (3) at (1,0) {};
    \path (1) edge [edge] node {$e_1$} (2);
    \path (2) edge [edge] node {$e_2$} (3);
    \path (2) edge [edge, bend left=80] node[above] {$e_3$} (3);
\end{tikzpicture}
\end{center}

\begin{center}
\begin{tikzpicture}[auto,swap]
\tikzstyle{vertex}=[circle,fill=black,minimum size=3pt,inner sep=0pt]
\tikzstyle{edge}=[draw,->]
\tikzset{every loop/.style={min distance=20mm,in=130,out=50,looseness=50}}
    \node[vertex] (1) at (-1,0) {};
   \node[vertex] (2) at (0,0) {};
    \node[vertex] (3) at (1,0) {};
    \path (1) edge [edge] node {$e_1$} (2);
    \path (2) edge [edge] node {$e_2$} (3);
    \path (1) edge [edge, bend left=80] node[above] {$e_3$} (2);
\end{tikzpicture}
\end{center}

\begin{center}
\begin{tikzpicture}[auto,swap]
\tikzstyle{vertex}=[circle,fill=black,minimum size=3pt,inner sep=0pt]
\tikzstyle{edge}=[draw,->]
\tikzset{every loop/.style={min distance=20mm,in=130,out=50,looseness=50}}
    \node[vertex] (1) at (-1,0) {};
   \node[vertex] (2) at (0,0) {};
    \node[vertex] (3) at (1,0) {};
    \node[vertex] (4) at (-1,1) {};
    \path (1) edge [edge] node {$e_1$} (2);
    \path (2) edge [edge] node {$e_2$} (3);
    \path (4) edge [edge] node[above] {$e_3$} (2);
\end{tikzpicture}
\end{center}

\begin{center}
\begin{tikzpicture}[auto,swap]
\tikzstyle{vertex}=[circle,fill=black,minimum size=3pt,inner sep=0pt]
\tikzstyle{edge}=[draw,->]
\tikzset{every loop/.style={min distance=20mm,in=130,out=50,looseness=50}}
    \node[vertex] (1) at (-1,0) {};
   \node[vertex] (2) at (0,0) {};
    \node[vertex] (3) at (1,0) {};
    \node[vertex] (4) at (2,0) {};
   
    \path (1) edge [edge] node {$e_3$} (2) ;
    \path (2) edge [edge] node {$e_1$} (3);
    \path (3) edge [edge] node {$e_2$} (4);
\end{tikzpicture}
\end{center}
There are at most $2$ different paths of lenght $2$.
\begin{proposition}
Let $E$ be a graph with $N\geq 2$ edges and no loops. Then there are at most two different paths of length $N-1$.
\end{proposition}
\begin{proof}
A path of length $k$ must be of the form
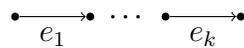
\begin{center}
\begin{tikzpicture}[auto,swap]
\tikzstyle{vertex}=[circle,fill=black,minimum size=3pt,inner sep=0pt]
\tikzstyle{edge}=[draw,->]
\tikzset{every loop/.style={min distance=20mm,in=130,out=50,looseness=50}}
    \node[vertex] (1) at (-1,0) {};
   \node[vertex] (2) at (0,0) {};
    \node (3) at (0.5,0) {\ldots};
    \node[vertex] (4) at (1,0) {};
    \node[vertex] (5) at (2,0) {};
   
    \path (1) edge [edge] node {$e_1$} (2) ;
    \path (4) edge [edge] node {$e_k$} (5);
\end{tikzpicture}
\end{center}
so, if we have at least one path of length $N-1$, our graph must be of the form
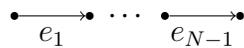
\begin{center}
\begin{tikzpicture}[auto,swap]
\tikzstyle{vertex}=[circle,fill=black,minimum size=3pt,inner sep=0pt]
\tikzstyle{edge}=[draw,->]
\tikzset{every loop/.style={min distance=20mm,in=130,out=50,looseness=50}}
    \node[vertex] (1) at (-1,0) {};
   \node[vertex] (2) at (0,0) {};
    \node (3) at (0.5,0) {\ldots};
    \node[vertex] (4) at (1,0) {};
    \node[vertex] (5) at (2,0) {};
   
    \path (1) edge [edge] node {$e_1$} (2) ;
    \path (4) edge [edge] node {$e_{N-1}$} (5);
\end{tikzpicture}
\end{center}
and $e_N$ attached somewhere. The only attachment possibilities increasing the number of paths
of length $N-1$ are:
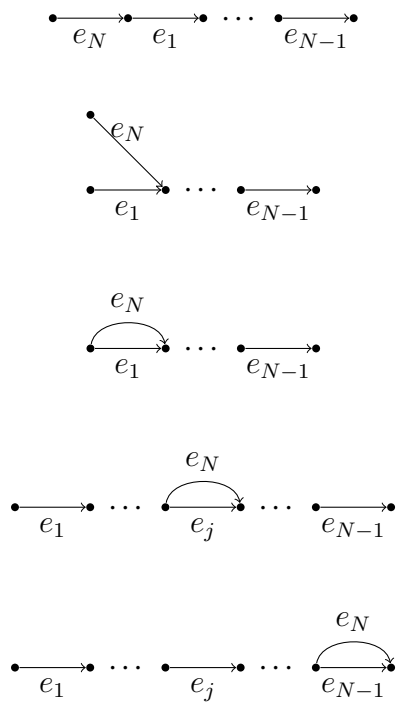
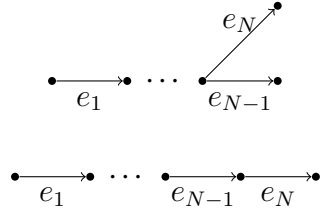
\begin{center}
\begin{tikzpicture}[auto,swap]
\tikzstyle{vertex}=[circle,fill=black,minimum size=3pt,inner sep=0pt]
\tikzstyle{edge}=[draw,->]
\tikzset{every loop/.style={min distance=20mm,in=130,out=50,looseness=50}}
   \node[vertex] (0) at (-2,0) {};
    \node[vertex] (1) at (-1,0) {};
   \node[vertex] (2) at (0,0) {};
    \node (3) at (0.5,0) {\ldots};
    \node[vertex] (4) at (1,0) {};
    \node[vertex] (5) at (2,0) {};
   
   \path (0) edge [edge] node {$e_N$} (1);
    \path (1) edge [edge] node {$e_1$} (2) ;
    \path (4) edge [edge] node {$e_{N-1}$} (5);
\end{tikzpicture}
\end{center}
\begin{center}
\begin{tikzpicture}[auto,swap]
\tikzstyle{vertex}=[circle,fill=black,minimum size=3pt,inner sep=0pt]
\tikzstyle{edge}=[draw,->]
\tikzset{every loop/.style={min distance=20mm,in=130,out=50,looseness=50}}
  \node[vertex] (0) at (-1,1) {};
    \node[vertex] (1) at (-1,0) {};
   \node[vertex] (2) at (0,0) {};
    \node (3) at (0.5,0) {\ldots};
    \node[vertex] (4) at (1,0) {};
    \node[vertex] (5) at (2,0) {};
   
   \path (0) edge [edge] node[above] {$e_N$} (2);
    \path (1) edge [edge] node {$e_1$} (2) ;
    \path (4) edge [edge] node {$e_{N-1}$} (5);
\end{tikzpicture}
\end{center}
\begin{center}
\begin{tikzpicture}[auto,swap]
\tikzstyle{vertex}=[circle,fill=black,minimum size=3pt,inner sep=0pt]
\tikzstyle{edge}=[draw,->]
\tikzset{every loop/.style={min distance=20mm,in=130,out=50,looseness=50}}
    \node[vertex] (1) at (-1,0) {};
   \node[vertex] (2) at (0,0) {};
    \node (3) at (0.5,0) {\ldots};
    \node[vertex] (4) at (1,0) {};
    \node[vertex] (5) at (2,0) {};
   
   \path (1) edge [edge, bend left=80, above] node {$e_N$} (2);
    \path (1) edge [edge] node {$e_1$} (2) ;
    \path (4) edge [edge] node {$e_{N-1}$} (5);
\end{tikzpicture}
\end{center}
\begin{center}
\begin{tikzpicture}[auto,swap]
\tikzstyle{vertex}=[circle,fill=black,minimum size=3pt,inner sep=0pt]
\tikzstyle{edge}=[draw,->]
\tikzset{every loop/.style={min distance=20mm,in=130,out=50,looseness=50}}
    \node[vertex] (1) at (-1,0) {};
   \node[vertex] (2) at (0,0) {};
    \node (3) at (0.5,0) {\ldots};
    \node[vertex] (4) at (1,0) {};
    \node[vertex] (5) at (2,0) {};
    \node (6) at (2.5,0) {\ldots};
    \node[vertex] (7) at (3,0) {};
    \node[vertex] (8) at (4,0) {};
   
    \path (1) edge [edge] node {$e_1$} (2) ;
  \path (4) edge [edge,above,bend left=80] node {$e_N$} (5);
    \path (4) edge [edge] node {$e_j$} (5);
    \path (7) edge [edge] node {$e_{N-1}$} (8);
\end{tikzpicture}
\end{center}
\begin{center}
\begin{tikzpicture}[auto,swap]
\tikzstyle{vertex}=[circle,fill=black,minimum size=3pt,inner sep=0pt]
\tikzstyle{edge}=[draw,->]
\tikzset{every loop/.style={min distance=20mm,in=130,out=50,looseness=50}}
    \node[vertex] (1) at (-1,0) {};
   \node[vertex] (2) at (0,0) {};
    \node (3) at (0.5,0) {\ldots};
    \node[vertex] (4) at (1,0) {};
    \node[vertex] (5) at (2,0) {};
    \node (6) at (2.5,0) {\ldots};
    \node[vertex] (7) at (3,0) {};
    \node[vertex] (8) at (4,0) {};
   
    \path (1) edge [edge] node {$e_1$} (2) ;
  \path (7) edge [edge,above,bend left=80] node {$e_N$} (8);
    \path (4) edge [edge] node {$e_j$} (5);
    \path (7) edge [edge] node {$e_{N-1}$} (8);
\end{tikzpicture}
\end{center}
\begin{center}
\begin{tikzpicture}[auto,swap]
\tikzstyle{vertex}=[circle,fill=black,minimum size=3pt,inner sep=0pt]
\tikzstyle{edge}=[draw,->]
\tikzset{every loop/.style={min distance=20mm,in=130,out=50,looseness=50}}
  \node[vertex] (0) at (2,1) {};
    \node[vertex] (1) at (-1,0) {};
   \node[vertex] (2) at (0,0) {};
    \node (3) at (0.5,0) {\ldots};
    \node[vertex] (4) at (1,0) {};
    \node[vertex] (5) at (2,0) {};
   
   \path (4) edge [edge] node[above] {$e_N$} (0);
    \path (1) edge [edge] node {$e_1$} (2) ;
    \path (4) edge [edge] node {$e_{N-1}$} (5);
\end{tikzpicture}
\end{center}
\begin{center}
\begin{tikzpicture}[auto,swap]
\tikzstyle{vertex}=[circle,fill=black,minimum size=3pt,inner sep=0pt]
\tikzstyle{edge}=[draw,->]
\tikzset{every loop/.style={min distance=20mm,in=130,out=50,looseness=50}}
   \node[vertex] (0) at (-2,0) {};
    \node[vertex] (1) at (-1,0) {};
   \node (2) at (-0.5,0) {\ldots};
    \node[vertex] (3) at (0,0) {};
    \node[vertex] (4) at (1,0) {};
    \node[vertex] (5) at (2,0) {};
   
   \path (0) edge [edge] node {$e_1$} (1);
    \path (3) edge [edge] node {$e_{N-1}$} (4);
    \path (4) edge [edge] node {$e_{N}$} (5);
\end{tikzpicture}
\end{center}
In each of the above cases, we have exactly two different paths of length $N-1$.
\end{proof}

\begin{lemma}
Let $E$ be a graph with $N\geq 2$ edges and no loops. Assume that $N\geq k\geq N-k$.
Then there are at most $2^{N-k}$ different paths of length $k$ in $E$ and the bound is optimal.
\end{lemma}
\begin{proof}
One can always construct a graph with a path $p_1$ of length $k$. Then there remain precisely $N-k$
many edges that can be used to construct more paths. Call the set of all these edges $F^1$.
Any path  of length $k$ is composed out of $l$ edges in $F^1$ and $k-l$ edges from the path $p_1$.
For instance:
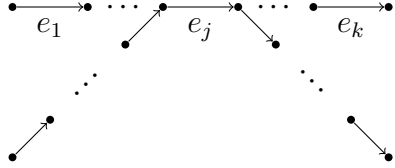
\begin{center}
\begin{tikzpicture}[auto,swap]
\tikzstyle{vertex}=[circle,fill=black,minimum size=3pt,inner sep=0pt]
\tikzstyle{edge}=[draw,->]
\tikzset{every loop/.style={min distance=20mm,in=130,out=50,looseness=50}}
    \node[vertex] (1) at (-1,0) {};
   \node[vertex] (2) at (0,0) {};
    \node (3) at (0.5,0) {\ldots};
    \node[vertex] (4) at (1,0) {};
    \node[vertex] (5) at (2,0) {};
    \node (6) at (2.5,0) {\ldots};
    \node[vertex] (7) at (3,0) {};
    \node[vertex] (8) at (4,0) {};
    \node[vertex] (9) at (0.5,-0.5) {};
    \node (10) at (0, -1) {\rotatebox{45}{\ldots}};
    \node[vertex] (11) at (-0.5,-1.5) {};
    \node[vertex] (12) at (-1,-2) {};
    \node[vertex] (13) at (2.5,-0.5) {};
    \node (14) at (3,-1) {\rotatebox{-45}{\ldots}};
    \node[vertex] (15) at (3.5,-1.5) {};
    \node[vertex] (16) at (4,-2) {};
   
    \path (1) edge [edge] node {$e_1$} (2) ;
    \path (4) edge [edge] node {$e_j$} (5);
    \path (7) edge [edge] node {$e_{k}$} (8);
  \path (9) edge[edge] node {} (4);
  \path (5) edge[edge] node {} (13);
  \path (12) edge[edge] node {} (11);
  \path (15) edge[edge] node {} (16);
\end{tikzpicture}
\end{center}
Any such a path is uniquely determined by the choice of $l$ edges from $F^1$ because there is always
only one way in which edges from the path $p_1$ can connect disconnected subpaths composed from
edges in $F^1$ and edges in a path cannot be rearranged. This gives at most $\binom{N-k}{l}$
possibilities for having paths of length $k$ with $l$ edges from $F^1$. As $l$ can vary from $0$
to $N-k$, there are at most $\sum_{l=0}^{N-k}\binom{N-k}{l}=2^{N-k}$ different paths
of length $k$. The bound is optimal because the graph
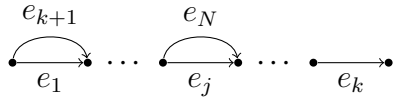
\begin{center}
\begin{tikzpicture}[auto,swap]
\tikzstyle{vertex}=[circle,fill=black,minimum size=3pt,inner sep=0pt]
\tikzstyle{edge}=[draw,->]
\tikzset{every loop/.style={min distance=20mm,in=130,out=50,looseness=50}}
    \node[vertex] (1) at (-1,0) {};
   \node[vertex] (2) at (0,0) {};
    \node (3) at (0.5,0) {\ldots};
    \node[vertex] (4) at (1,0) {};
    \node[vertex] (5) at (2,0) {};
    \node (6) at (2.5,0) {\ldots};
    \node[vertex] (7) at (3,0) {};
    \node[vertex] (8) at (4,0) {};
   
    \path (1) edge [edge] node {$e_1$} (2) ;
  \path (1) edge [edge,above,bend left=80] node {$e_{k+1}$} (2);
  \path (4) edge [edge,above,bend left=80] node {$e_N$} (5);
    \path (4) edge [edge] node {$e_j$} (5);
    \path (7) edge [edge] node {$e_{k}$} (8);
\end{tikzpicture}
\end{center}
has exactly $2^{N-k}$ edges of length $k$.
\end{proof}
\begin{theorem}\footnote{Joint work with Alexandru Chirvasitu.}
Let $E$ be a graph with $N\geq 2$ edges and no loops, and let $1\leq k\leq N=:nk+r$,
$0\leq r\leq k-1$. Then there are at most
\[
\boxed{
P^N_k:=(n+1)^rn^{k-r}
}
\]
different paths of length $k$ and the bound is optimal.
\end{theorem}
\begin{remark}{\rm 
For $k>N-k$, we have $N=1\cdot k+(N-k)$, so $P_N(k)=2^{N-k}\cdot 1^{k-(N-k)}=2^{N-k}$.
Also, if $k=N-k$, then $N=2\cdot k+0$, so $P_N(k)=3^0\cdot2^k=2^{N-k}$.
Hence the preceding proposition proves the theorem for $k\geq N-k$.}
\end{remark}
\begin{proof}
Our first step is to transform the graph $E$ into a graph $E_1$ with the same amount
of edges but with all vertices on its longest path $p_1$. We need to show that we can
always do this without introducing loops or decreasing the amount of different paths
of length $k$.
Clearly, we can first remove all vertices in $E^0$ that are not in $s(E^1)\cup t(E^1)$.
This way we end up with finitely many vertices. Furthermore, we identify unrelated
vertices. In any graph, we call a pair of vertices \underline{unrelated} iff there is
no path between them. If our graph admits a pair of unrelated vertices,
then we can choose such a pair and identify the vertices. We repeat the procedure
until there are no unrelated vertices. We call the thus obtained graph $E_1$.
\begin{lemma}
$E_1$ is a graph with $N$ edges, no loops and all vertices on its longest path $p_1$:
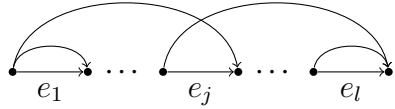
\begin{center}
\begin{tikzpicture}[auto,swap]
\tikzstyle{vertex}=[circle,fill=black,minimum size=3pt,inner sep=0pt]
\tikzstyle{edge}=[draw,->]
\tikzset{every loop/.style={min distance=20mm,in=130,out=50,looseness=50}}
    \node[vertex] (1) at (-1,0) {};
   \node[vertex] (2) at (0,0) {};
    \node (3) at (0.5,0) {\ldots};
    \node[vertex] (4) at (1,0) {};
    \node[vertex] (5) at (2,0) {};
    \node (6) at (2.5,0) {\ldots};
    \node[vertex] (7) at (3,0) {};
    \node[vertex] (8) at (4,0) {};
   
    \path (1) edge [edge] node {$e_1$} (2) ;
    \path (1) edge [edge, bend left=80] node {} (5);
  \path (1) edge [edge,above,bend left=80] node {} (2);
  \path (4) edge [edge,above,bend left=80] node {} (8);
    \path (4) edge [edge] node {$e_j$} (5);
    \path (7) edge [edge] node {$e_{l}$} (8);
    \path (7) edge [edge, bend left=80] node {} (8);
\end{tikzpicture}
\end{center}
It admits at least as many different paths of length $k$ as $E$.
\end{lemma}
\begin{proof}
If identifying two vertices $v_1$ and $v_2$ introduces a loop,
then breaking them apart destroys the loop. Hence the identified $v_1$ and
$v_2$ are on the loop, so there was a path from $v_1$ to $v_2$ or the other
way around, which means that $v_1$ and $v_2$ were not unrelated.
It follows that identifying unrelated vertices introduces no loops.
Next, suppose that all vertices are related but that there is a vertex $v$ that is not
on the path $p_1$:
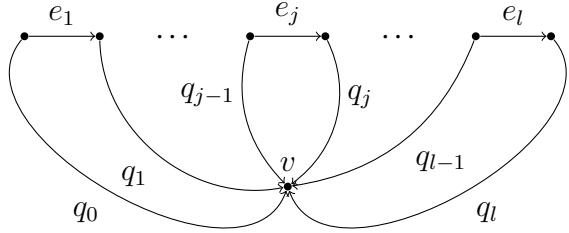
\begin{center}
\begin{tikzpicture}[auto,swap]
\tikzstyle{vertex}=[circle,fill=black,minimum size=3pt,inner sep=0pt]
\tikzstyle{edge}=[draw,->]
\tikzset{every loop/.style={min distance=20mm,in=130,out=50,looseness=50}}
    \node[vertex] (1) at (-1,0) {};
   \node[vertex] (2) at (0,0) {};
   \node (10) at (0.5,0) {};
    \node (3) at (1,0) {\ldots};
    \node (11) at (1.5,0) {};
    \node[vertex] (4) at (2,0) {};
    \node[vertex] (5) at (3,0) {};
    \node (12) at (3.5,0) {};
    \node (6) at (4,0) {\ldots};
    \node (13) at (4.5,0) {};
    \node[vertex] (7) at (5,0) {};
    \node[vertex] (8) at (6,0) {};
    \node[vertex, label=$v$] (9) at (2.5,-2) {};
   
    \path (1) edge [edge,above] node {$e_1$} (2) ;
    \path (4) edge [edge, above] node {$e_j$} (5);
    \path (7) edge [edge, above] node {$e_{l}$} (8);
  \path (1) edge[edge, below left, bend right=100] node {$q_0$} (9);
  \path (2) edge[edge, below left, bend right=50] node {$q_1$} (9);
  \path (4) edge[edge, above left, bend right] node {$q_{j-1}$} (9);
  \path (5) edge[edge, above right, bend left=40] node {$q_j$} (9);
  \path (7) edge[edge, below right,bend left] node {$q_{l-1}$} (9);
  \path (8) edge[edge, below right, bend left=100] node {$q_l$} (9);
\end{tikzpicture}
\end{center}
The path $q_0$ must go from $s(e_1)$ to $r$ as otherwise $p_1$ would not be
the longest path. Furthermore, the fact that $p_1$ is of maximal length forces
adjacent paths to have the same orientation. Hence all these paths, like $q_0$,
must end in $v$. However, $q_l$ ending in $v$ contradicts the maximality
of the length of $p_1$. Finally, $E_1$ has obviously at least as many paths
of length $\geq 1$ as $E$ because identifying vertices can only increase the number
of such paths.
\end{proof}
We can assume that the length of $p_1$ is $l\geq k$ as otherwise there are
no paths of length $k$. Our next step is to transform $E_1$ into a graph $E_2$
will all edges that start in $s(e_1)$ ending in $t(e_1)$:
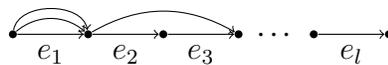
\begin{center}
\begin{tikzpicture}[auto,swap]
\tikzstyle{vertex}=[circle,fill=black,minimum size=3pt,inner sep=0pt]
\tikzstyle{edge}=[draw,->]
\tikzset{every loop/.style={min distance=20mm,in=130,out=50,looseness=50}}
   \node[vertex] (0) at (-2,0) {};
    \node[vertex] (1) at (-1,0) {};
   \node[vertex] (2) at (0,0) {};
    \node[vertex] (3) at (1,0) {};
    \node (6) at (1.5,0) {\ldots};
    \node[vertex] (4) at (2,0) {};
    \node[vertex] (5) at (3,0) {};
   
   \path (0) edge[edge] node {$e_1$} (1);
   \path (0) edge[edge, bend left=40] (1);
   \path (0) edge[edge, bend left=80] (1);
   \path (1) edge[edge] node {$e_2$} (2);
   \path (2) edge[edge] node {$e_3$} (3);
   \path (1) edge[edge, bend left] (3);
   \path (4) edge[edge] node {$e_l$} (5);
\end{tikzpicture}
\end{center}
If we have an edge starting in $s(e_1)$ but ending in $t(e_i)$, $i>1$, then we shift
the beginning of such an edge to $s(e_i)$. As there are no edges ending in $s(e_1)$,
we do not loose any paths this way. Now we transform $E_2$ into $E_3$ by shifting
the beginnings of edges from $s(e_2)$ to $t(e_j)$, $j>2$, to $s(e_3)$:
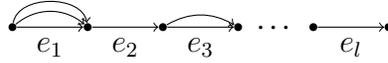
\begin{center}
\begin{tikzpicture}[auto,swap]
\tikzstyle{vertex}=[circle,fill=black,minimum size=3pt,inner sep=0pt]
\tikzstyle{edge}=[draw,->]
\tikzset{every loop/.style={min distance=20mm,in=130,out=50,looseness=50}}
   \node[vertex] (0) at (-2,0) {};
    \node[vertex] (1) at (-1,0) {};
   \node[vertex] (2) at (0,0) {};
    \node[vertex] (3) at (1,0) {};
    \node (6) at (1.5,0) {\ldots};
    \node[vertex] (4) at (2,0) {};
    \node[vertex] (5) at (3,0) {};
   
   \path (0) edge[edge] node {$e_1$} (1);
   \path (0) edge[edge, bend left=40] (1);
   \path (0) edge[edge, bend left=80] (1);
   \path (1) edge[edge] node {$e_2$} (2);
   \path (2) edge[edge] node {$e_3$} (3);
   \path (2) edge[edge, bend left] (3);
   \path (4) edge[edge] node {$e_l$} (5);
\end{tikzpicture}
\end{center}
This time possibly we loose the paths of length $k$ that started in $s(e_1)$ and
involved the just shifted edges, we possibly gain paths of length $k$ that start in $s(e_2)$
and involve the shifted edges. Let $a_i$ denote the number of edges starting at $s(e_i)$.
Then, if the shifted edges are the first of $x$ different paths of length $k-1$, we loose
$a_1\cdot x$ paths of length $k$ but gain $a_2\cdot x$ paths of length $k$.
To ensure that we gain at least as much as we loose, we transform $E_3$ to $E_4$
by switching places of the edges from $s(e_1)$ to $t(e_1)$ with the edges from
$s(e_2)$ to $t(e_2)$, if $a_1>a_2$:
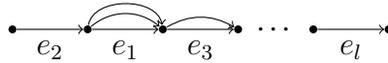
\begin{center}
\begin{tikzpicture}[auto,swap]
\tikzstyle{vertex}=[circle,fill=black,minimum size=3pt,inner sep=0pt]
\tikzstyle{edge}=[draw,->]
\tikzset{every loop/.style={min distance=20mm,in=130,out=50,looseness=50}}
   \node[vertex] (0) at (-2,0) {};
    \node[vertex] (1) at (-1,0) {};
   \node[vertex] (2) at (0,0) {};
    \node[vertex] (3) at (1,0) {};
    \node (6) at (1.5,0) {\ldots};
    \node[vertex] (4) at (2,0) {};
    \node[vertex] (5) at (3,0) {};
   
   \path (1) edge[edge] node {$e_1$} (2);
   \path (1) edge[edge, bend left=40] (2);
   \path (1) edge[edge, bend left=80] (2);
   \path (0) edge[edge] node {$e_2$} (1);
   \path (2) edge[edge] node {$e_3$} (3);
   \path (2) edge[edge, bend left] (3);
   \path (4) edge[edge] node {$e_l$} (5);
\end{tikzpicture}
\end{center}

As the number of paths of length $k$ beginning with a shifted edge is unchanged,
and the number of paths of length $k$ with a shifted edge as the second edge 
is not decreased, the number of paths of length $k$ involving a shifted edge
does not decrease. 
Now we have to make sure that transforming $E_2$ to $E_4$  
we did not decrease the number of all paths of length $k$ not involving
the shifted edges. 

If $k=1$, we are done because in any graph $E$ with
$N$ edges we have $P_N(1)=N$.
If $k\geq 2$, then any path of length $k$ in $E_2$ that does not involve any shifted edge and
 that starts in $s(e_1)$ must involve edges from $s(e_1)$ to $t(e_1)$
and from $s(e_2)$ to $t(e_2)$, 
so the number of paths of length $k$
not involving the shifted edges and starting at the leftmost vertex is
the same in $E_2$ as in $E_4$ even if $a_1>a_2$ and we made the switch:
$a_1a_2y=a_2a_1y$, where $y$ is the number of paths of length $k-2$
not involving the shifted edges and starting at $s(e_3)$.
(In the case $l=k=2$, we take $y=1$.)
Next, concerning the number of paths of length $k$ starting at the second
vertex from the left and not involving the shifted edges, it does not
decrease when moving from $E_2$ to $E_4$ as we have at least as many
edges going from the second to the third vertex and exactly as many paths
not involving the shifted edges starting at the third vertex in $E_2$
as in $E_4$. Finally, the number of paths of length $k$ not involving
the shifted edges and starting at the third or further vertex is unaffected
when going from $E_2$ to $E_4$.

We can continue this $E_2$-$E_3$-$E_4$ procedure until we obtain a graph $F_k$
whose all edges emitted from first $k$ vertices end in the consecutive vertex 
and with the number of edges satisfying the inequalities $a_1\leq a_2\leq\ldots\leq a_k$:
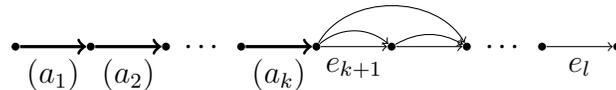
\begin{center}
\begin{tikzpicture}[auto,swap]
\tikzstyle{vertex}=[circle,fill=black,minimum size=3pt,inner sep=0pt]
\tikzstyle{edge}=[draw,->]
\tikzstyle{thickedge}=[draw,very thick,->]
\tikzset{every loop/.style={min distance=20mm,in=130,out=50,looseness=50}}
   
   \node[vertex] (-3) at (-5,0) {};
   \node[vertex] (-2) at (-4,0) {};
   \node[vertex] (-1) at (-3,0) {};
   \node (7) at (-2.5,0) {\ldots};
   \node[vertex] (0) at (-2,0) {};
    \node[vertex] (1) at (-1,0) {};
   \node[vertex] (2) at (0,0) {};
    \node[vertex] (3) at (1,0) {};
    \node (6) at (1.5,0) {\ldots};
    \node[vertex] (4) at (2,0) {};
    \node[vertex] (5) at (3,0) {};
   
   \path (-3) edge[thickedge] node {$(a_1)$} (-2);
   \path (-2) edge[thickedge] node {$(a_2)$} (-1);
   \path (1) edge[edge] node {$e_{k+1}$} (2);
   \path (1) edge[edge, bend left=40] (2);
   \path (1) edge[edge, bend left=60] (3);
   \path (0) edge[thickedge] node {$(a_{k})$} (1);
   \path (2) edge[edge] node {} (3);
   \path (2) edge[edge, bend left] (3);
   \path (4) edge[edge] node {$e_l$} (5);
\end{tikzpicture}
\end{center}
Indeed, take $m<k$ and apply the $E_2$-$E_3$-$E_4$ procedure to the graph $F_m$
defined as $F_k$ but with $k$ replaced by $m$. Assume that $a_j\leq a_{m+1}\leq a_{j+1}$\, for some $j$.
Then we move the beginning of any edge starting at the $(m+1)$ vertex and
ending at the $(m+3)$ vertex or further to the $(m+2)$ vertex. Next, we implement the
swap of edges:
\[
(m+1)\mapsto (j+1),\quad (j+1)\mapsto (j+2),\quad \ldots\quad ,\quad m\mapsto (m+1),
\]
and obtain:
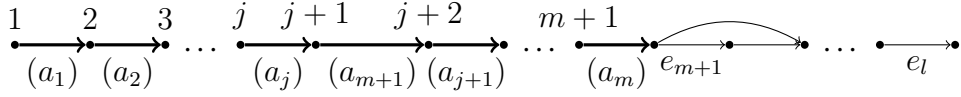
\begin{center}
\begin{tikzpicture}[auto,swap]
\tikzstyle{vertex}=[circle,fill=black,minimum size=3pt,inner sep=0pt]
\tikzstyle{edge}=[draw,very thick,->]
\tikzstyle{thinedge}=[draw,->]
\tikzset{every loop/.style={min distance=20mm,in=130,out=50,looseness=50}}
   
   \node[vertex, above, label=$1$] (0) at (0,0) {};
   \node[vertex, above, label=$2$] (1) at (1,0) {};
   \node[vertex, above, label=$3$] (2) at (2,0) {};
   \node (25) at (2.5,0) {\ldots};
   \node[vertex, above, label=$j$] (3) at (3,0) {};
   \node[vertex, above, label=$j+1$] (4) at (4,0) {};
   \node[vertex, above, label=$j+2$] (5) at (5.5,0) {};
   \node[vertex, above] (6) at (6.5,0) {};
   \node (65) at (7,0) {\ldots};
   \node[vertex,above,label=$m+1$] (7) at (7.5,0) {};
   \node[vertex,above] (8) at (8.5,0) {};
   \node[vertex,above] (9) at (9.5,0) {};
   \node[vertex,above] (10) at (10.5,0) {};
   \node (105) at (11,0) {\ldots};
   \node[vertex,above] (11) at (11.5,0) {};
   \node[vertex,above] (12) at (12.5,0) {};
   
   \path (0) edge[edge,below] node {$(a_1)$} (1);
   \path (1) edge[edge,below] node {$(a_2)$} (2); 
   \path (3) edge[edge,below] node {$(a_j)$} (4);
   \path (4) edge[edge,below] node {$(a_{m+1})$} (5);
   \path (5) edge[edge,below] node {$(a_{j+1})$} (6);
   \path (7) edge[edge,below] node {$(a_{m})$} (8);
   \path (8) edge[thinedge,below] node{$e_{m+1}$} (9);
   \path (8) edge[thinedge,below,bend left] node {} (10);
   \path (9) edge[thinedge] node {} (10);
   \path (11) edge[thinedge,below] node {$e_l$} (12);

\end{tikzpicture}
\end{center}
If the shifted edges were the first edges of $x_1$ paths of length $k-1$,
$x_2$ of length $k-2$, \ldots, and $x_{m+1}$ of length $k-(m+1)$, respectively, then by shifting the edges
we lost
\[
L:=a_mx_1+a_{m-1}a_mx_2+\ldots+a_1\ldots a_mx_m
\]
paths of length $k$, but, due to the re-ordering procedure, we gained
\begin{align*}
G:&=a_mx_1+a_{m-1}a_mx_2+\ldots+a_{j+1}\ldots a_mx_{m-j}+a_{m+1}a_{j+1}\ldots a_mx_{m-j+1}\\
&+\ldots a_2\ldots a_{m+1}a_{j+1}\ldots a_mx_m+a_1a_2\ldots a_{m+1}a_{j+1}\ldots a_mx_{m+1}.
\end{align*}
The first $m-j$ terms of $L$ are the same as in $G$, the next $j$ terms of $L$ are the same as in $G$,
the next $j$ terms in $L$ are no bigger then they are in $G$, and the last term in $G$ does not appear in $L$.
Thus, applying the $E_2$-$E_3$-$E_4$ procedure, we did not decrease the amount of paths of length $k$ involving the shifted edges.

Concerning the paths of length $k$ not involving the shifted edges, the re-arrangement procedure does not change the amount of paths starting at the first vertex, does
not decrease the amount of paths starting at the vertices $2$, \ldots, $m$, and does not change the amount of paths starting at $m+1$ or further.

We cannot apply the $E_2$-$E_3$-$E_4$ procedure any further because,
if $a_1>a_{k+1}$, swaping the edges starting at the first vertex with the edges
starting at the $(k+1)$ vertex will decrease the amount of paths of length $k$
beginning at the first vertex:
\[
a_1a_2\ldots a_k>a_{k+1}a_2\ldots a_k\,.
\]

Now we need to make a move decreasing the number of vertices. We identify the first vertex
with the $(k+1)$ vertex and shift all edges from the first vertex to the second vertex to 
become edges from the $(k+1)$ to $(k+2)$ vertex:
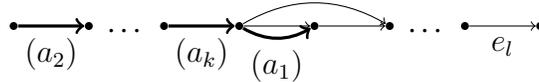
\begin{center}
\begin{tikzpicture}[auto,swap]
\tikzstyle{vertex}=[circle,fill=black,minimum size=3pt,inner sep=0pt]
\tikzstyle{edge}=[draw,very thick,->]
\tikzstyle{thinedge}=[draw,->]
\tikzset{every loop/.style={min distance=20mm,in=130,out=50,looseness=50}}
   
   \node[vertex, above] (1) at (1,0) {};
   \node[vertex, above] (2) at (2,0) {};
   \node (25) at (2.5,0) {\ldots};
   \node[vertex,above] (7) at (3,0) {};
   \node[vertex,above] (8) at (4,0) {};
   \node[vertex,above] (9) at (5,0) {};
   \node[vertex,above] (10) at (6,0) {};
   \node (105) at (6.5,0) {\ldots};
   \node[vertex,above] (11) at (7,0) {};
   \node[vertex,above] (12) at (8,0) {};
   
   \path (1) edge[edge,below] node {$(a_2)$} (2); 
   \path (7) edge[edge,below] node {$(a_{k})$} (8);
   \path (8) edge[thinedge,below] node{} (9);
   \path (8) edge[thinedge,below,bend left] node {} (10);
   \path (9) edge[thinedge] node {} (10);
   \path (11) edge[thinedge,below] node {$e_l$} (12);
   \path (8) edge[edge,below,bend right] node {$(a_1)$} (9);

\end{tikzpicture}
\end{center}
We call the thus obtained new graph $G_l$.

Let $G$ be a graph with finitely
many totaly ordered (by paths) vertices,
and $N$ edges. Denote by $P^N_k(G,m)$ the number of paths in $G$ of length $k>1$
that start at the $m$ vertex. Furthermore, let $y_m$ denote the number of all paths in $F_k$
of length $1\leq m\leq k$ starting at the $(k+1)$ vertex.
Then there are
\[
\sum_{m=1}^{l-k+1}P^N_k(F_k,m)
\]
many paths of length $k$ in $F_k$ and
\[
\sum_{m=1}^{l-k}P^N_k(G_l,m)
\]
many paths of length $k$ in $G_l$.
For the first sum, we have
\[
P^N_k(F_k,1)=a_1\ldots a_k,\qquad P^N_k(F_k,2)=a_2\ldots a_ky_1,\qquad \ldots\quad,
\]
\[
P^N_k(F_k,k)=a_ky_{k-1},\qquad P^N_k(F_k,k+1)=y_k\,.
\]
For the second sum, we have
\[
P^N_k(G_l,1)=a_2\ldots a_k(a_1+y_1),\qquad P^N_k(G_l,2)\geq a_3\ldots a_ky_2=P^N_k(F_k,3),\qquad \ldots\quad,
\]
\[
P^N_k(G_l,k-1)\geq a_ky_{k-1}=P^N_k(F_k,k),\qquad P^N_k(G_l,k)\geq y_k=P^N_k(F_k,k+1).
\]
Consequently,
\begin{align*}
\sum_{m=1}^{l-k+1}P^N_k,(F_k,m)&=P^N_k(G_l,1)+\sum_{m=3}^{l-k+1}P^N_k(F_k,m)\\
&=P^N_k(G_l,1)+\sum_{m=3}^{k+1}P^N_k(F_k,m)+\sum_{m=k+2}^{l-k+1}P^N_k(F_k,m)\\
&=P^N_k(G_l,1)+\sum_{m=2}^kP^N_k(F_k,m+1)+\sum_{m=k+2}^{l-k+1}P^N_k(G_l,m-1)\\
&=P^N_k(G_l,1)+\sum_{m=2}P^N_k(F_k,m+1)+\sum_{m=k+1}^{l-k}P^N_k(G_l,m)\\
&\leq P^N_k(G_l,1)+\sum_{m=2}^kP^N_k(G_l,m)+\sum_{m=k+1}^{l-k}P^N_k(G_l,m)\\
&=\sum_{m=1}^{l-k}P^N_k(G_l,m).
\end{align*}
Thus there are at least as many paths of length $k$ in $G_l$ as there are in $F_k$.
If $l=k+1$, we have the desired thick path:
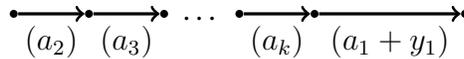
\begin{center}
\begin{tikzpicture}[auto,swap]
\tikzstyle{vertex}=[circle,fill=black,minimum size=3pt,inner sep=0pt]
\tikzstyle{edge}=[draw,very thick,->]
\tikzstyle{thinedge}=[draw,->]
\tikzset{every loop/.style={min distance=20mm,in=130,out=50,looseness=50}}
   
   \node[vertex, above] (0) at (0,0) {};
   \node[vertex, above] (1) at (1,0) {};
   \node[vertex, above] (2) at (2,0) {};
   \node (25) at (2.5,0) {\ldots};
   \node[vertex,above] (7) at (3,0) {};
   \node[vertex,above] (8) at (4,0) {};
   \node[vertex,above] (9) at (6,0) {};
   
   \path (0) edge[edge,below] node {$(a_2)$} (1);
   \path (1) edge[edge,below] node {$(a_3)$} (2); 
   \path (7) edge[edge,below] node {$(a_{k})$} (8);
   \path (8) edge[edge,below] node {$(a_1+y_1)$} (9);

\end{tikzpicture}
\end{center}
Otherwise we repeat the $E_2$-$E_3$-$F_k$-$G_l$ procedure
decreasing the amount of vertices by one but not decreasing
the amount of paths of legth $k$.

All this shows that we can always transform our graph into a graph with totally ordered $(k+1)$ vertices
that are on a path of length $k$ without changing
the amount $N$ of all edges and without decreasing the number of paths of length $k$. 
In such a graph, if there are still edges that begin and end not in consecutive vertices,
they do not contribute to paths of length $k$, so we can re-attach them so that
they begin and end in consecutive vertices.

Now, the final step is to prove that given a thick path
c
with differences between numbers of edges bigger than one, we can evenly re-distribute the edges increasing the number of paths of length $k$ to the bound $P_N(k)$.

If there are any two indices $i\neq j$ such that $b_i-b_j>1$, then we define
\[
b'_n:=\begin{cases}b_n & \text{for $n\neq i, j$}\\b_n-1 & \text{for $n=i$}\\b_n+1 & \text{for $n=j$}\end{cases}
\]
and compute
\begin{align*}
\prod_{n=1}^kb'_n&=\left(\prod_{n\in\{1,\ldots,k\}\setminus\{i,j\}}b_n\right)(b_i-1)(b_j+1)\\
&=\left(\prod_{n\in\{1,\ldots,k\}\setminus\{i,j\}}b_n\right)(b_ib_j+b_i-b_j-1)\\
&=\prod_{n=1}^k b_n+\left(\prod_{1,\ldots,k\setminus\{i,j\}}b_n\right)(b_i-b_j-1)>\prod_{n=1}^kb_n\,.
\end{align*}
We can repeat this procedure until there is no pair of indices $i\neq j$ with the property $b_i-b_j-1$.
Thus we arrive at a graph with $s$ pairs of consecutive vertices joined by $(b+1)$ and $(k-s)$ pairs
joined by $b$ edges. Hence
\[
s(b+1)+(k-s)b=s+kb=N.
\]
Therefore, if $0\leq s\leq k$, then $s=r$ and $b=n$. If $s=k$, $r=0$ and $n=b+1$.
The number of all paths of length $k$ is
\[
(b+1)^sb^{k-s}=(n+1)^rn^{k-r}=P^N_k.
\]
\end{proof}

\noindent
\underline{An example:}

We take a graph $E$ with $N=16$ edges, and ask about the number of all $3$-paths.

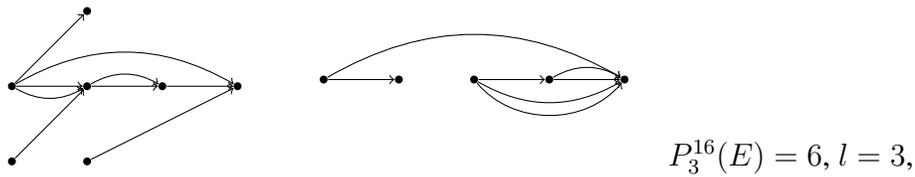
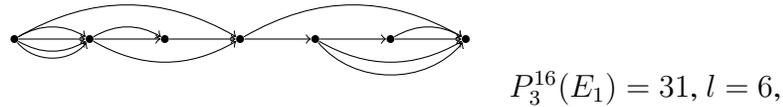
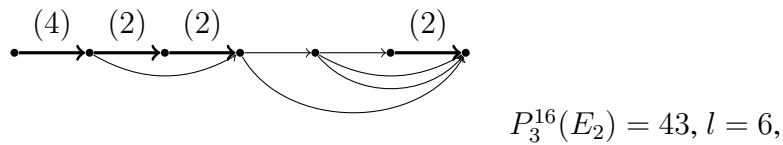
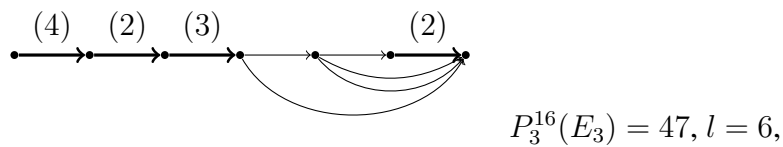
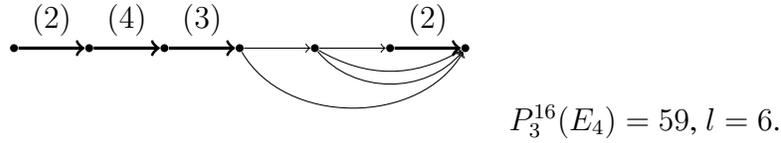
\begin{center}
\begin{tikzpicture}[auto,swap]
\tikzstyle{vertex}=[circle,fill=black,minimum size=3pt,inner sep=0pt]
\tikzstyle{edge}=[draw,very thick,->]
\tikzstyle{thinedge}=[draw,->]
\tikzset{every loop/.style={min distance=20mm,in=130,out=50,looseness=50}}
   
\node[vertex] (1) at (0,0) {};
\node[vertex] (2) at (1,0) {};
\node[vertex] (3) at (2,0) {};
\node[vertex] (4) at (3,0) {};
\node[vertex] (5) at (1,1) {};
\node[vertex] (6) at (0,-1) {};
\node[vertex] (7) at (1,-1) {};
   
   \path (1) edge[thinedge,bend left] node{} (4);
   \path (1) edge[thinedge] node{} (5);
   \path (6) edge[thinedge] node{} (2);
   \path (1) edge[thinedge] node{} (2);
   \path (1) edge[thinedge, bend right] node{} (2);
   \path (2) edge[thinedge] node{} (3);
   \path (2) edge[thinedge, bend left] node{} (3);
   \path (3) edge[thinedge] node{} (4);
   \path (7) edge[thinedge] node{} (4);
   
\end{tikzpicture}
\qquad\begin{tikzpicture}[auto,swap]
\tikzstyle{vertex}=[circle,fill=black,minimum size=3pt,inner sep=0pt]
\tikzstyle{edge}=[draw,very thick,->]
\tikzstyle{thinedge}=[draw,->]
\tikzset{every loop/.style={min distance=20mm,in=130,out=50,looseness=50}}
   
\node[vertex] (1) at (0,0) {};
\node[vertex] (2) at (1,0) {};
\node[vertex] (3) at (2,0) {};
\node[vertex] (4) at (3,0) {};
\node[vertex] (5) at (4,0) {};
\node (6) at (0,-1) {};
\node (7) at (0,1) {};
   
   \path (1) edge[thinedge,bend left] node{} (5);
   \path (1) edge[thinedge] node{} (2);
   \path (3) edge[thinedge] node{} (4);
   \path (3) edge[thinedge, bend right=50] node{} (5);
   \path (3) edge[thinedge, bend right] node{} (5);
   \path (4) edge[thinedge] node{} (5);
   \path (4) edge[thinedge,bend left] node{} (5);
   
\end{tikzpicture}
\quad
 $P^{16}_3(E)=6$, $l=3$,
\end{center}

\begin{center}
\begin{tikzpicture}[auto,swap]
\tikzstyle{vertex}=[circle,fill=black,minimum size=3pt,inner sep=0pt]
\tikzstyle{edge}=[draw,very thick,->]
\tikzstyle{thinedge}=[draw,->]
\tikzset{every loop/.style={min distance=20mm,in=130,out=50,looseness=50}}
   
\node[vertex] (1) at (0,0) {};
\node[vertex] (2) at (1,0) {};
\node[vertex] (3) at (2,0) {};
\node[vertex] (4) at (3,0) {};
\node[vertex] (5) at (4,0) {};
\node[vertex] (6) at (5,0) {};
\node[vertex] (7) at (6,0) {};
   
   \path (1) edge[thinedge,bend left] node{} (4);
   \path (1) edge[thinedge,bend left] node{} (2);
  \path (1) edge[thinedge] node{} (2);
  \path (1) edge[thinedge,bend right] node{} (2);
  \path (1) edge[thinedge,bend right=50] node{} (2);
   \path (2) edge[thinedge,bend left] node{} (3);
    \path (2) edge[thinedge] node{} (3);
    \path (2) edge[thinedge,bend right] node{} (4);
     \path (3) edge[thinedge] node{} (4);
        \path (4) edge[thinedge] node{} (5);
           \path (5) edge[thinedge] node{} (6);
              \path (6) edge[thinedge] node{} (7);
              \path (4) edge[thinedge,bend left] node{} (7);
              \path (5) edge[thinedge,bend right] node{} (7);
              \path (5) edge[thinedge,bend right=50] node{} (7);
              \path (6) edge[thinedge, bend left] node{} (7);
   
\end{tikzpicture}
\quad
 $P^{16}_3(E_1)=31$, $l=6$,
\end{center}

\begin{center}
\begin{tikzpicture}[auto,swap]
\tikzstyle{vertex}=[circle,fill=black,minimum size=3pt,inner sep=0pt]
\tikzstyle{edge}=[draw,very thick,->]
\tikzstyle{thinedge}=[draw,->]
\tikzset{every loop/.style={min distance=20mm,in=130,out=50,looseness=50}}
   
\node[vertex] (1) at (0,0) {};
\node[vertex] (2) at (1,0) {};
\node[vertex] (3) at (2,0) {};
\node[vertex] (4) at (3,0) {};
\node[vertex] (5) at (4,0) {};
\node[vertex] (6) at (5,0) {};
\node[vertex] (7) at (6,0) {};
   
   \path (1) edge[edge,above] node{$(4)$} (2);
    \path (2) edge[edge, above] node{$(2)$} (3);
    \path (2) edge[thinedge,bend right] node{} (4);
     \path (3) edge[edge, above] node{$(2)$} (4);
        \path (4) edge[thinedge] node{} (5);
           \path (5) edge[thinedge] node{} (6);
              \path (6) edge[edge,above] node{$(2)$} (7);
              \path (4) edge[thinedge,bend right=60] node{} (7);
              \path (5) edge[thinedge,bend right] node{} (7);
              \path (5) edge[thinedge,bend right=50] node{} (7);
   
\end{tikzpicture}
\quad
 $P^{16}_3(E_2)=43$, $l=6$,
\end{center}

\begin{center}
\begin{tikzpicture}[auto,swap]
\tikzstyle{vertex}=[circle,fill=black,minimum size=3pt,inner sep=0pt]
\tikzstyle{edge}=[draw,very thick,->]
\tikzstyle{thinedge}=[draw,->]
\tikzset{every loop/.style={min distance=20mm,in=130,out=50,looseness=50}}
   
\node[vertex] (1) at (0,0) {};
\node[vertex] (2) at (1,0) {};
\node[vertex] (3) at (2,0) {};
\node[vertex] (4) at (3,0) {};
\node[vertex] (5) at (4,0) {};
\node[vertex] (6) at (5,0) {};
\node[vertex] (7) at (6,0) {};
   
   \path (1) edge[edge,above] node{$(4)$} (2);
    \path (2) edge[edge, above] node{$(2)$} (3);
     \path (3) edge[edge, above] node{$(3)$} (4);
        \path (4) edge[thinedge] node{} (5);
           \path (5) edge[thinedge] node{} (6);
               \path (6) edge[edge,above] node{$(2)$} (7);
              \path (4) edge[thinedge,bend right=60] node{} (7);
              \path (5) edge[thinedge,bend right] node{} (7);
              \path (5) edge[thinedge,bend right=50] node{} (7);
   
\end{tikzpicture}
\quad
 $P^{16}_3(E_3)=47$, $l=6$,
\end{center}

\begin{center}
\begin{tikzpicture}[auto,swap]
\tikzstyle{vertex}=[circle,fill=black,minimum size=3pt,inner sep=0pt]
\tikzstyle{edge}=[draw,very thick,->]
\tikzstyle{thinedge}=[draw,->]
\tikzset{every loop/.style={min distance=20mm,in=130,out=50,looseness=50}}
   
\node[vertex] (1) at (0,0) {};
\node[vertex] (2) at (1,0) {};
\node[vertex] (3) at (2,0) {};
\node[vertex] (4) at (3,0) {};
\node[vertex] (5) at (4,0) {};
\node[vertex] (6) at (5,0) {};
\node[vertex] (7) at (6,0) {};
   
   \path (1) edge[edge,above] node{$(2)$} (2);
    \path (2) edge[edge, above] node{$(4)$} (3);
     \path (3) edge[edge, above] node{$(3)$} (4);
        \path (4) edge[thinedge] node{} (5);
           \path (5) edge[thinedge] node{} (6);
               \path (6) edge[edge,above] node{$(2)$} (7);
              \path (4) edge[thinedge,bend right=60] node{} (7);
              \path (5) edge[thinedge,bend right] node{} (7);
              \path (5) edge[thinedge,bend right=50] node{} (7);   
              
\end{tikzpicture}
\quad
 $P^{16}_3(E_4)=59$, $l=6$.
\end{center}

Now, we repeat the $E_2$-$E_3$-$E_4$ procedure to obtain $F_3$:

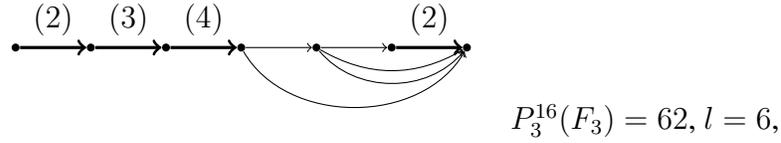
\begin{center}
\begin{tikzpicture}[auto,swap]
\tikzstyle{vertex}=[circle,fill=black,minimum size=3pt,inner sep=0pt]
\tikzstyle{edge}=[draw,very thick,->]
\tikzstyle{thinedge}=[draw,->]
\tikzset{every loop/.style={min distance=20mm,in=130,out=50,looseness=50}}
   
\node[vertex] (1) at (0,0) {};
\node[vertex] (2) at (1,0) {};
\node[vertex] (3) at (2,0) {};
\node[vertex] (4) at (3,0) {};
\node[vertex] (5) at (4,0) {};
\node[vertex] (6) at (5,0) {};
\node[vertex] (7) at (6,0) {};
   
   \path (1) edge[edge,above] node{$(2)$} (2);
    \path (2) edge[edge, above] node{$(3)$} (3);
     \path (3) edge[edge, above] node{$(4)$} (4);
        \path (4) edge[thinedge] node{} (5);
           \path (5) edge[thinedge] node{} (6);
               \path (6) edge[edge,above] node{$(2)$} (7);
              \path (4) edge[thinedge,bend right=60] node{} (7);
              \path (5) edge[thinedge,bend right] node{} (7);
              \path (5) edge[thinedge,bend right=50] node{} (7);   
              
\end{tikzpicture}
\quad
 $P^{16}_3(F_3)=62$, $l=6$,
\end{center}

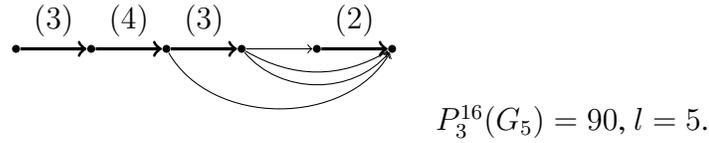
\begin{center}
\begin{tikzpicture}[auto,swap]
\tikzstyle{vertex}=[circle,fill=black,minimum size=3pt,inner sep=0pt]
\tikzstyle{edge}=[draw,very thick,->]
\tikzstyle{thinedge}=[draw,->]
\tikzset{every loop/.style={min distance=20mm,in=130,out=50,looseness=50}}
   
\node[vertex] (2) at (1,0) {};
\node[vertex] (3) at (2,0) {};
\node[vertex] (4) at (3,0) {};
\node[vertex] (5) at (4,0) {};
\node[vertex] (6) at (5,0) {};
\node[vertex] (7) at (6,0) {};
   
    \path (2) edge[edge, above] node{$(3)$} (3);
     \path (3) edge[edge, above] node{$(4)$} (4);
        \path (4) edge[edge,above] node{$(3)$} (5);
           \path (5) edge[thinedge] node{} (6);
              \path (6) edge[edge,above] node{$(2)$} (7);
              \path (4) edge[thinedge,bend right=60] node{} (7);
              \path (5) edge[thinedge,bend right] node{} (7);
              \path (5) edge[thinedge,bend right=50] node{} (7);
   
\end{tikzpicture}
\quad
 $P^{16}_3(G_5)=90$, $l=5$.
\end{center}

Applying again the $E_2$-$E_3$-$E_4$ procedure, yields:

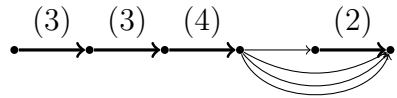
\begin{center}
\begin{tikzpicture}[auto,swap]
\tikzstyle{vertex}=[circle,fill=black,minimum size=3pt,inner sep=0pt]
\tikzstyle{edge}=[draw,very thick,->]
\tikzstyle{thinedge}=[draw,->]
\tikzset{every loop/.style={min distance=20mm,in=130,out=50,looseness=50}}
   
\node[vertex] (2) at (1,0) {};
\node[vertex] (3) at (2,0) {};
\node[vertex] (4) at (3,0) {};
\node[vertex] (5) at (4,0) {};
\node[vertex] (6) at (5,0) {};
\node[vertex] (7) at (6,0) {};
   
    \path (2) edge[edge, above] node{$(3)$} (3);
     \path (3) edge[edge, above] node{$(3)$} (4);
        \path (4) edge[edge,above] node{$(4)$} (5);
           \path (5) edge[thinedge] node{} (6);
              \path (6) edge[edge,above] node{$(2)$} (7);
              \path (5) edge[thinedge,bend right=70] node{} (7);
              \path (5) edge[thinedge,bend right] node{} (7);
              \path (5) edge[thinedge,bend right=50] node{} (7);
   
\end{tikzpicture}
\end{center}

Next, repeating the $F$-$G$-procedure, we obtain $G_4$:

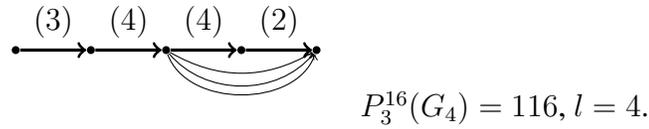
\begin{center}
\begin{tikzpicture}[auto,swap]
\tikzstyle{vertex}=[circle,fill=black,minimum size=3pt,inner sep=0pt]
\tikzstyle{edge}=[draw,very thick,->]
\tikzstyle{thinedge}=[draw,->]
\tikzset{every loop/.style={min distance=20mm,in=130,out=50,looseness=50}}
   
\node[vertex] (3) at (2,0) {};
\node[vertex] (4) at (3,0) {};
\node[vertex] (5) at (4,0) {};
\node[vertex] (6) at (5,0) {};
\node[vertex] (7) at (6,0) {};
   
     \path (3) edge[edge, above] node{$(3)$} (4);
        \path (4) edge[edge,above] node{$(4)$} (5);
           \path (5) edge[edge,above] node{$(4)$} (6);
              \path (6) edge[edge, above] node{$(2)$} (7);
              \path (5) edge[thinedge,bend right=70] node{} (7);
              \path (5) edge[thinedge,bend right] node{} (7);
              \path (5) edge[thinedge,bend right=50] node{} (7);

\end{tikzpicture}
\quad
 $P^{16}_3(G_4)=116$, $l=4$.
\end{center}

We still need to apply the $E_2$-$E_3$-$E_4$-$F$-$G$ procedure to obtain $G_3$:

\begin{center}
\begin{tikzpicture}[auto,swap]
\tikzstyle{vertex}=[circle,fill=black,minimum size=3pt,inner sep=0pt]
\tikzstyle{edge}=[draw,very thick,->]
\tikzstyle{thinedge}=[draw,->]
\tikzset{every loop/.style={min distance=20mm,in=130,out=50,looseness=50}}
   
\node[vertex] (4) at (3,0) {};
\node[vertex] (5) at (4,0) {};
\node[vertex] (6) at (5,0) {};
\node[vertex] (7) at (6,0) {};
   
        \path (4) edge[edge,above] node{$(4)$} (5);
           \path (5) edge[edge,above] node{$(4)$} (6);
              \path (6) edge[edge, above] node{$(8)$} (7);

\end{tikzpicture}
\quad
 $P^{16}_3(G_3)=128$, $l=3$.
\end{center}

The final equal-distribution procedure provides us with an optimal graph $M$
maximizing the number of $k$-paths and reaching the bound:

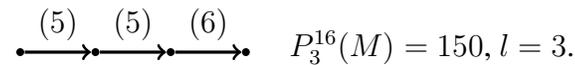
\begin{center}
\begin{tikzpicture}[auto,swap]
\tikzstyle{vertex}=[circle,fill=black,minimum size=3pt,inner sep=0pt]
\tikzstyle{edge}=[draw,very thick,->]
\tikzstyle{thinedge}=[draw,->]
\tikzset{every loop/.style={min distance=20mm,in=130,out=50,looseness=50}}
   
\node[vertex] (4) at (3,0) {};
\node[vertex] (5) at (4,0) {};
\node[vertex] (6) at (5,0) {};
\node[vertex] (7) at (6,0) {};
   
        \path (4) edge[edge,above] node{$(5)$} (5);
           \path (5) edge[edge,above] node{$(5)$} (6);
              \path (6) edge[edge, above] node{$(6)$} (7);

\end{tikzpicture}
\quad
 $P^{16}_3(M)=150$, $l=3$.
\end{center}

It agrees with the theorem: $N=nk+r$, $16=5\cdot 3+1$,
\[
P^{16}_3=(5+1)^15^{3-1}=6\cdot 5\cdot 5=150.
\]

\subsection{Adjacency matrices}

\begin{definition}
Let $E$ be a finite graph. The \underline{adjaceny matrix} $A(E)$ of the graph $E$
is the square matrix whose entries are labelled by the pairs of vertices and each $(v,w)$-entry
equals the number of edges that start at $v$ and end at $w$.
\end{definition}

\noindent\underline{Examples:}
\begin{enumerate}
\item Consider graph $E$:
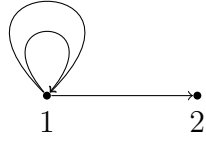
\[
\begin{tikzpicture}[auto,swap]
\tikzstyle{vertex}=[circle,fill=black,minimum size=3pt,inner sep=0pt]
\tikzstyle{edge}=[draw,->]
\tikzstyle{cycle1}=[draw,->,out=130, in=50, loop, distance=40pt]
\tikzstyle{cycle2}=[draw,->,out=135, in=45, loop, distance=65pt]
   
\node[vertex,label=below:$1$] (0) at (0,-1) {};
\node[vertex,label=below:$2$] (1) at (2,-1) {};

\path (0) edge[cycle1] node {} (0);
\path (0) edge[edge] node {} (1);
\path (0) edge[cycle2] node {} (0);

\end{tikzpicture}
\]
Then,
\[
A(E)=\begin{bmatrix}
2 & 1 \\ 0 & 0
\end{bmatrix}.
\]
\item Consider graph $E$:
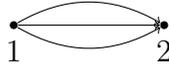
\[
\begin{tikzpicture}[auto,swap]
\tikzstyle{vertex}=[circle,fill=black,minimum size=3pt,inner sep=0pt]
\tikzstyle{edge}=[draw,->]
\tikzstyle{cycle1}=[draw,->,out=130, in=50, loop, distance=40pt]
\tikzstyle{cycle2}=[draw,->,out=135, in=45, loop, distance=65pt]
   
\node[vertex,label=below:$1$] (0) at (0,-1) {};
\node[vertex,label=below:$2$] (1) at (2,-1) {};

\path (0) edge[edge] node {} (1);
\path (0) edge[edge, bend right] node {} (1);
\path (0) edge[edge, bend left] node {} (1);

\end{tikzpicture}
\]
Then,
\[
A(E)=\begin{bmatrix}
0 & 3 \\ 0 & 0
\end{bmatrix}.
\]
\item Consider graph $E$:
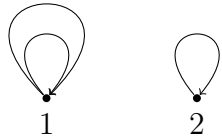
\[
\begin{tikzpicture}[auto,swap]
\tikzstyle{vertex}=[circle,fill=black,minimum size=3pt,inner sep=0pt]
\tikzstyle{edge}=[draw,->]
\tikzstyle{cycle1}=[draw,->,out=130, in=50, loop, distance=40pt]
\tikzstyle{cycle2}=[draw,->,out=135, in=45, loop, distance=65pt]
   
\node[vertex,label=below:$1$] (0) at (0,-1) {};
\node[vertex,label=below:$2$] (1) at (2,-1) {};

\path (0) edge[cycle1] node {} (0);
\path (0) edge[cycle2] node {} (0);
\path (1) edge[cycle1] node {} (1);

\end{tikzpicture}
\]
Then,
\[
A(E)=\begin{bmatrix}
2 & 0 \\ 0 & 1
\end{bmatrix}.
\]
\item Consider graph $E$:
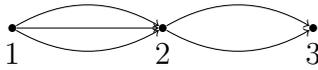
\[
\begin{tikzpicture}[auto,swap]
\tikzstyle{vertex}=[circle,fill=black,minimum size=3pt,inner sep=0pt]
\tikzstyle{edge}=[draw,->]
\tikzstyle{cycle1}=[draw,->,out=130, in=50, loop, distance=40pt]
\tikzstyle{cycle2}=[draw,->,out=135, in=45, loop, distance=65pt]
   
\node[vertex,label=below:$1$] (0) at (0,-1) {};
\node[vertex,label=below:$2$] (1) at (2,-1) {};
\node[vertex,label=below:$3$] (2) at (4,-1) {};

\path (0) edge[edge] node {} (1);
\path (0) edge[edge, bend right] node {} (1);
\path (0) edge[edge, bend left] node {} (1);
\path (1) edge[edge, bend right] node {} (2);
\path (1) edge[edge, bend left] node {} (2);

\end{tikzpicture}
\]
Then,
\[
A(E)=\begin{bmatrix}
0 & 3 & 0 \\ 0 & 0 & 2\\ 0 & 0 & 0
\end{bmatrix}.
\]
\end{enumerate}

In a finite graph $E$ with $N$ edges, consider all paths of length $k$ starting at a vertex $v$ and ending at a vertex $w$.
If $k_1+k_2=k$ and $k_1,k_2\in\mathbb{N}\setminus\{0\}$, then each path $p$ with $s(p)=v$ and $t(p)=w$ decomposes
into a path $p_1$ of length $k_1$ with $s(p_1)=v$, $t(p_1)=u$, and a path $p_2$ of length $k_2$ with $s(p_2)=u$, $t(p_2)=w$.
Hence the number $N_k(v,w)$ of all $k$-paths from $v$ to $w$ equals
\[
N_k(v,w)=\sum_{u\in E^0}N_{k_1}(v,u)N_{k_2}(u,w).
\]
Thus we have shown:
\begin{proposition}\label{adjmat}
Let $E$ be a finite graph, and let $A_l(E)$ be a generalized adjacency matrix whose entries count the number
of all $l$-paths between vertices. Then, $\forall\; k_1,k_2\in\mathbb{N}\setminus\{0\}$:
\[
A_{k_1+k_2}(E)=A_{k_1}(E)A_{k_2}(E).
\]
\end{proposition}
\begin{corollary}
$\forall\;n\in\mathbb{N}\setminus\{0\}:A_n(E)=A(E)^n$.
\end{corollary}
\begin{proof}
The statement holds for $n=1$, and taking $k_1=n$ and $k_2=1$ in Proposition~\ref{adjmat}
proves the inductive step.
\end{proof}
\begin{corollary}
A finite graph $E$ has no loops if and only if its adjacency matrix is nilpotent.
\end{corollary}
\begin{proof}
The finite graph $E$ has no loops $\iff$ $FP(E)$ is finite. The latter is equivalent to the existence of a longest path.
Indeed, if there is no longest path, then there are paths of all lengths, so $FP(E)$ is infinite. Vice versa, if $FP(E)$ is infinite,
then there is a loop, so there is no longest path.

Next, if the length of a longest path is $l$, then $A(E)^{l+1}=0$, so $A(E)$ is nilpotent.
Vice versa, if $A(E)$ is nilpotent, then there exists $n$ such that $A(E)^{n}=0$.
Hence there are no paths of length $\geq n$, so there exists a longest path.
\end{proof}
\begin{corollary}
The number of all $k$-paths is given by $FP_k(E)=\sum_{v,w\in E^0}\left(A(E)^k\right)_{vw}$.
\end{corollary}

\noindent\underline{Examples:}
\begin{enumerate}
\item Consider graph $E$:
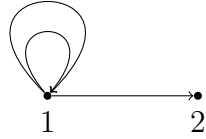
\[
\begin{tikzpicture}[auto,swap]
\tikzstyle{vertex}=[circle,fill=black,minimum size=3pt,inner sep=0pt]
\tikzstyle{edge}=[draw,->]
\tikzstyle{cycle1}=[draw,->,out=130, in=50, loop, distance=40pt]
\tikzstyle{cycle2}=[draw,->,out=135, in=45, loop, distance=65pt]
   
\node[vertex,label=below:$1$] (0) at (0,-1) {};
\node[vertex,label=below:$2$] (1) at (2,-1) {};

\path (0) edge[cycle1] node {} (0);
\path (0) edge[edge] node {} (1);
\path (0) edge[cycle2] node {} (0);

\end{tikzpicture}
\]
Then,
\[
A(E)^k=\begin{bmatrix}
2^k & 2^{k-1} \\ 0 & 0
\end{bmatrix},
\]
so there are $2^k+2^{k-1}=3\cdot 2^{k-1}$ many $k$-paths in $E$.
\item Consider graph $E$:
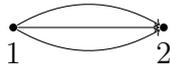
\[
\begin{tikzpicture}[auto,swap]
\tikzstyle{vertex}=[circle,fill=black,minimum size=3pt,inner sep=0pt]
\tikzstyle{edge}=[draw,->]
\tikzstyle{cycle1}=[draw,->,out=130, in=50, loop, distance=40pt]
\tikzstyle{cycle2}=[draw,->,out=135, in=45, loop, distance=65pt]
   
\node[vertex,label=below:$1$] (0) at (0,-1) {};
\node[vertex,label=below:$2$] (1) at (2,-1) {};

\path (0) edge[edge] node {} (1);
\path (0) edge[edge, bend right] node {} (1);
\path (0) edge[edge, bend left] node {} (1);

\end{tikzpicture}
\]
Then,
\[
A(E)^2=\begin{bmatrix}
0 & 3 \\ 0 & 0
\end{bmatrix}^2=\begin{bmatrix}
0 & 0 \\ 0 & 0
\end{bmatrix},
\]
so there are no paths longer than $1$.
\item Consider graph $E$:
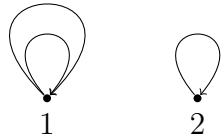
\[
\begin{tikzpicture}[auto,swap]
\tikzstyle{vertex}=[circle,fill=black,minimum size=3pt,inner sep=0pt]
\tikzstyle{edge}=[draw,->]
\tikzstyle{cycle1}=[draw,->,out=130, in=50, loop, distance=40pt]
\tikzstyle{cycle2}=[draw,->,out=135, in=45, loop, distance=65pt]
   
\node[vertex,label=below:$1$] (0) at (0,-1) {};
\node[vertex,label=below:$2$] (1) at (2,-1) {};

\path (0) edge[cycle1] node {} (0);
\path (0) edge[cycle2] node {} (0);
\path (1) edge[cycle1] node {} (1);

\end{tikzpicture}
\]
Then,
\[
A(E)^k=\begin{bmatrix}
2^k & 0 \\ 0 & 1
\end{bmatrix},
\]
so there are $2^k+1$ many $k$-paths in $E$.
\item Consider graph $E$:
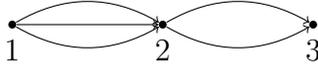
\[
\begin{tikzpicture}[auto,swap]
\tikzstyle{vertex}=[circle,fill=black,minimum size=3pt,inner sep=0pt]
\tikzstyle{edge}=[draw,->]
\tikzstyle{cycle1}=[draw,->,out=130, in=50, loop, distance=40pt]
\tikzstyle{cycle2}=[draw,->,out=135, in=45, loop, distance=65pt]
   
\node[vertex,label=below:$1$] (0) at (0,-1) {};
\node[vertex,label=below:$2$] (1) at (2,-1) {};
\node[vertex,label=below:$3$] (2) at (4,-1) {};

\path (0) edge[edge] node {} (1);
\path (0) edge[edge, bend right] node {} (1);
\path (0) edge[edge, bend left] node {} (1);
\path (1) edge[edge, bend right] node {} (2);
\path (1) edge[edge, bend left] node {} (2);

\end{tikzpicture}
\]
Then,
\[
A(E)^2=\begin{bmatrix}
0 & 3 & 0 \\ 0 & 0 & 2\\ 0 & 0 & 0
\end{bmatrix}^2=\begin{bmatrix}
0 & 0 & 6 \\ 0 & 0 & 0\\ 0 & 0 & 0
\end{bmatrix},
\]
so there are $6$ paths of length $2$.
\end{enumerate}

Note that there finitely many graphs with $N$ edges and whose all vertices
emit or receive at least one edge. Indeed, for any such graph $E$,
$E^0=\{1,2,\ldots,m\}$, $m\leq 2N$, $E^1=\{1,2,\ldots,N\}$,
$s\subseteq E^1\times E^0$, $t\subseteq E^1\times E^0$,
so the number of all such graphs is limited by $2N\cdot 2^{Nm}\cdot 2^{Nm}=N\cdot 2^{2Nm+1}$.

\begin{corollary}
Let $\mathcal{E}_N$ denote the set of all graphs with $N$ edges, no loops, and whose
all vertices emit or receive at least one edge. Then, $\forall\; k\in\{1,\ldots, N\}$
\[
\max_{E\in\mathcal{E}_n}\left\{\sum_{v,w\in E^0}\left(A(E)^k\right)_{vw}\right\}=(n+1)^rn^{k-r},
\]
where $N=:nk+r$ with $r\in\{0,1,\ldots,k-1\}$.
\end{corollary}

Next, observe that, if $A(E)\in M_n(\mathbb{N})$,
then it is the adjacency matrix of the graph $E(A)$ with $E(A)^0=\{1,\ldots,n\}$
and $E(A)^1=\bigcup_{i,j\in E^0}E_{ij}$, where $E_{ij}$ is the set of $A_{ij}$-many
edges from $i$ to $j$. For instance, for $n=3$, we have
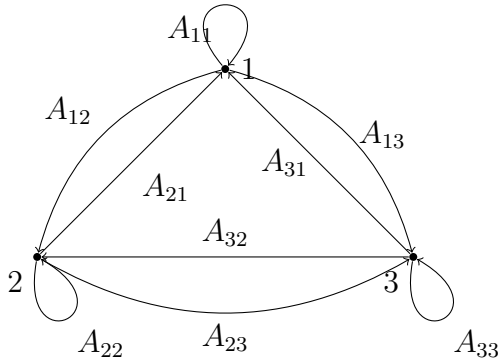
\[
\begin{tikzpicture}[auto,swap]
\tikzstyle{vertex}=[circle,fill=black,minimum size=3pt,inner sep=0pt]
\tikzstyle{edge}=[draw,->]
\tikzstyle{cycle1}=[draw,->,out=130, in=50, loop, distance=40pt]
\tikzstyle{cycle2}=[draw,->,out=-100, in=-30, loop, distance=40pt]
   
\node[vertex,label=right:$1$] (0) at (0,0) {};
\node[vertex,label=below left:$3$] (1) at (2.5,-2.5) {};
\node[vertex,label=below left:$2$] (2) at (-2.5,-2.5) {};

\path (0) edge[cycle1] node[below left] {$A_{11}$} (0);
\path (0) edge[edge, bend left] node[right] {$A_{13}$} (1);
\path (1) edge[edge] node[left] {$A_{31}$} (0);
\path (0) edge[edge,bend right] node {$A_{12}$} (2);
\path (1) edge[cycle2] node {$A_{33}$} (1);
\path (1) edge[edge] node {$A_{32}$} (2);
\path (2) edge[edge] node {$A_{21}$} (0);
\path (2) edge[edge, bend right] node {$A_{23}$} (1);
\path (2) edge[cycle2] node {$A_{22}$} (2);

\end{tikzpicture}
\]
A vertex $i$ of the graph $E(A)$ emits or receives at least one edge if and only if
\[
\sum_{k=1}^nA_{ik}+A_{ki}>0.
\]
Note also that $A(E(M))=M$ and $E(A(G))=G$. Therefore, as no loops in $E$ means that $A(E)$ is nilpotent,
we can reformulate the foregoing corollary as follows:
\begin{corollary}
Let 
\[
\mathcal{A}_N:=\left\{A\in \bigcup_{k\in\mathbb{N}\setminus\{0\}}M_k(\mathbb{N})~|~\text{satisfying } 1,2,3\text{ below}\right\}.
\]
\begin{enumerate}
\item $A$ is nilpotent ($E$ has no loops),
\item $\sum\limits_{{\rm all}~i, j}A_{ij}=N$ ($E$ has $N$ edges),
\item $\forall\;i:\sum\limits_{{\rm all}~j}A_{ij}+A_{ji}>0$ (each vertex of $E$ emits or receives).
\end{enumerate}
Then, $\forall\;k\in\{1,\ldots,N\}$:
\[
\max_{A\in\mathcal{A}_n}\left\{\sum_{{\rm all}~i, j}\left(A^k\right)_{ij}\right\}=(n+1)^rn^{k-r},
\]
where $N=:nk+r$ with $r\in\{0,1,\ldots,k-1\}$.
\end{corollary}
\begin{conjecture}
Let $N$ be a non-negative real number, and let
\[
\mathcal{A}_N:=\left\{A\in\bigcup_{k\in\mathbb{N}\setminus\{0\}}M_k(\mathbb{R}_{\geq 0})~|~\text{satisfying } 1,2,3\text{ above}\right\}.
\]
Then, $\forall\;k\in\{1,\ldots,[N]\}$:
\[
\sup_{A\in\mathcal{A}_N}\left\{\sum_{{\rm all}~i,j}\left(A^k\right)_{ij}\right\}=\left(\frac{N}{k}\right)^k.
\]
Here $[N]$ stands for the integer part of $N$.
\end{conjecture}

\subsection{The structure of graphs}
\begin{definition}
Let $E$ be a graph. An \underline{undirected finite path} in $E$ is a finite sequence
of edges $(e_1,\ldots,e_n)$ satisfying at least one of the $4$ equalities:
\begin{enumerate}
\item $s(e_i)=s(e_{i+1})$,
\qquad\begin{tikzpicture}[auto,swap]
\tikzstyle{vertex}=[circle,fill=black,minimum size=3pt,inner sep=0pt]
\tikzstyle{edge}=[draw,->]
\tikzstyle{cycle1}=[draw,->,out=130, in=50, loop, distance=40pt]
\tikzstyle{cycle2}=[draw,->,out=-100, in=-30, loop, distance=40pt]
   
\node[vertex] (0) at (0,0) {};
\node[vertex] (1) at (1,0) {};
\node[vertex] (2) at (2,0) {};

\path (1) edge[edge] node {$e_i$} (0);
\path (1) edge[edge] node[above] {$e_{i+1}$} (2);

\end{tikzpicture}
\item $s(e_i)=t(e_{i+1})$,
\qquad\begin{tikzpicture}[auto,swap]
\tikzstyle{vertex}=[circle,fill=black,minimum size=3pt,inner sep=0pt]
\tikzstyle{edge}=[draw,->]
\tikzstyle{cycle1}=[draw,->,out=130, in=50, loop, distance=40pt]
\tikzstyle{cycle2}=[draw,->,out=-100, in=-30, loop, distance=40pt]
   
\node[vertex] (0) at (0,0) {};
\node[vertex] (1) at (1,0) {};
\node[vertex] (2) at (2,0) {};

\path (1) edge[edge] node {$e_i$} (0);
\path (2) edge[edge] node[above] {$e_{i+1}$} (1);

\end{tikzpicture}
\item $t(e_i)=s(e_{i+1})$,
\qquad\begin{tikzpicture}[auto,swap]
\tikzstyle{vertex}=[circle,fill=black,minimum size=3pt,inner sep=0pt]
\tikzstyle{edge}=[draw,->]
\tikzstyle{cycle1}=[draw,->,out=130, in=50, loop, distance=40pt]
\tikzstyle{cycle2}=[draw,->,out=-100, in=-30, loop, distance=40pt]
   
\node[vertex] (0) at (0,0) {};
\node[vertex] (1) at (1,0) {};
\node[vertex] (2) at (2,0) {};

\path (0) edge[edge] node[above] {$e_i$} (1);
\path (1) edge[edge] node[above] {$e_{i+1}$} (2);

\end{tikzpicture}
\item $t(e_i)=t(e_{i+1})$.
\qquad\begin{tikzpicture}[auto,swap]
\tikzstyle{vertex}=[circle,fill=black,minimum size=3pt,inner sep=0pt]
\tikzstyle{edge}=[draw,->]
\tikzstyle{cycle1}=[draw,->,out=130, in=50, loop, distance=40pt]
\tikzstyle{cycle2}=[draw,->,out=-100, in=-30, loop, distance=40pt]
   
\node[vertex] (0) at (0,0) {};
\node[vertex] (1) at (1,0) {};
\node[vertex] (2) at (2,0) {};

\path (0) edge[edge] node[above] {$e_i$} (1);
\path (2) edge[edge] node[above] {$e_{i+1}$} (1);

\end{tikzpicture}
\end{enumerate}
for all $i\in \{1,\ldots,n-1\}$.

An \underline{undirected infinite path} in $E$ is an infinite sequence
$(e_1,\ldots,e_n,\ldots)$ satisfying at least one of the above $4$ equalities
for all $i\in \mathbb{N}\setminus\{0\}$.
\end{definition}
\begin{definition}
We say that a finite graph is connected $E$ is \underline{connected}
iff for any pair of vertices $(v,w)$, $v\neq w$, there exists an \underline{undirected} finite path
$p=(e_1,\ldots,e_n)$ between $v$ and $w$: ($s(e_1)=v$ or $t(e_1)=v$) and ($t(e_n)=w$ or $s(e_n)=w$).
\end{definition}
\begin{definition}
A vertex $v$ in a graph $E$ is called a \underline{sink} iff $s^{-1}(v)=\emptyset$.
\end{definition}
\begin{proposition}
If $E$ is a graph with finitely many edges, no loops, and exactly one sink, then $E$ is connected.
\end{proposition}
\begin{proof}
Denote the sink by $v_0$. If it is the only vertex of $E$, then $E$ is connected.
If there is $v_1\neq v_0$, then there exists a path from $v_1$ to $v_0$. Indeed, as $v_0$ is the unique sink, $v_1$ emits an edge $e_1$.
Consider any path $p_n:=(e_1,\ldots,e_n)$, e.g. $p_1=e_1$. If $s(p_n)\neq v_0$, then, sd $v_0$ is the unique sink, $s(p_n)$ emits $e_{n+1}$
yielding the path $p_{n+1}:=(e_1,\ldots,e_n,e_{n+1})$ which is longer than $p_n$. Hence, if no path starting at $v_1$ terminates at $v_0$,
we can have paths of arbitrary lengths, which is impossible because $E$ has finitely many edges and no loops.
Therefore, there is a path from $v_1$ to $v_0$. Now, take any pair of distinct vertices in $E$. If one of them is $v_0$,
then they are connected by a path. If $w_1\neq v_0$, $w_2\neq v_0$, $w_1\neq w_2$, then there is a path $q_1:=(f_1,\ldots,f_k)$
from $w_1$ to $v_0$ and a path $q_2:=(g_1,\ldots, g_l)$ from $w_2$ to $v_0$.
They combine into the undirected path $(f_1,\ldots,f_k,g_l,\ldots,g_1)$ from $w_1$ to $w_2$, so $E$ is connected.
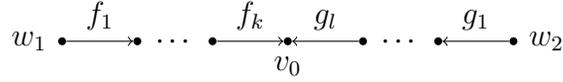
\[
\begin{tikzpicture}[auto,swap]
\tikzstyle{vertex}=[circle,fill=black,minimum size=3pt,inner sep=0pt]
\tikzstyle{edge}=[draw,->]
\tikzstyle{cycle1}=[draw,->,out=130, in=50, loop, distance=40pt]
\tikzstyle{cycle2}=[draw,->,out=-100, in=-30, loop, distance=40pt]
   
\node[vertex,label=left:$w_1$] (0) at (0,0) {};
\node[vertex] (1) at (1,0) {};
\node (2) at (1.5,0) {\ldots};
\node[vertex] (3) at (2,0) {};
\node[vertex,label=below:$v_0$] (4) at (3,0) {};
\node[vertex] (5) at (4,0) {};
\node (6) at (4.5,0) {\ldots};
\node[vertex] (7) at (5,0) {};
\node[vertex,label=right:$w_2$] (8) at (6,0) {};

\path (0) edge[edge] node[above] {$f_1$} (1);
\path (3) edge[edge] node[above] {$f_k$} (4);
\path (5) edge[edge] node[above] {$g_l$} (4);
\path (8) edge[edge] node[above] {$g_1$} (7);

\end{tikzpicture}
\]
\end{proof}

\noindent\underline{Question:} What is the maximal number of all finite paths in a graph with $N$ edges, no loops, and exactly one sink?

\begin{definition}
Let $E$ be a graph. A subset $H\subseteq E^0$ is called \underline{hereditary}
iff any finite edge starting at $v\in H$ ends at $w\in H$.
\end{definition}
\noindent
Note that in the above definition one can replace the word ``edge'' by the word ``path''. Indeed, if $H$ is path-hereditary, then it is, in particular edge-hereditary.
Also, if $H$ is not path hereditary, then there exists a path starting in $H$ and ending outside of $H$. Such a path must contain an edge starting at $H$ and ending outside of $H$,
so it is also not edge-hereditary. This proves the equivalence of these two definitions.

\noindent\underline{Examples:}
\begin{enumerate}
\item In any graph $E$, $\emptyset$ and $E^0$ are hereditary.
\item Consider a graph $E$:
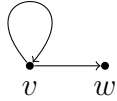
\[
\begin{tikzpicture}[auto,swap]
\tikzstyle{vertex}=[circle,fill=black,minimum size=3pt,inner sep=0pt]
\tikzstyle{edge}=[draw,->]
\tikzstyle{cycle1}=[draw,->,out=130, in=50, loop, distance=40pt]
\tikzstyle{cycle2}=[draw,->,out=100, in=30, loop, distance=40pt]
   
\node[vertex,label=below:$v$] (0) at (0,0) {};
\node[vertex,label=below:$w$] (1) at (1,0) {};

\path (0) edge[cycle1] node {} (0);
\path (0) edge[edge] node {} (1);

\end{tikzpicture}
\]
Then, $\{w\}$ is hereditary and $\{v\}$ is not hereditary.
\end{enumerate}
\begin{definition}
Let $E=(E^0,E^1,s,t)$ be a graph.
$F=(F^0,F^1,s_F,t_F)$ is a \underline{subgraph} of $E$ iff
$F^0\subseteq E^0$, $F^1\subseteq E^1$, and
\[
\forall\;e\in F^1:\quad s_F(e)=s(e)\in F^0,\quad t_F(e)=t(e)\in F^0.
\]
\end{definition}
\begin{proposition}
Let $E=(E^0,E^1,s,t)$ be a graph, and let $H\subseteq E^0$.
Set $F^0:=E^0\setminus H$ and $F^1:=E^1\setminus t^{-1}(H)$.
Then the formulas $s_F(e)=s(e)$, $t_F(e)=t(e)$, for any $e\in F^1$,
define a subgraph of $E$ if and only if $H$ is hereditary.
\end{proposition}
\begin{proof}
Note first that, if $e\in F^1:=E^1\setminus t^{-1}(H)$,
then $t_F(e)=t(e)\in E^0\setminus H=:F^0$.
Hence, $t:E^1\to E^0$ always restricts-corestricts to $t_F:F^1\to F^0$.

Assume now that $H$ is hereditary. Then, if $e\in E^1$ and $s(e)\in H$,
we have $t(e)\in H$. Hence 
$$
e\in F^1:=E^1\setminus t^{-1}(H)\iff
t(e)\not\in H\Rightarrow s(e)\not\in H\iff s(e)\in F^0:=E^0\setminus H.
$$
Therefore, $s:E^1\to E^0$ restricts-corestricts to $s_F:F^1\to F^0$.
Vice versa, assume that $s$ restricts-corestricts to~$s_F$.
Then 
$$
t(e)\not\in H\iff e\in F^1\Rightarrow s(e)\in F^0\iff s(e)\not\in H,
$$
so $s(e)\in H$ $\Rightarrow$ $t(e)\in H$, i.e.\ $H$ is hereditary. 
\end{proof}

\begin{definition}
Let $E$ be a graph. A subset $H\subseteq E^0$ is called
\underline{saturated} iff
\[
\not\!\exists\;v\in E^0\setminus H\colon\; 0<|s^{-1}(v)|<\infty \text{ and } t(s^{-1}(v))\subseteq H.
\]
\end{definition}
\noindent\underline{Examples:}
\begin{enumerate}
\item In any graph $E$, $\emptyset$ and $E^0$ are saturated.
\item Consider a graph $E$:
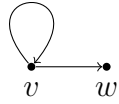
\[
\begin{tikzpicture}[auto,swap]
\tikzstyle{vertex}=[circle,fill=black,minimum size=3pt,inner sep=0pt]
\tikzstyle{edge}=[draw,->]
\tikzstyle{cycle1}=[draw,->,out=130, in=50, loop, distance=40pt]
\tikzstyle{cycle2}=[draw,->,out=100, in=30, loop, distance=40pt]
   
\node[vertex,label=below:$v$] (0) at (0,0) {};
\node[vertex,label=below:$w$] (1) at (1,0) {};

\path (0) edge[cycle1] node {} (0);
\path (0) edge[edge] node {} (1);

\end{tikzpicture}
\]
Then all subsetes of $E^0$ are saturated.
\item Consider a graph $E$:
\[
\begin{tikzpicture}[auto,swap]
\tikzstyle{vertex}=[circle,fill=black,minimum size=3pt,inner sep=0pt]
\tikzstyle{edge}=[draw,->]
\tikzstyle{cycle1}=[draw,->,out=130, in=50, loop, distance=40pt]
\tikzstyle{cycle2}=[draw,->,out=100, in=30, loop, distance=40pt]
   
\node[vertex,label=below:$v$] (0) at (0,0) {};
\node[vertex,label=below:$w$] (1) at (1,0) {};

\path (0) edge[edge] node {} (1);

\end{tikzpicture}
\]
Then $\{w\}$ is not saturated but it is hereditary.
\item Consider a graph $E$:
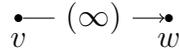
\[
\begin{tikzpicture}[auto,swap]
\tikzstyle{vertex}=[circle,fill=black,minimum size=3pt,inner sep=0pt]
\tikzstyle{edge}=[draw,->]
\tikzstyle{cycle1}=[draw,->,out=130, in=50, loop, distance=40pt]
\tikzstyle{cycle2}=[draw,->,out=100, in=30, loop, distance=40pt]
   
\node[vertex,label=below:$v$] (0) at (0,0) {};
\node (2) at (1,0) {$(\infty)$};
\node[vertex,label=below:$w$] (1) at (2,0) {};

\path (0) edge[draw] node {} (2);
\path (2) edge[edge] node {} (1);

\end{tikzpicture}
\]
Then $\{w\}$ is both saturated and hereditary.
\end{enumerate}

\subsection{Homomorphisms of graphs}
\begin{definition}
A \underline{homomorphism} from a graph $E:=(E^0,E^1,s_E,t_E)$ to a graph $F:=(F^0,F^1,s_F,t_F)$
is a pair of maps \[(f^0:E^0\to F^0,f^1:E^1\to F^1)\] satisfying the conditions:
\[ s_F\circ f^1=f^0\circ s_E\,,\qquad t_F\circ f^1=f^0\circ t_E\,.\]
\end{definition}

\noindent\underline{Examples:}
\begin{enumerate}
\item Inclusions of subgraphs
$E^0\overset{f^0}{\hookrightarrow}F^0$, $E^1\overset{f^1}{\hookrightarrow}F^1$, e.g.:
\[
\begin{tikzpicture}[auto,swap]
\tikzstyle{vertex}=[circle,fill=black,minimum size=3pt,inner sep=0pt]
\tikzstyle{edge}=[draw,->]
\tikzstyle{cycle1}=[draw,->,out=130, in=50, loop, distance=40pt]
\tikzstyle{cycle2}=[draw,->,out=100, in=30, loop, distance=40pt]
   
\node[vertex,label=below:$v$] (0) at (0,0) {};

\path (0) edge[cycle1] node[above] {$e$} (0);

\end{tikzpicture}\quad\overset{(f^0,f^1)}{\longto}\quad
\begin{tikzpicture}[auto,swap]
\tikzstyle{vertex}=[circle,fill=black,minimum size=3pt,inner sep=0pt]
\tikzstyle{edge}=[draw,->]
\tikzstyle{cycle1}=[draw,->,out=130, in=50, loop, distance=40pt]
\tikzstyle{cycle2}=[draw,->,out=100, in=30, loop, distance=40pt]
   
\node[vertex,label=below:$v$] (0) at (0,0) {};
\node[vertex,label=below:$w$] (1) at (1,0) {};

\path (0) edge[cycle1] node[above] {$e$} (0);
\path (0) edge[edge] node[above] {$g$} (1);

\end{tikzpicture}
\]
\[
f^0(v)=v,\qquad f^1(e)=e,\qquad\text{or}
\]
\[
\begin{tikzpicture}[auto,swap]
\tikzstyle{vertex}=[circle,fill=black,minimum size=3pt,inner sep=0pt]
\tikzstyle{edge}=[draw,->]
\tikzstyle{cycle1}=[draw,->,out=130, in=50, loop, distance=40pt]
\tikzstyle{cycle2}=[draw,->,out=100, in=30, loop, distance=40pt]
   
\node[vertex,label=below:$v$] (0) at (0,0) {};
\node[vertex,label=below:$w$] (1) at (1,0) {};

\path (0) edge[edge] node[above] {$e$} (1);

\end{tikzpicture}\quad\longto\quad
\begin{tikzpicture}[auto,swap]
\tikzstyle{vertex}=[circle,fill=black,minimum size=3pt,inner sep=0pt]
\tikzstyle{edge}=[draw,->]
\tikzstyle{cycle1}=[draw,->,out=130, in=50, loop, distance=40pt]
\tikzstyle{cycle2}=[draw,->,out=100, in=30, loop, distance=40pt]
   
\node[vertex,label=below:$v$] (0) at (0,0) {};
\node[vertex,label=below:$w$] (1) at (2,0) {};

\path (0) edge[edge] node[above] {$n$-edges} (1);

\end{tikzpicture}
\]
\[
f^0(v)=v,\qquad f^0(w)=w,\qquad f^1(e)=e_1\,.
\]
\item Collapsing edges between the same vertices to one edge, e.g.
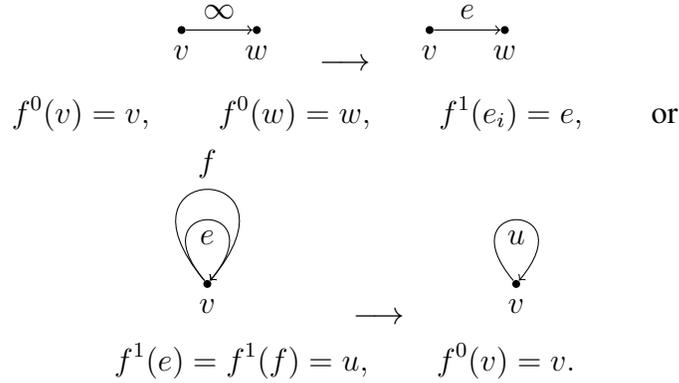
\[
\begin{tikzpicture}[auto,swap]
\tikzstyle{vertex}=[circle,fill=black,minimum size=3pt,inner sep=0pt]
\tikzstyle{edge}=[draw,->]
\tikzstyle{cycle1}=[draw,->,out=130, in=50, loop, distance=40pt]
\tikzstyle{cycle2}=[draw,->,out=100, in=30, loop, distance=40pt]
   
\node[vertex,label=below:$v$] (0) at (0,0) {};
\node[vertex,label=below:$w$] (1) at (1,0) {};

\path (0) edge[edge] node[above] {$\infty$} (1);

\end{tikzpicture}\quad\longto\quad
\begin{tikzpicture}[auto,swap]
\tikzstyle{vertex}=[circle,fill=black,minimum size=3pt,inner sep=0pt]
\tikzstyle{edge}=[draw,->]
\tikzstyle{cycle1}=[draw,->,out=130, in=50, loop, distance=40pt]
\tikzstyle{cycle2}=[draw,->,out=100, in=30, loop, distance=40pt]
   
\node[vertex,label=below:$v$] (0) at (0,0) {};
\node[vertex,label=below:$w$] (1) at (1,0) {};

\path (0) edge[edge] node[above] {$e$} (1);

\end{tikzpicture}
\]
\[
f^0(v)=v,\qquad f^0(w)=w,\qquad f^1(e_i)=e,\qquad\text{or}
\]
\[
\begin{tikzpicture}[auto,swap]
\tikzstyle{vertex}=[circle,fill=black,minimum size=3pt,inner sep=0pt]
\tikzstyle{edge}=[draw,->]
\tikzstyle{cycle1}=[draw,->,out=130, in=50, loop, distance=40pt]
\tikzstyle{cycle2}=[draw,->,out=130, in=50, loop, distance=60pt]
   
\node[vertex,label=below:$v$] (0) at (0,0) {};

\path (0) edge[cycle1] node {$e$} (0);
\path (0) edge[cycle2] node[above] {$f$} (0);

\end{tikzpicture}\quad\longto\quad
\begin{tikzpicture}[auto,swap]
\tikzstyle{vertex}=[circle,fill=black,minimum size=3pt,inner sep=0pt]
\tikzstyle{edge}=[draw,->]
\tikzstyle{cycle1}=[draw,->,out=130, in=50, loop, distance=40pt]
\tikzstyle{cycle2}=[draw,->,out=130, in=50, loop, distance=60pt]
   
\node[vertex,label=below:$v$] (0) at (0,0) {};

\path (0) edge[cycle1] node {$u$} (0);

\end{tikzpicture}
\]
\[
f^1(e)=f^1(f)=u,\qquad f^0(v)=v.
\]
\item A combination of both, e.g.
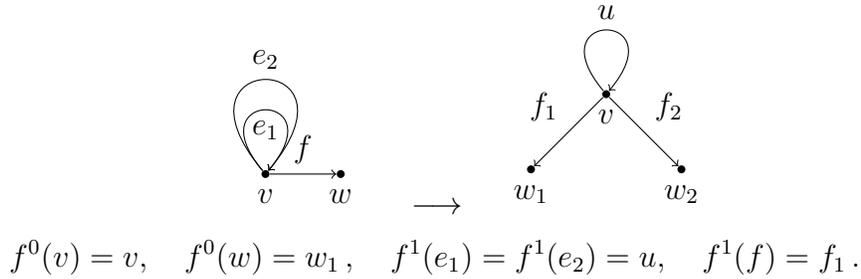
\[
\begin{tikzpicture}[auto,swap]
\tikzstyle{vertex}=[circle,fill=black,minimum size=3pt,inner sep=0pt]
\tikzstyle{edge}=[draw,->]
\tikzstyle{cycle1}=[draw,->,out=130, in=50, loop, distance=40pt]
\tikzstyle{cycle2}=[draw,->,out=130, in=50, loop, distance=60pt]
   
\node[vertex,label=below:$v$] (0) at (0,0) {};
\node[vertex,label=below:$w$] (1) at (1,0) {};

\path (0) edge[cycle1] node {$e_1$} (0);
\path (0) edge[cycle2] node[above] {$e_2$} (0);
\path (0) edge[edge] node[above] {$f$} (1);

\end{tikzpicture}\quad\longto\quad
\begin{tikzpicture}[auto,swap]
\tikzstyle{vertex}=[circle,fill=black,minimum size=3pt,inner sep=0pt]
\tikzstyle{edge}=[draw,->]
\tikzstyle{cycle1}=[draw,->,out=130, in=50, loop, distance=40pt]
\tikzstyle{cycle2}=[draw,->,out=100, in=30, loop, distance=40pt]
   
\node[vertex,label=below:$v$] (0) at (0,0) {};
\node[vertex,label=below:$w_1$] (1) at (-1,-1) {};
\node[vertex,label=below:$w_2$] (2) at (1,-1) {};

\path (0) edge[cycle1] node[above] {$u$} (0);
\path (0) edge[edge] node[above left] {$f_1$} (1);
\path (0) edge[edge] node[above right] {$f_2$} (2);

\end{tikzpicture}
\]
\[
f^0(v)=v,\quad f^0(w)=w_1\,,\quad f^1(e_1)=f^1(e_2)=u,\quad f^1(f)=f_1\,.
\]
\end{enumerate}

From the graph-algebra point of view, of particular interest are injective graph homomorphisms
$(f^0,f^1):E\to F$ (both $f^0$ and $f^1$ injective) satisfying certain conditions.

\begin{definition}
We call an injective homomorphism of graphs $(f^0,f^1):E\to F$
an \underline{admissible} \underline{inclusion} iff it satisfies the following conditions:
\begin{enumerate}
\item $F^0\setminus f^0(E^0)$ is hereditary and saturated,
\item $f^1(E^1)= t^{-1}_F( f^0(E^0))$.
\end{enumerate}
\end{definition}

\noindent \underline{Examples:}
\begin{enumerate}
\item
\[
\begin{tikzpicture}[auto,swap]
\tikzstyle{vertex}=[circle,fill=black,minimum size=3pt,inner sep=0pt]
\tikzstyle{edge}=[draw,->]
\tikzstyle{cycle1}=[draw,->,out=130, in=50, loop, distance=40pt]
\tikzstyle{cycle2}=[draw,->,out=100, in=30, loop, distance=40pt]
   
\node[vertex,label=below:$v$] (0) at (0,0) {};

\path (0) edge[cycle1] node[above] {$e$} (0);

\end{tikzpicture}\quad\longto\quad
\begin{tikzpicture}[auto,swap]
\tikzstyle{vertex}=[circle,fill=black,minimum size=3pt,inner sep=0pt]
\tikzstyle{edge}=[draw,->]
\tikzstyle{cycle1}=[draw,->,out=130, in=50, loop, distance=40pt]
\tikzstyle{cycle2}=[draw,->,out=100, in=30, loop, distance=40pt]
   
\node[vertex,label=below:$v$] (0) at (0,0) {};
\node[vertex,label=below:$w$] (1) at (1,0) {};

\path (0) edge[cycle1] node[above] {$e$} (0);
\path (0) edge[edge] node[above] {$g$} (1);

\end{tikzpicture}
\]
$F^0\setminus f^0(E^0)=\{w\}$ is hereditary and saturated. Also, $f^1(\{e\})=\{e\}$ and
\[
t^{-1}_F(f^0(E^0))= t^{-1}_F(\{v\})=\{e\}.
\]
\item
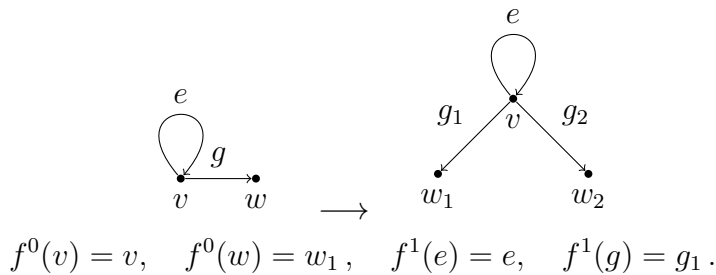
\[
\begin{tikzpicture}[auto,swap]
\tikzstyle{vertex}=[circle,fill=black,minimum size=3pt,inner sep=0pt]
\tikzstyle{edge}=[draw,->]
\tikzstyle{cycle1}=[draw,->,out=130, in=50, loop, distance=40pt]
\tikzstyle{cycle2}=[draw,->,out=100, in=30, loop, distance=40pt]
   
\node[vertex,label=below:$v$] (0) at (0,0) {};
\node[vertex,label=below:$w$] (1) at (1,0) {};

\path (0) edge[cycle1] node[above] {$e$} (0);
\path (0) edge[edge] node[above] {$g$} (1);

\end{tikzpicture}\quad\longto\quad
\begin{tikzpicture}[auto,swap]
\tikzstyle{vertex}=[circle,fill=black,minimum size=3pt,inner sep=0pt]
\tikzstyle{edge}=[draw,->]
\tikzstyle{cycle1}=[draw,->,out=130, in=50, loop, distance=40pt]
\tikzstyle{cycle2}=[draw,->,out=100, in=30, loop, distance=40pt]
   
\node[vertex,label=below:$v$] (0) at (0,0) {};
\node[vertex,label=below:$w_1$] (1) at (-1,-1) {};
\node[vertex,label=below:$w_2$] (2) at (1,-1) {};

\path (0) edge[cycle1] node[above] {$e$} (0);
\path (0) edge[edge] node[above left] {$g_1$} (1);
\path (0) edge[edge] node[above right] {$g_2$} (2);

\end{tikzpicture}
\]
\[
f^0(v)=v,\quad f^0(w)=w_1\,,\quad f^1(e)=e,\quad f^1(g)=g_1\,.
\]
$F^0\setminus f^0(E^0)=\{w_2\}$ is hereditary and saturated. Also, $f^1(\{e,g\})=\{e,g_1\}$ and
\[
t^{-1}_F(f^0(E^0))= t^{-1}_F(\{v,w_1\})=\{e,g_1\}.
\]
\end{enumerate}

\noindent\underline{Counterexamples:}
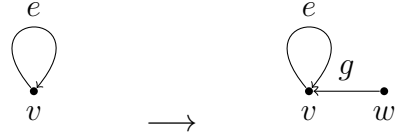
\begin{enumerate}
\item 
\[
\begin{tikzpicture}[auto,swap]
\tikzstyle{vertex}=[circle,fill=black,minimum size=3pt,inner sep=0pt]
\tikzstyle{edge}=[draw,->]
\tikzstyle{cycle1}=[draw,->,out=130, in=50, loop, distance=40pt]
\tikzstyle{cycle2}=[draw,->,out=100, in=30, loop, distance=40pt]
   
\node[vertex,label=below:$v$] (0) at (0,0) {};

\path (0) edge[cycle1] node[above] {$e$} (0);

\end{tikzpicture}\quad\longto\quad
\begin{tikzpicture}[auto,swap]
\tikzstyle{vertex}=[circle,fill=black,minimum size=3pt,inner sep=0pt]
\tikzstyle{edge}=[draw,->]
\tikzstyle{cycle1}=[draw,->,out=130, in=50, loop, distance=40pt]
\tikzstyle{cycle2}=[draw,->,out=100, in=30, loop, distance=40pt]
   
\node[vertex,label=below:$v$] (0) at (0,0) {};
\node[vertex,label=below:$w$] (1) at (1,0) {};

\path (0) edge[cycle1] node[above] {$e$} (0);
\path (1) edge[edge] node[above] {$g$} (0);

\end{tikzpicture}
\]
\[
f^0(v)=v,\quad f^1(e)=e.
\]
$F^0\setminus f^0(E^0)=\{v,w\}\setminus\{v\}=\{w\}$ is saturated but not hereditary.
Also,
\[
 t^{-1}_F( f^0(E^0))= t^{-1}_F(\{v\})
=\{e,g\}\neq \{e\}=f^1(E^1).
\]
\item
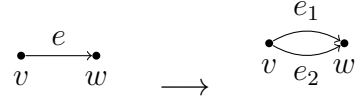
\[
\begin{tikzpicture}[auto,swap]
\tikzstyle{vertex}=[circle,fill=black,minimum size=3pt,inner sep=0pt]
\tikzstyle{edge}=[draw,->]
\tikzstyle{cycle1}=[draw,->,out=130, in=50, loop, distance=40pt]
\tikzstyle{cycle2}=[draw,->,out=100, in=30, loop, distance=40pt]
   
\node[vertex,label=below:$v$] (0) at (0,0) {};
\node[vertex,label=below:$w$] (1) at (1,0) {};

\path (0) edge[edge] node[above] {$e$} (1);

\end{tikzpicture}\quad\longto\quad
\begin{tikzpicture}[auto,swap]
\tikzstyle{vertex}=[circle,fill=black,minimum size=3pt,inner sep=0pt]
\tikzstyle{edge}=[draw,->]
\tikzstyle{cycle1}=[draw,->,out=130, in=50, loop, distance=40pt]
\tikzstyle{cycle2}=[draw,->,out=100, in=30, loop, distance=40pt]
   
\node[vertex,label=below:$v$] (0) at (0,0) {};
\node[vertex,label=below:$w$] (1) at (1,0) {};

\path (0) edge[edge, bend right] node[below] {$e_2$} (1);
\path (0) edge[edge,bend left] node[above] {$e_1$} (1);

\end{tikzpicture}
\]\[
f^0(v)=v,\quad f^0(w)=w,\quad f^1(e)=e_1\,.
\]
$F^0\setminus f^0(E^0)=\{v,w\}\setminus\{v,w\}=\emptyset$ is hereditary and saturated.
But
\[
 t^{-1}_F( f^0(E^0))=F^1=\{e_1,e_2\}\neq \{e_1\}=f^1(E^1).
\]
\end{enumerate}

\noindent\underline{The intersection of graphs:}

Let $F$ and $G$ be graphs. Assume that $s_F$ and $t_F$ agree, respectively, with $s_G$ and $t_G$ on $F^1\cap G^1$.
Then we can define the \underline{intersection} graph 
\[
F\cap G:=(F^0\cap G^0,F^1\cap G^1,s_\cap,t_\cap),
\]
where $s_\cap,t_\cap:F^1\cap G^1\to F^0\cap G^0$,
\[
\forall\; e\in F^1\cap G^1:\quad s_\cap(e)=s_G(e)=s_F(e),\quad t_\cap(e)=t_G(e)=t_F(e).
\]
$F\cap G$ is, clearly, a subgraph of both $F$ and $G$. We say that the intersection is \underline{admissible}
iff both inclusions $F\cap G\hookrightarrow F$ and $F\cap G\hookrightarrow G$ are admissible inclusions.

\noindent\underline{Examples:}
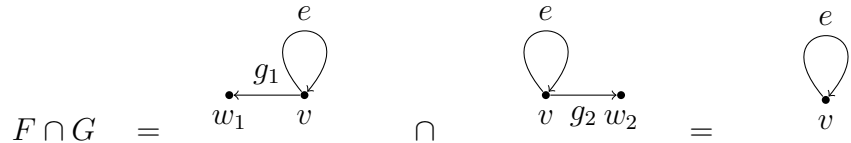
\begin{enumerate}
\item
\[
F\cap G\quad =\quad
\begin{tikzpicture}[auto,swap]
\tikzstyle{vertex}=[circle,fill=black,minimum size=3pt,inner sep=0pt]
\tikzstyle{edge}=[draw,->]
\tikzstyle{cycle1}=[draw,->,out=130, in=50, loop, distance=40pt]
\tikzstyle{cycle2}=[draw,->,out=100, in=30, loop, distance=40pt]
   
\node[vertex,label=below:$v$] (0) at (0,0) {};
\node[vertex,label=below:$w_1$] (1) at (-1,0) {};

\path (0) edge[cycle1] node[above] {$e$} (0);
\path (0) edge[edge] node {$g_1$} (1);

\end{tikzpicture}\quad\cap\quad
\begin{tikzpicture}[auto,swap]
\tikzstyle{vertex}=[circle,fill=black,minimum size=3pt,inner sep=0pt]
\tikzstyle{edge}=[draw,->]
\tikzstyle{cycle1}=[draw,->,out=130, in=50, loop, distance=40pt]
\tikzstyle{cycle2}=[draw,->,out=100, in=30, loop, distance=40pt]
   
\node[vertex,label=below:$v$] (0) at (0,0) {};
\node[vertex,label=below:$w_2$] (1) at (1,0) {};

\path (0) edge[cycle1] node[above] {$e$} (0);
\path (0) edge[edge] node {$g_2$} (1);

\end{tikzpicture}\quad = \quad
\begin{tikzpicture}[auto,swap]
\tikzstyle{vertex}=[circle,fill=black,minimum size=3pt,inner sep=0pt]
\tikzstyle{edge}=[draw,->]
\tikzstyle{cycle1}=[draw,->,out=130, in=50, loop, distance=40pt]
\tikzstyle{cycle2}=[draw,->,out=100, in=30, loop, distance=40pt]
   
\node[vertex,label=below:$v$] (0) at (0,0) {};

\path (0) edge[cycle1] node[above] {$e$} (0);

\end{tikzpicture}
\]
\[
F^0\cap G^0=\{v\},\quad F^1\cap G^1=\{e\},\quad s_\cap(e)=t_\cap(e)=v.
\]
The intersection is admissible because:
\begin{enumerate}
\item The subset $F^0\setminus(F^0\cap G^0)=\{v,w_1\}\setminus\{v\}=\{w_1\}$ is hereditary and saturated in $F$,
and the subset $G^0\setminus(F^0\cap G^0)=\{v,w_2\}\setminus\{v\}=\{w_2\}$ is hereditary saturated in $G$;
\item $t_F^{-1}(F^0\cap G^0)= t^{-1}_F(\{v\})=\{e\}=F^1\cap G^1$ and
$t^{-1}_G(F^0\cap G^0)= t^{-1}_G(\{v\})=F^1\cap G^1$.
\end{enumerate}
\item 
\[
F\cap G\quad = \quad
\begin{tikzpicture}[auto,swap]
\tikzstyle{vertex}=[circle,fill=black,minimum size=3pt,inner sep=0pt]
\tikzstyle{edge}=[draw,->]
   
\node[vertex,label=below:$w_1$] (0) at (0,0) {};
\node[vertex,label=below:$v$] (1) at (2,0) {}; 

\path (1) edge[edge] node[above] {$\{x_i\}_{i\in\mathbb{N}}$} (0);

\end{tikzpicture}\quad \cap\quad
\begin{tikzpicture}[auto,swap]
\tikzstyle{vertex}=[circle,fill=black,minimum size=3pt,inner sep=0pt]
\tikzstyle{edge}=[draw,->]
   
\node[vertex,label=below:$v$] (0) at (0,0) {};
\node[vertex,label=below:$w_2$] (1) at (2,0) {}; 

\path (0) edge[edge] node[above] {$\{y_i\}_{i\in\mathbb{N}}$} (1);

\end{tikzpicture}\quad =\quad
\begin{tikzpicture}[auto,swap]
\tikzstyle{vertex}=[circle,fill=black,minimum size=3pt,inner sep=0pt]
\tikzstyle{edge}=[draw,->]
   
\node[vertex,label=below:$v$] (1) at (2,0) {}; 


\end{tikzpicture}
\]
is admissible becase:
\begin{enumerate}
\item $F^0\setminus(F^0\cap G^0)=\{w_1,v\}\setminus\{v\}=\{w_1\}$ is hereditary and saturated in $F$,
and 
$$
G^0\setminus(F^0\cap G^0)=\{w_2,v\}\setminus\{v\}=\{w_2\}
$$ 
is hereditary and saturated in $G$;
\item $t^{-1}_F(F^0\cap G^0)= t^{-1}_F(\{v\})=\emptyset=t^{-1}_G(\{v\})=t^{-1}_G(F^0\cap G^0)$ and $F^1\cap G^1=\emptyset$.
\end{enumerate}
\end{enumerate}
\noindent\underline{Counterexamples:}
\begin{enumerate}
\item
\[
F\cap G\quad = \quad
 \begin{tikzpicture}[auto,swap]
\tikzstyle{vertex}=[circle,fill=black,minimum size=3pt,inner sep=0pt]
\tikzstyle{edge}=[draw,->]
   
\node[vertex,label=below:$w_1$] (0) at (0,0) {};
\node[vertex,label=below:$v$] (1) at (1,0) {}; 

\path (1) edge[edge] node[above] {$e_1$} (0);

\end{tikzpicture}\quad \cap\quad
 \begin{tikzpicture}[auto,swap]
\tikzstyle{vertex}=[circle,fill=black,minimum size=3pt,inner sep=0pt]
\tikzstyle{edge}=[draw,->]
   
\node[vertex,label=below:$v$] (0) at (0,0) {};
\node[vertex,label=below:$w_2$] (1) at (1,0) {}; 

\path (0) edge[edge] node[above] {$e_2$} (1);

\end{tikzpicture}\quad =\quad
 \begin{tikzpicture}[auto,swap]
\tikzstyle{vertex}=[circle,fill=black,minimum size=3pt,inner sep=0pt]
\tikzstyle{edge}=[draw,->]
   
\node[vertex,label=below:$v$] (0) at (0,0) {};


\end{tikzpicture}
\]
is \underline{not} admissible because $F^0\setminus(F^0\cap G^0)=\{w_1,v\}\setminus\{v\}=\{w_1\}$ is \underline{not}
saturated. (It is hereditary.)
\item
\[
F\cap G\quad = \quad
 \begin{tikzpicture}[auto,swap]
\tikzstyle{vertex}=[circle,fill=black,minimum size=3pt,inner sep=0pt]
\tikzstyle{edge}=[draw,->]
   
\node[vertex,label=below:$w_1$] (0) at (0,0) {};
\node[vertex,label=below:$v$] (1) at (1,0) {}; 

\path (0) edge[edge] node[above] {$e_1$} (1);

\end{tikzpicture}\quad \cap \quad
 \begin{tikzpicture}[auto,swap]
\tikzstyle{vertex}=[circle,fill=black,minimum size=3pt,inner sep=0pt]
\tikzstyle{edge}=[draw,->]
   
\node[vertex,label=below:$v$] (0) at (0,0) {};
\node[vertex,label=below:$w_2$] (1) at (1,0) {}; 

\path (1) edge[edge] node[above] {$e_2$} (0);

\end{tikzpicture}\quad = \quad
 \begin{tikzpicture}[auto,swap]
\tikzstyle{vertex}=[circle,fill=black,minimum size=3pt,inner sep=0pt]
\tikzstyle{edge}=[draw,->]
   
\node[vertex,label=below:$v$] (0) at (0,0) {};


\end{tikzpicture}
\]
is \underline{not} admissible because $F^0\setminus(F^0\cap G^0)=\{w_1,v\}\setminus\{v\}=\{w_1\}$
is \underline{not} hereditary. (It is saturated.)
\item
\[
F\cap G\quad=\quad
 \begin{tikzpicture}[auto,swap]
\tikzstyle{vertex}=[circle,fill=black,minimum size=3pt,inner sep=0pt]
\tikzstyle{edge}=[draw,->]
   
\node[vertex,label=below:$v$] (0) at (0,0) {};
\node[vertex,label=below:$w$] (1) at (1,0) {}; 

\path (0) edge[edge,bend left] node[above] {$e_1$} (1);
\path (0) edge[edge] node[below] {$e_2$} (1);

\end{tikzpicture}\quad\cap\quad
 \begin{tikzpicture}[auto,swap]
\tikzstyle{vertex}=[circle,fill=black,minimum size=3pt,inner sep=0pt]
\tikzstyle{edge}=[draw,->]
   
\node[vertex,label=below:$v$] (0) at (0,0) {};
\node[vertex,label=below:$w$] (1) at (1,0) {}; 

\path (0) edge[edge,bend right] node[below] {$e_3$} (1);
\path (0) edge[edge] node[above] {$e_2$} (1);

\end{tikzpicture}\quad=\quad 
 \begin{tikzpicture}[auto,swap]
\tikzstyle{vertex}=[circle,fill=black,minimum size=3pt,inner sep=0pt]
\tikzstyle{edge}=[draw,->]
   
\node[vertex,label=below:$v$] (0) at (0,0) {};
\node[vertex,label=below:$w$] (1) at (1,0) {}; 

\path (0) edge[edge] node[above] {$e_2$} (1);

\end{tikzpicture}
\]
is \underline{not} admissible because
\[
t^{-1}_F(F^0\cap G^0)= t^{-1}_F(\{v,w\})
=\{e_1,e_2\}\neq\{e_2\}=F^1\cap G^1.
\]
However, both $F^0\setminus(F^0\cap G^0)$ and $G^0\setminus(F^0\cap G^0)$ are empty,
so they are hereditary and saturated.
\end{enumerate}

\noindent\underline{The union of graphs:}

Let $F$ and $G$ be graphs. Again, assume that $s_F$ and $t_F$ agree, respectively, with $s_G$ and $t_G$ on $F^1\cap G^1$.
Then we can define the \underline{union} graph
\[
F\cup G:=(F^0\cup G^0,F^1\cup G^1,s_\cup,t_\cup),
\]
where $s_\cup,t_\cup:F^1\cup G^1\to F^0\cup G^0$,
\[
\forall\;x\in F^1\cup G^1:\quad s_\cup=\begin{cases}s_F(x) & \text{for }x\in F^1\\s_G(x) & \text{for } x\in G^1\end{cases},
\]
\[
\forall\;x\in F^1\cup G^1:\quad t_\cup=\begin{cases}t_F(x) & \text{for }x\in F^1\\t_G(x) & \text{for } x\in G^1\end{cases}.
\]
Note that $F$ and $G$ are subgraphs of $F\cup G$. We say that the union is \underline{admissible}
iff both the inclusions $F\hookrightarrow F\cup G$ and $G\hookrightarrow F\cup G$ are admissible.

\begin{lemma}
Let $F$ and $G$ be graphs whose source and target maps agree, respectively on $F^1\cap G^1$.
Then, if the intersection graph $F\cap G$ is admissible, so is the union graph $F\cup G$.
\end{lemma}
\begin{proof}
\begin{enumerate}
\item $(F^0\cup G^0)\setminus F^0=G^0\setminus F^0=G^0\setminus(F^0\cap G^0)$.
Since $F\cap G\hookrightarrow G$ is admissible, $G^0\setminus(F^0\cap G^0)$ is hereditary
and saturated in $G$. 

We need to show that $G^0\setminus(F^0\cap G^0)$ is hereditary
and saturated in $F\cup G$. To this end, consider $p:=(e_1,\ldots,e_n)\in FP(F\cup G)$ such that
$s_\cup(p)\in G^0\setminus(F^0\cap G^0)$. Then $s_\cup(p)=s_\cup(e_1)\notin F^0$,
so $e_1\notin F^1$, whence $e_1\in G^1$, so $s_\cup(e_1)=s_G(e_1)$.
As $G^0\setminus(F^0\cap G^0)$ is hereditary in $G$, $s_\cup(e_2)=t_\cup(e_1)=t_G(e_1)\in G^0\setminus(F^0\cap G^0)$.
Repeating this reasoning for all $\{e_i\}_{i=2}^n$, we conclude that 
$t_\cup(p)=t_\cup(e_n)=t_G(e_n)\in G^0\setminus(F^0\cap G^0)$, so $G^0\setminus(F^0\cap G^0)$ is hereditary in $F\cup G$.

Next, to establish that $G^0\setminus(F^0\cap G^0)$ is saturated in $F\cup G$, we consider all elements in $F^0\cup G^0$
that emit an edge. Note first that any edge ending at a vertex in $G^0\setminus(F^0\cap G^0)$ must begin at a vertex in $G^0$,
so we only need to consider vertices in $F^0\cap G^0$:
\[
\begin{tikzpicture}[auto,swap]
\tikzstyle{vertex}=[circle,fill=black,minimum size=3pt,inner sep=0pt]
\tikzstyle{edge}=[draw,->]
   
\node[vertex,label=below:$F^0\cap G^0$] (1) at (1,0) {}; 
\node[vertex] (2) at (2.5,0) {}; 

\path (1) edge[edge] node {} (2);

\draw (0,0) circle (1.85) (-0.5,1) node {$F^0\setminus G^0$};
\draw (2,0) circle (1.85) (2.5,1) node {$G^0\setminus F^0$};

\end{tikzpicture}
\]
Also, if $v$ is a sink in $G$ but not in $F\cup G$, it emits an edge ending at a vertex outside of $G^0\setminus(F^0\cap G^0)$,
so we can disregard it. Furthermore, since $s^{-1}_G(\{v\})\subseteq s^{-1}_\cup(\{v\})$, the finiteness of $s^{-1}_\cup(\{v\})$
implies the finiteness of $s^{-1}_G(\{v\})$. Hence we only need to consider vertices in $F^0\cap G^0$ that are no sinks in $G$,
that are finite emitters in $G$, and that emit all their edges to $G^0\setminus(F^0\cap G^0)$. For all such vertices, we have
\[
\{t_\cup(e)~|~e\in s^{-1}_\cup(\{v\})\}=\{t_G(e)~|~e\in s^{-1}_G(\{v\})\}
\]
because $t_\cup(e)\in G^0\setminus(F^0\cap G^0)$ implies that $e\in G^1$. Finally, such vertices do not exist by the
saturation property of $G^0\setminus(F^0\cap G^0)$ in $G$, so $G^0\setminus(F^0\cap G^0)$ is saturated in $F\cup G$.

A symmetric argument proves that $F^0\setminus(F^0\cap G^0)$ is hereditary and staurated in $F\cup G$.
\item 
First, taking an advantage of the admissibility of $(F\cap G)\subseteq G$, we compute
\begin{align*}
t^{-1}_\cup(F^0)\setminus F^1=(G^1\setminus F^1)\cap t^{-1}_G(F^0\cap G^0)=(G^1\setminus F^1)\cap (F^1\cap G^1)=\emptyset.
\end{align*}
Therefore, as $F^1\subseteq t^{-1}_\cup(F^0)$, we conclude that $F^1= t^{-1}_\cup(F^0)$.
Much in the same way, one shows that $t^{-1}_\cup(G^0)=G^1$. 
\end{enumerate}
\end{proof}
\noindent\underline{Remark:}
The opposite implication:
\[
F\hookrightarrow F\cup G\hookleftarrow G\quad \text{is admissible}\quad \Rightarrow\quad
F\hookleftarrow F\cap G\hookrightarrow G\quad \text{is admissible},
\]
is \underline{not} true:
\[
F:=\quad
\begin{tikzpicture}[auto,swap]
\tikzstyle{vertex}=[circle,fill=black,minimum size=3pt,inner sep=0pt]
\tikzstyle{edge}=[draw,->]
   
\node[vertex,label=below:$w_1$] (0) at (0,0) {};
\node[vertex,label=below:$v$] (1) at (2,0) {}; 

\path (1) edge[edge] node[above] {$\{x_i\}_{i\in\mathbb{N}}$} (0);

\end{tikzpicture}\qquad\qquad
G:=\quad
\begin{tikzpicture}[auto,swap]
\tikzstyle{vertex}=[circle,fill=black,minimum size=3pt,inner sep=0pt]
\tikzstyle{edge}=[draw,->]
   
\node[vertex,label=below:$v$] (0) at (0,0) {};
\node[vertex,label=below:$w_2$] (1) at (1,0) {}; 

\path (0) edge[edge] node[above] {$e$} (1);

\end{tikzpicture}
\]
\[
F\cap G=(\{v\},\emptyset,\emptyset,\emptyset),\qquad F\cup G=(\{w_1,v,w_2\},\{x_i\}_{i\in\mathbb{N}}\cup\{e\},s_\cup,t_\cup).
\]

Let us check first that $F\hookrightarrow F\cup G$ is admissible. The set 
\[
(F^0\cup G^0)\setminus F^0=G^0\setminus(F^0\cap G^0)=\{v,w_2\}\setminus\{v\}=\{w_2\}
\]
is hereditary in $F\cup G$ because $w_2$ is a sink in $F\cup G$. It is also saturated
in $F\cup G$ because $w_1$ is a sink in $F\cup G$ and $v$ is an infinite emitter in $F\cup G$.

Next, 
\[
t^{-1}_\cup(F^0)= t^{-1}_\cup(\{v,w_1\})
=\{x_i\}_{i\in\mathbb{N}}=F^1.
\]
Hence $F\hookrightarrow F\cup G$ is admissible.
For the inclusion $G\hookrightarrow F\cup G$, consider the set
\[
(F^0\cup G^0)\setminus G^0=F^0\setminus(F^0\cap G^0)=
\{w_1,v\}\setminus\{v\}=\{w_1\}.
\]
It is hereditary in $F\cup G$ because $w_1$ is a sink in $F\cup G$. It is also saturated
in $F\cup G$ beacuse $w_2$ is a sink in $F\cup G$.
Finally,
\[
t^{-1}_\cup(G^0)=t^{-1}_\cup(\{v,w_2\})=\{e\}=G^1.
\]
Thus we have shown that the union $F\cup G$ is admissible. On the other hand,
the intersection $F\cap G$ is \underline{not} admissible because the set
$G^0\setminus(F^0\cap G^0)=\{w_2\}$ is not saturated in $G$:
\[
v\notin\{w_2\}\quad\text{and}\quad \{t_G(x)~|~x\in s^{-1}_G(\{v\})\}=\{w_2\}.
\]

\noindent\underline{Elementary observations:}
\begin{enumerate}
\item The properties of being hereditary and saturated are \underline{not} preserved
by the inclusion of graphs:
\begin{enumerate}
\item
\begin{tikzpicture}[auto,swap]
\tikzstyle{vertex}=[circle,fill=black,minimum size=3pt,inner sep=0pt]
\tikzstyle{edge}=[draw,->]
   
\node[vertex,label=below:$v$] (0) at (0,0) {};
\node[vertex,label=below:$w$] (3) at (1,0) {}; 

\path (0) edge[edge] node[above] {$e$} (3);

\end{tikzpicture}\qquad
\begin{tikzpicture}[auto,swap]
\tikzstyle{vertex}=[circle,fill=black,minimum size=3pt,inner sep=0pt]
\tikzstyle{edge}=[draw,->]
   
\node[vertex,label=below:$v$] (0) at (0,0) {};
\node[vertex,label=below:$w$] (1) at (1,0) {}; 
\node[vertex,label=below:$w'$] (2) at (2,0) {}; 

\path (0) edge[edge] node[above] {$e$} (1);
\path (1) edge[edge] node[above] {$e'$} (2);

\end{tikzpicture}

\quad$F$\qquad\quad  $\subseteq$ \quad\qquad  $G$

$\{w\}$ is hereditary in $F$ but not in $G$.
\item
\begin{tikzpicture}[auto,swap]
\tikzstyle{vertex}=[circle,fill=black,minimum size=3pt,inner sep=0pt]
\tikzstyle{edge}=[draw,->]
   
\node[vertex,label=below:$v$] (0) at (0,0) {};
\node[vertex,label=below:$w$] (3) at (1,0) {}; 

\path (0) edge[edge] node[above] {$e$} (3);

\end{tikzpicture}\qquad
\begin{tikzpicture}[auto,swap]
\tikzstyle{vertex}=[circle,fill=black,minimum size=3pt,inner sep=0pt]
\tikzstyle{edge}=[draw,->]
   
\node[vertex,label=below:$w'$] (0) at (0,0) {};
\node[vertex,label=below:$v$] (1) at (1,0) {}; 
\node[vertex,label=below:$w$] (2) at (2,0) {}; 

\path (0) edge[edge] node[above] {$e'$} (1);
\path (1) edge[edge] node[above] {$e$} (2);

\end{tikzpicture}

\quad$F$\qquad\quad  $\subseteq$ \quad\qquad  $G$

$\{v\}$ is saturated in $F$ but not in $G$.
\end{enumerate}
However, both properties are preserved by special inclusions $F\subseteq F\cup G$
for the special set $F^0\setminus(F^0\cap G^0)$
because there are no edges like this:
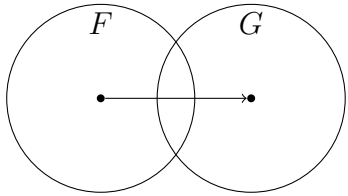
\[
\begin{tikzpicture}[auto,swap]
\tikzstyle{vertex}=[circle,fill=black,minimum size=3pt,inner sep=0pt]
\tikzstyle{edge}=[draw,->]
   
\node[vertex] (0) at (0,0) {};
\node[vertex] (2) at (2,0) {}; 

\path (0) edge[edge] node {} (2);

\draw (0,0) circle (1.25) (0,1) node {$F$};
\draw (2,0) circle (1.25) (2,1) node {$G$};

\end{tikzpicture}
\]
\item Restriction of graphs to subgraphs does \underline{not} preserve the saturation property
even in the special case of $F^0\setminus(F^0\cap G^0)$ in $F\subseteq F\cup G$.
However, it always preserves the property of being hereditary: if $H\subseteq F^0$, $F\subseteq G$,
is not hereditary in $F$, it is not hereditary in $G$. Indeed, if there is a path starting at $v\in H$ and ending
at $w\in F^0\setminus H$, then it is also a path starting at $v\in H$ and ending at $w\in G^0\setminus H$.
\end{enumerate}

\noindent\underline{Extended graph:}

Let $E=(E^0,E^1,s_E,t_E)$ be a graph. The \underline{extended graph} $\bar{E}:=(\bar{E}^0,\bar{E}^1,s_{\bar{E}},t_{\bar{E}})$
of the graph $E$ is defined as follows
\[
\bar{E}^0:=E^0,\quad \bar{E}^1:=E^1\sqcup (E^1)^*,\quad (E^1)^*:=\{e^*~|~e\in E^1\},
\]
\[
\forall\; e\in E^1:\quad s_{\bar{E}}(e):=s_E(e),\quad t_{\bar{E}}(e):=t_E(e),
\]
\[
\forall\; e^*\in (E^1)^*:\quad s_{\bar{E}}(e^*):=t_E(e),\quad t_{\bar{E}}(e^*):=s_E(e).
\]

Thus $E$ is a subgraph of $\bar{E}$.

\noindent\underline{Examples:}
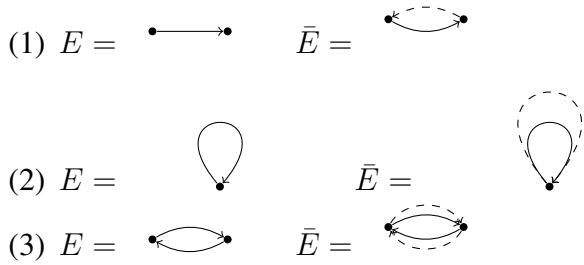
\begin{enumerate}
\item $E=$\quad\begin{tikzpicture}[auto,swap]
\tikzstyle{vertex}=[circle,fill=black,minimum size=3pt,inner sep=0pt]
\tikzstyle{edge}=[draw,->]
   
\node[vertex] (0) at (0,0) {};
\node[vertex] (3) at (1,0) {}; 

\path (0) edge[edge] node {} (3);

\end{tikzpicture}\qquad $\bar{E}=$\quad
\begin{tikzpicture}[auto,swap]
\tikzstyle{vertex}=[circle,fill=black,minimum size=3pt,inner sep=0pt]
\tikzstyle{edge}=[draw,->]
\tikzstyle{cycle1}=[draw,->,out=130, in=50, loop, distance=40pt]
\tikzstyle{cycle2}=[draw,->,out=100, in=30, loop, distance=40pt]
   
\node[vertex] (0) at (0,0) {};
\node[vertex] (3) at (1,0) {}; 

\path (0) edge[edge,bend right] node[below] {} (3);
\path (3) edge[edge,bend right,dashed] node {} (0);

\end{tikzpicture}
\item
$E=$\quad
\begin{tikzpicture}[auto,swap]
\tikzstyle{vertex}=[circle,fill=black,minimum size=3pt,inner sep=0pt]
\tikzstyle{edge}=[draw,->]
\tikzstyle{cycle1}=[draw,->,out=130, in=50, loop, distance=40pt]
\tikzstyle{cycle2}=[draw,->,out=130, in=50, loop, distance=60pt]
   
\node[vertex] (0) at (0,0) {};

\path (0) edge[cycle1] node {} (0);

\end{tikzpicture}\qquad
$\bar{E}=$\quad
\begin{tikzpicture}[auto,swap]
\tikzstyle{vertex}=[circle,fill=black,minimum size=3pt,inner sep=0pt]
\tikzstyle{edge}=[draw,->]
\tikzstyle{cycle1}=[draw,->,out=130, in=50, loop, distance=40pt]
\tikzstyle{cycle2}=[draw,->,out=130, in=50, loop, distance=60pt]
   
\node[vertex] (0) at (0,0) {};

\path (0) edge[cycle1] node {} (0);
\path (0) edge[cycle2,dashed] node {} (0);

\end{tikzpicture}
\item
$E=$\quad
\begin{tikzpicture}[auto,swap]
\tikzstyle{vertex}=[circle,fill=black,minimum size=3pt,inner sep=0pt]
\tikzstyle{edge}=[draw,->]
\tikzstyle{cycle1}=[draw,->,out=130, in=50, loop, distance=40pt]
\tikzstyle{cycle2}=[draw,->,out=100, in=30, loop, distance=40pt]
   
\node[vertex] (0) at (0,0) {};
\node[vertex] (3) at (1,0) {}; 

\path (0) edge[edge,bend left] node {} (3);
\path (3) edge[edge,bend left] node {} (0);

\end{tikzpicture}\qquad
$\bar{E}=$\quad
\begin{tikzpicture}[auto,swap]
\tikzstyle{vertex}=[circle,fill=black,minimum size=3pt,inner sep=0pt]
\tikzstyle{edge}=[draw,->]
\tikzstyle{cycle1}=[draw,->,out=130, in=50, loop, distance=40pt]
\tikzstyle{cycle2}=[draw,->,out=100, in=30, loop, distance=40pt]
   
\node[vertex] (0) at (0,0) {};
\node[vertex] (3) at (1,0) {}; 

\path (0) edge[edge,bend left] node {} (3);
\path (0) edge[edge,bend left=60,dashed] node {} (3);
\path (3) edge[edge,bend left] node {} (0);
\path (3) edge[edge,bend left=60,dashed] node {} (0);

\end{tikzpicture}
\end{enumerate}

\section{Graph algebras}

\subsection{Path algebras}
Let $V$ be any vector space over a field $k$. To endow $V$ with an algebra structure, we have
to define the multiplication map $V\times V\overset{m}{\mapsto}V$, which is a bilinear map
satisfying some conditions. Any such a map is uniquely determined by its value on pairs of basis elements $(e_i,e_j)$,
and any assignment $(e_i,e_j)\mapsto v_{ij}\in V$ defines a bilinear map from $V\times V$ to $V$.
Now, let $E$ be any graph, and $FP(E)$ the set of all its finite paths. Consider the vector space
\[
kE:=\{f\in{\rm Map}(FP(E),k)~|~f(p)\neq 0\text{ for finitely many } p\in FP(E)\},
\]
where the addition and scalar multiplication are pointwise. Then the set of functions $\{\chi_p\}_{p\in FP(E)}$
given by
\[
\chi_p(q)=\begin{cases}1 &\text{for }p=q\\0 & \text{otherwise}\end{cases}
\]
is a linear basis of $kE$. Indeed, let $\{q_1,\ldots,q_n\}$ be the support of $f\in kE$. Then
\[
f=\sum_{i=1}^nf(q_i)\chi_{q_i}
\]
because, $\forall\;p\in FP(E)$:
\begin{align*}
\left(\sum_{i=1}^nf(q_i)\chi_{q_i}\right)(p)&=\sum_{i=1}^nf(q_i)\chi_{q_i}(p)\\
&=\begin{cases}0 & \text{if }p\notin\{q_1,\ldots,q_n\}\\f(q_i) & \text{if }p=q_i\text{ for }i\in\{1,\ldots,n\}\end{cases}\\
&=f(p).
\end{align*}
Hence $\{\chi_p\}_{p\in FP(E)}$ spans $kE$.

To see the linear independence, take any finite subset $\{\chi_{p_1},\ldots,\chi_{p_m}\}\subseteq\{\chi_p\}_{p\in FP(E)}$,
and suppose that $\sum_{i=1}^m\alpha_i\chi_{p_i}=0$. Then $\forall\; j\in\{1,\ldots,m\}$
\[
0=\left(\sum_{i=1}^m\alpha_i\chi_{p_i}\right)(p_j)=\sum_{i=1}^m\alpha_i\chi_{p_i}(p_j)=\alpha_j\,.
\]
Thus we have shown that $\{\chi_p\}_{p\in FP(E)}$ is a basis of $kE$. Now we will use $\{\chi_p\}_{p\in FP(E)}$
to define a bilinear map:
\[
m:kE\times kE\longto kE,\qquad m(\chi_p,\chi_q):=\begin{cases}\chi_{pq} & \text{if }t(p)=s(q)\\0 & \text{otherwise}\end{cases}.
\]

\begin{proposition}
The bilinear map $m:kE\times kE\to kE$ defines an algebra structure on $kE$.
\end{proposition}
\begin{proof}
To check the associativity of $m$ it suffices to verify it on basis elements:
\begin{align*}
m(m(\chi_p,\chi_q),\chi_x)&
=\begin{cases}
m(\chi_{pq},\chi_x) & \text{if $t(p)=s(q)$}\\
0 & \text{if $t(p)\neq s(q)$}\end{cases}\\
&=\begin{cases}
\chi_{pqx} & \text{if $t(q)=s(x)$ and $t(p)=s(q)$}\\
0 & \text{if $t(q)\neq s(x)$ and $t(p)=s(q)$}\\
0 & \text{if $t(p)\neq s(q)$}\end{cases},
\end{align*}
\begin{align*}
m(\chi_p,m(\chi_q,\chi_x))&
=\begin{cases}
m(\chi_p,\chi_{qx}) & \text{if $t(q)=s(x)$}\\
0 & \text{if $t(q)\neq s(x)$}\end{cases}\\
&=\begin{cases}
\chi_{pqx} & \text{if $t(p)=s(q)$ and $t(q)=s(x)$}\\
0 & \text{if $t(p)\neq s(q)$ and $t(q)=s(x)$}\\
0 & \text{if $t(q)\neq s(x)$}\end{cases}.
\end{align*}

Hence $m(m(\chi_p,\chi_q),\chi_x)=m(\chi_p,m(\chi_q,\chi_x))$ for any $p,q,x\in FP(E)$.
(The distributivity follows from the bilinearity of $m$.)
\end{proof}
\begin{definition}
Let $E$ be a graph. The above constructed algebra $(kE,+,0,m)$ is called the \underline{path algebra} of $E$.
\end{definition}

\noindent\underline{Elementary facts:}

The path algebra $kE$ of a graph $E$ is
\begin{enumerate}
\item finite dimensional $\iff$ $E$ is finite and acyclic (no loops),
\item unital $\iff$ $E^0$ is finite,
\item commutative $\iff$ $E^1=\emptyset$ or each edge is a loop starting/ending at a different vertex.
\end{enumerate}

\subsection{Leavitt path algebras}
\begin{definition}
Let $A$ be a $k$-algebra and $S$ a subset of $A$. The \underline{ideal generated} by $S$ is the set of all
finite sums
\[
\sum_{i\in F}x_is_iy_i\,,
\]
where $s_i\in S$ and $x_i\,,y_i\in A$ for all $i\in F$.
\end{definition}

\begin{definition}
Let $E$ be a graph and $k$ be a field. The \underline{Leavitt path algebra} $L_k(E)$ of $E$
is the path algebra $k\bar{E}$ of the extended graph $\bar{E}$ divided by the ideal generated
by the following elements:
\begin{enumerate}
\item $\{\chi_{e^*}\chi_f-\delta_{ef}\chi_{t(e)}~|~e,f\in E^1\}$,
\item $\{\sum_{e\in s^{-1}(v)}\chi_{e}\chi_{e^*}-\chi_v~|~v\in E^0,~0<|s^{-1}(v)|<\infty\}$.
\end{enumerate}
\end{definition}

\noindent\underline{Examples:}
\begin{enumerate}
\item Matrix algebras:

$E=$\quad
\begin{tikzpicture}[auto,swap]
\tikzstyle{vertex}=[circle,fill=black,minimum size=3pt,inner sep=0pt]
\tikzstyle{edge}=[draw,->]
\tikzstyle{cycle1}=[draw,->,out=130, in=50, loop, distance=40pt]
\tikzstyle{cycle2}=[draw,->,out=100, in=30, loop, distance=40pt]
   
\node[vertex,label=below:$v_1$] (0) at (0,0) {};
\node[vertex,label=below:$v_2$] (3) at (0.5,0) {}; 
\node (2) at (1,0) {$\ldots$};
\node[vertex] (4) at (1.5,0) {};
\node[vertex,label=below:$v_n$] (1) at (2,0) {};

\path (0) edge[edge] node {} (3);
\path (4) edge[edge] node {} (1);

\end{tikzpicture}\qquad $L_k(E)=M_n(k)$

$E=$\quad
\begin{tikzpicture}[auto,swap]
\tikzstyle{vertex}=[circle,fill=black,minimum size=3pt,inner sep=0pt]
\tikzstyle{edge}=[draw,->]
\tikzstyle{cycle1}=[draw,->,out=130, in=50, loop, distance=40pt]
\tikzstyle{cycle2}=[draw,->,out=100, in=30, loop, distance=40pt]
   
\node[vertex] (0) at (0,0) {};
\node[vertex] (3) at (0.5,0) {}; 
\node (2) at (1,0) {$\ldots$};
\node[vertex] (4) at (1.5,0) {};
\node[vertex] (1) at (2,0) {};
\node (5) at (2.5,0) {$\ldots$};

\path (0) edge[edge] node {} (3);
\path (4) edge[edge] node {} (1);

\end{tikzpicture}\qquad $L_k(E)=M_\infty(k)=\bigcup_{n\in\mathbb{N}\setminus\{0\}}M_n(k)$
\[
M_n(k)\ni M\longmapsto 
\begin{bmatrix} M & \begin{matrix} 0 \\ \vdots \end{matrix} \\  \begin{matrix} 0 & \ldots \end{matrix} & 0\end{bmatrix}\in M_{n+1}(k)
\]
(arbitrary size finite matrices).
   


\item Laurent polynomial algebra:
$E=$\quad
\begin{tikzpicture}[auto,swap]
\tikzstyle{vertex}=[circle,fill=black,minimum size=3pt,inner sep=0pt]
\tikzstyle{edge}=[draw,->]
\tikzstyle{cycle1}=[draw,->,out=130, in=50, loop, distance=40pt]
\tikzstyle{cycle2}=[draw,->,out=130, in=50, loop, distance=70pt]
   
\node[vertex] (0) at (0,0) {};

\path (0) edge[cycle1] node {} (0);

\end{tikzpicture}\qquad $L_k(E)=k[\mathbb{Z}]$.
\item Leavitt algebras:
$E=$\quad
\begin{tikzpicture}[auto,swap]
\tikzstyle{vertex}=[circle,fill=black,minimum size=3pt,inner sep=0pt]
\tikzstyle{edge}=[draw,->]
\tikzstyle{cycle1}=[draw,->,out=130, in=50, loop, distance=40pt]
\tikzstyle{cycle2}=[draw,->,out=130, in=50, loop, distance=70pt]
   
\node[vertex] (0) at (0,0) {};
\node (1) at (0,1.25) {\vdots};

\path (0) edge[cycle1] node {$e_1$} (0);
\path (0) edge[cycle2] node[above] {$e_n$} (0);

\end{tikzpicture}\qquad $L_k(E)=L_k(1,n)$

($\sum_{i=1}^n\chi_{e_i}\chi_{e_i^*}$ $\Rightarrow$ $L_k(1,n)^n\cong L_k(1,n)$ as modules.)
\item $E=$\quad
\begin{tikzpicture}[auto,swap]
\tikzstyle{vertex}=[circle,fill=black,minimum size=3pt,inner sep=0pt]
\tikzstyle{edge}=[draw,->]
\tikzstyle{cycle1}=[draw,->,out=130, in=50, loop, distance=40pt]
\tikzstyle{cycle2}=[draw,->,out=100, in=30, loop, distance=40pt]
   
\node[vertex] (0) at (0,0) {};
\node[vertex] (1) at (0,1) {};
\node[vertex] (2) at (1,1.5) {};
\node (3) at (1.75,1.3) {$\ddots$};
\node (4) at (2,0.5) {$\vdots$};
\node[vertex] (5) at (1,-0.5) {};
\node (6) at (1.75,-0.2) {\reflectbox{$\ddots$}};

\path (0) edge[edge] node {} (1);
\path (1) edge[edge] node {} (2);
\path (5) edge[edge] node {} (0);

\end{tikzpicture}\qquad$L_k(E)\cong M_n(k)[\mathbb{Z}]\cong M_n(k[\mathbb{Z}])$

(Laurent polynomials with matrix coefficients or matrices over Laurent polynomials.)
\end{enumerate}

\begin{lemma}
Let $E\hookrightarrow F$ be an admissible inclusion of row-finite (no infinite emitters) graphs
and $k$ be a field. Then the formulas
\begin{gather*}
\chi_v\longmapsto\begin{cases}\chi_v & \text{if } v\in E^0
\\ 0 & \text{if }v\in F^0\setminus E^0\end{cases},
\\
\chi_e\longmapsto\begin{cases}\chi_e & \text{if } e\in E^1
\\ 0 & \text{if }e\in F^1\setminus E^1\end{cases},
\\
\chi_{e^*}\longmapsto\begin{cases}\chi_{e^*} & \text{if } e^*\in (E^1)^*
\\ 0 & \text{if }e^*\in (F^1)^*\setminus(E^1)^*\end{cases},
\end{gather*}
define a homomorphism $L_k(F)\overset{\pi}{\to}L_k(E)$ of algebras yielding the short exact sequence
\[
0\longto I(F^0\setminus E^0)\overset{\text{inclusion map}}{\longrightarrow} L_k(F)\overset{\pi}{\longto}L_k(E)\longto 0,
\]
where $I(F^0\setminus E^0)$ is the ideal of $L_k(E)$ generated by $F^0\setminus E^0$.
\end{lemma}
\begin{corollary}
\[
L_k(E)\cong L_k(F)/I(F^0\setminus E^0)
\]
\end{corollary}
\begin{remark}
If $A\overset{f}{\to} B$ is a surjective homomorphism of algebras and $A$ is unital, then $B$ is also unital
and $f(1_A)=1_B$. Indeed, $\forall\;b\in B$:
\[
f(1_A)b=f(1_A)f(a)=f(a)=b,\qquad bf(1_A)=f(a)f(1_A)=f(a)=b.
\]
\end{remark}
\begin{definition}
Let $A_1\overset{f_1}{\longto}B\overset{f_2}{\longleftarrow}A_2$ be homomorphisms of algebras.
The \underline{pullback algebra} $P(f_1,f_2)$ of $f_1$ and $f_2$ is
\[
\boxed{P(f_1,f_2):=\left\{(x,y)\in A_1\oplus A_2~|~f_1(x)=f_2(y)\right\}.}
\]
Here $A_1\oplus A_2$ is viewed as an algebra with componentwise multiplication.
$P(f_1,f_2)$ is a subalgebra of $A_1\oplus A_2$ because $f_1$ and $f_2$ are algebra homomorphisms.
\end{definition}
\begin{theorem}[\cite{hrt20}]\footnote{Joint work with Sarah Reznikoff.}
Let $F_1$ and $F_2$ be row-finite graphs whose intersection $F_1\cap F_2$ is admissible,
and let $k$ be a field. Furthermore, let 
\[
L_k(F_1)\overset{\pi_1}{\longto}L_k(F_1\cap F_2)\overset{\pi_2}{\longleftarrow}L_k(F_2)
\]
and 
\[
L_k(F_1)\overset{p_1}{\longleftarrow}L_k(F_1\cup F_2)\overset{p_2}{\longto}L_k(F_2)
\]
be the canonical surjections of the preceding lemma. Then the map
\[
L_k(F_1\cup F_2)\ni x\longmapsto (p_1(x),p_2(x))\in L_k(F_1)\oplus L_k(F_2)
\]
corestricts to an \underline{isomorphism} $L_k(F_1\cup F_2)\to P(\pi_1,\pi_2)$ of algebras.
\end{theorem}

\section{Exercises}

\begin{enumerate}

\item[\bf Problem 1]
Construct a graph with 5 edges and no loops such that there are exactly 5 different paths of length~2.
\item[\bf Solution:]
$$\begin{tikzpicture}[auto,swap]
\tikzstyle{vertex}=[circle,fill=black,minimum size=3pt,inner sep=0pt]
\tikzstyle{edge}=[draw,->]
   
\node[vertex] (0) at (0,0) {};
\node[vertex] (1) at (1,0) {};
\node[vertex] (2) at (2,0) {};
\node[vertex] (3) at (3,0) {};
\node[vertex] (4) at (4,0) {};

\path (0) edge[edge] node {} (1);
\path (1) edge[edge] node {} (2);
\path (1) edge[edge,bend left] node {} (2);
\path (2) edge[edge] node {} (3);
\path (3) edge[edge] node {} (4);

\end{tikzpicture}
$$
\bigskip

\item[\bf Problem 2]
Let $E$ be a finite graph whose all vertices emit at least one edge. Prove that there is a loop in~$E$.
\item[\bf Solution:]
If every vertex emits at least one edge, then there exists an infinite path. This in turn means
that the set of all finite paths $FP(E)$ is infinite. However, we proved that $FP(E)$ is finite
if and only if there are no loops in $E$. Hence the claim follows.
\bigskip

\item[\bf Problem 3]
Construct a graph with $7$ edges, no loops, and whose longest path is longer than~3, but for which the number of paths of length $3$
is still maximal.
\item[\bf Solution:]
$$\begin{tikzpicture}[auto,swap]
\tikzstyle{vertex}=[circle,fill=black,minimum size=3pt,inner sep=0pt]
\tikzstyle{edge}=[draw,->]
   
\node[vertex] (0) at (0,0) {};
\node[vertex] (1) at (1,0) {};
\node[vertex] (2) at (2,0) {};
\node[vertex] (3) at (3,0) {};
\node[vertex] (4) at (4,0) {};

\path (0) edge[edge] node {} (1);
\path (1) edge[edge] node {} (2);
\path (1) edge[edge,bend left] node {} (2);
\path (2) edge[edge] node {} (3);
\path (2) edge[edge,bend left] node {} (3);
\path (3) edge[edge] node {} (4);
\path (3) edge[edge,bend left] node {} (4);

\end{tikzpicture}
$$
\bigskip

\item[\bf Problem 4]
Construct all graphs with $4$ edges, no loops, and all vertices in the image of the source map or the  target map,
such that the number of \emph{all} positive-length finite paths is maximal.
\item[\bf Solution:]
$$\begin{tikzpicture}[auto,swap]
\tikzstyle{vertex}=[circle,fill=black,minimum size=3pt,inner sep=0pt]
\tikzstyle{edge}=[draw,->]
   
\node[vertex] (0) at (0,0) {};
\node[vertex] (1) at (1,0) {};
\node[vertex] (2) at (2,0) {};
\node[vertex] (3) at (3,0) {};
\node[vertex] (4) at (4,0) {};

\path (0) edge[edge] node {} (1);
\path (1) edge[edge] node {} (2);
\path (2) edge[edge] node {} (3);
\path (3) edge[edge] node {} (4);

\end{tikzpicture}\qquad
\begin{tikzpicture}[auto,swap]
\tikzstyle{vertex}=[circle,fill=black,minimum size=3pt,inner sep=0pt]
\tikzstyle{edge}=[draw,->]
   
\node[vertex] (0) at (0,0) {};
\node[vertex] (1) at (1,0) {};
\node[vertex] (2) at (2,0) {};
\node[vertex] (3) at (3,0) {};

\path (0) edge[edge] node {} (1);
\path (1) edge[edge] node {} (2);
\path (1) edge[edge,bend left] node {} (2);
\path (2) edge[edge] node {} (3);

\end{tikzpicture}
$$
\bigskip

\item[\bf Problem 5]
For the graph given below, find the number of all paths of length $k\geq 2$.
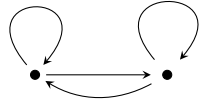
\begin{center}
\begin{tikzpicture}[>=stealth,node distance=50pt,main node/.style={circle,inner sep=2pt},
freccia/.style={->,shorten >=2pt,shorten <=2pt},
ciclo/.style={out=130, in=50, loop, distance=40pt, ->},
circle/.style={out=135, in=45, loop, distance=65pt, ->}]

      \node[main node] (1) {};
      \node (2) [right of=1] {};

      \filldraw (1) circle (0.06) node[below left] {};
      \filldraw (2) circle (0.06) node[below right] {};

      \path[freccia] (1) edge[ciclo] (1);
			\path[freccia] (2) edge[ciclo] (2);

      \path[freccia] (1) edge (2);
       \path[freccia] (2) edge[bend left] (1);
			   
\end{tikzpicture}
\end{center}
\item[\bf Solution:]
The adjacency matrix for the above graph is
$$
A(E)=\begin{bmatrix}1 & 1\\ 1 & 1\end{bmatrix}.
$$
We prove by induction that, for any $k\geq 1$, we have
$$
A(E)^k=\begin{bmatrix}2^{k-1} & 2^{k-1}\\ 2^{k-1} & 2^{k-1}\end{bmatrix}.
$$
Indeed, the result holds for $k=1,$ and the computation
$$
\begin{bmatrix}2^{k-1} & 2^{k-1}\\2^{k-1} & 2^{k-1}\end{bmatrix}
\begin{bmatrix}1 & 1\\ 1 & 1\end{bmatrix}=
\begin{bmatrix}2^{k} & 2^{k}\\2^{k} & 2^{k}\end{bmatrix}
$$
proves the inductive step. Hence, the number of all $k$-paths in $E$ equals 
$$\sum_{i=1}^2\sum_{j=1}^2\left(A(E)^k\right)_{ij}=4\cdot 2^{k-1}=2^{k+1}.$$
\bigskip

\item[\bf Problem 6]
Interpreting the matrix
$$
A(E)=\left( \begin{array}{cccccc}
    0 & a_1 & 0 & \dots & 0 & 0  \\
    0 & 0 & a_2 & & & \\
    0 & 0 & 0 & \ddots & &\vdots\\
    \vdots & & & \ddots&  & a_{n}\\
    0 &  & &\ldots & & 0
   \end{array} \right)
$$
as the adjacency matrix of a certain graph~$E$, prove that
$$
A(E)^n=\left( \begin{array}{cccccc}
    0 & 0 & 0 & \dots & 0 & \prod_{i=1}^{n}a_i  \\
    0 & 0 & 0 & & & \\
    0 & 0 & 0 & \ddots & &\vdots\\
    \vdots & & & \ddots&  & 0\\
    0 &  & &\ldots & & 0
   \end{array} \right).
$$
\item[\bf Solution:]
$A(E)$ can be viewed as the adjacency matrix of the following graph~$E$:
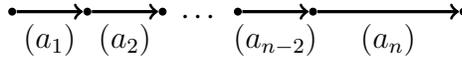
\begin{center}
\begin{tikzpicture}[auto,swap]
\tikzstyle{vertex}=[circle,fill=black,minimum size=3pt,inner sep=0pt]
\tikzstyle{edge}=[draw,very thick,->]
\tikzstyle{thinedge}=[draw,->]
\tikzset{every loop/.style={min distance=20mm,in=130,out=50,looseness=50}}
   
   \node[vertex, above] (0) at (0,0) {};
   \node[vertex, above] (1) at (1,0) {};
   \node[vertex, above] (2) at (2,0) {};
   \node (25) at (2.5,0) {\ldots};
   \node[vertex,above] (7) at (3,0) {};
   \node[vertex,above] (8) at (4,0) {};
   \node[vertex,above] (9) at (6,0) {};
   
   \path (0) edge[edge,below] node {$(a_1)$} (1);
   \path (1) edge[edge,below] node {$(a_2)$} (2); 
   \path (7) edge[edge,below] node {$(a_{n-2})$} (8);
   \path (8) edge[edge,below] node {$(a_{n})$} (9);

\end{tikzpicture}
\end{center}
There are $\prod_{i=1}^{n}a_i$ many paths of length $n$, so  the sum of all entires of $A(E)^n$  equals  $\prod_{i=1}^{n}a_i$. Furthermore, as
all paths of length $n$ start at the first vertex and end at the $n$-th vertex, the only non-zero entry of  $A(E)^n$ is the last entry in the first row.
Hence $A(E)^n$ is of the claimed  form.
\bigskip

\item[\bf Problem 7]
Prove that the adjacency matrix of the graph given below raised to the $5$-th power is zero.
$$
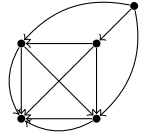
\begin{tikzpicture}[auto,swap]
\tikzstyle{vertex}=[circle,fill=black,minimum size=3pt,inner sep=0pt]
\tikzstyle{edge}=[draw,->]
   
\node[vertex] (0) at (0,0) {};
\node[vertex] (1) at (1,0) {};
\node[vertex] (2) at (0,-1) {};
\node[vertex] (3) at (1,-1) {};
\node[vertex] (4) at (1.5,0.5) {};

\path (4) edge[edge,bend right] node {} (0);
\path (4) edge[edge,bend left] node {} (3);
\path (4) edge[edge] node {} (1);
\path (0) edge[edge] node {} (2);
\path (0) edge[edge,bend right] node {} (2);
\path (0) edge[edge] node {} (3);
\path (1) edge[edge] node {} (0);
\path (1) edge[edge] node {} (2);
\path (1) edge[edge] node {} (3);
\path (3) edge[edge] node {} (2);
\path (3) edge[edge,bend left] node {} (2);

\end{tikzpicture}
$$
\item[\bf Solution:]
Call the above graph~$E$. Since there are no loops in $E$ and $E$ is finite, there exists a longest path.
Any longest path must end in a sink. There is only one sink in $E$, and
one can easily check that its longest path is of length~$4$.
Hence, there are no paths of length $5$, so the adjacency matrix of~$E$ raised to
the $5$-th power is zero.
\bigskip

\item[\bf Problem 8]
Find all hereditary subsets and all saturated subsets for the graph given below.
$$
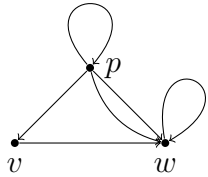
\begin{tikzpicture}[auto,swap]
\tikzstyle{vertex}=[circle,fill=black,minimum size=3pt,inner sep=0pt]
\tikzstyle{edge}=[draw,->]
\tikzstyle{cycle1}=[draw,->,out=130, in=50, loop, distance=40pt]
\tikzstyle{cycle2}=[draw,->,out=100, in=30, loop, distance=40pt]
   
\node[vertex,label=right:$p$] (0) at (0,0) {};
\node[vertex,label=below:$w$] (1) at (1,-1) {};
\node[vertex,label=below:$v$] (2) at (-1,-1) {};

\path (0) edge[cycle1] node {} (0);
\path (0) edge[edge] node {} (1);
\path (0) edge[edge,bend right] node {} (1);
\path (0) edge[edge] node {} (2);
\path (1) edge[cycle2] node {} (1);
\path (2) edge[edge] node {} (1);

\end{tikzpicture}
$$
\item[\bf Solution:]
Call the above graph $E$, and consider all possible subsets of $E^0$.
\begin{enumerate}
\item $\{p\}\subseteq E^0$ is not hereditary because $p$ emits arrows that end not at~$p$.
It is saturated since there are no arrows ending at $p$ which start not at~$p$.
\item $\{v\}\subseteq E^0$ is not hereditary because $v$ emits an arrow that ends not at $v$.
It is saturated since the only arrow ending in $v$ starts at $p$ which emits arrows ending not at~$v$.
\item $\{w\}\subseteq E^0$ is hereditary because $w$ emits only an edge ending at~$w$.
It is not saturated since $v$ emits only one edge, and the edge ends at $w$.
\item $\{p,v\}\subseteq E^0$ is not hereditary because both $p$ and $v$ emit an arrow that ends in $w$.
It is saturated since $w$  emits only one edge ending at~$w$.
\item $\{p,w\}\subseteq E^0$ is not hereditary because $p$ emits an arrow that ends at~$v$.
It is also not saturated since $v$ emits only one edge, and the edge ends at $w$.
\item $\{v,w\}\subseteq E^0$ is hereditary because $v$ emits only an arrow that ends at $w$m, and
$w$ emits only one edge, and the edge ends at~$w$. It is also saturated since there is an edge emitted by $p$ that ends at~$p$.
\item Both $\emptyset$ and $E^0$ are hereditary and saturated.
\end{enumerate}
Conclusion: There are four herditary subsets $\emptyset, \{w\},\{v,w\}, E^0$, and six saturated subsets
$\emptyset,\{p\},\{v\},\{p,v\},\{v,w\},E^0$. 
\bigskip

\item[\bf Problem 9]
Let $(f_0,f_1):E\to F$ a graph homomorphism with both $f_0$ and $f_1$ bijective.
Prove that $(f_0^{-1},f_1^{-1}):F\to E$ is a graph homomorphism.
\item[\bf Solution:]
Since $(f_0,f_1):E\to F$ is a graph homomorphism, we have
$$
s_F\circ f_1=f_0\circ s_E,\qquad t_F\circ f_1=f_0\circ t_E\,.
$$
Composing both of the above equlities with $f^{-1}_0$ on the left and $f^{-1}_1$ on the right yields
$$
f^{-1}_0\circ s_F=s_E\circ f^{-1}_1,\qquad f^{-1}_0\circ t_F=t_E\circ f^{-1}_1\,.
$$
This means that $(f_0^{-1},f_1^{-1}):F\to E$ is a graph homomorphism.
\bigskip

\item[\bf Problem 10]
Consider two graphs $E$ and $F$:
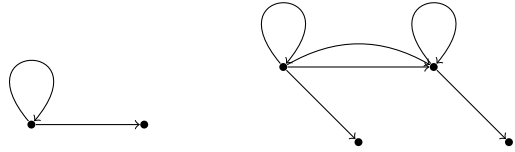
$$
\begin{tikzpicture}[auto,swap]
\tikzstyle{vertex}=[circle,fill=black,minimum size=3pt,inner sep=0pt]
\tikzstyle{cycle1}=[draw,->,out=130, in=50, loop, distance=40pt]
\tikzstyle{edge}=[draw,->]
   
\node[vertex] (0) at (0,0) {};
\node[vertex] (1) at (1.5,0) {};

\path (0) edge[cycle1] node {} (0);
\path (0) edge[edge] node {} (1);

\end{tikzpicture}\qquad 
\begin{tikzpicture}[auto,swap]
\tikzstyle{vertex}=[circle,fill=black,minimum size=3pt,inner sep=0pt]
\tikzstyle{cycle1}=[draw,->,out=130, in=50, loop, distance=40pt]
\tikzstyle{edge}=[draw,->]
   
\node[vertex] (0) at (0,0) {};
\node[vertex] (1) at (2,0) {};
\node[vertex] (2) at (1,-1) {};
\node[vertex] (3) at (3,-1) {};

\path (0) edge[cycle1] node {} (0);
\path (0) edge[edge] node {} (1);
\path (0) edge[edge,bend left] node {} (1);
\path (1) edge[cycle1] node {} (1);
\path (1) edge[edge] node {} (3);
\path (0) edge[edge] node {} (2);

\end{tikzpicture}
$$
\noindent Define two injective graph homomorphisms from $E$ to $F$
such that one of them is an admissible inclusion and the other one is not.
\item[\bf Solution:]
Let us label the vertices and edges in both graphs as follows:
$$
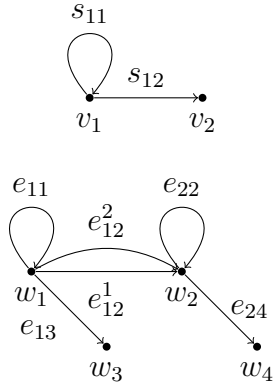
\begin{tikzpicture}[auto,swap]
\tikzstyle{vertex}=[circle,fill=black,minimum size=3pt,inner sep=0pt]
\tikzstyle{cycle1}=[draw,->,out=130, in=50, loop, distance=40pt]
\tikzstyle{edge}=[draw,->]
   
\node[vertex,label=below:$v_1$] (0) at (0,0) {};
\node[vertex, label=below:$v_2$] (1) at (1.5,0) {};

\path (0) edge[cycle1,above] node {$s_{11}$} (0);
\path (0) edge[edge,above] node {$s_{12}$} (1);

\end{tikzpicture}
$$
$$
\begin{tikzpicture}[auto,swap]
\tikzstyle{vertex}=[circle,fill=black,minimum size=3pt,inner sep=0pt]
\tikzstyle{cycle1}=[draw,->,out=130, in=50, loop, distance=40pt]
\tikzstyle{edge}=[draw,->]
   
\node[vertex,label=below:$w_1$] (0) at (0,0) {};
\node[vertex,label=below:$w_2$] (1) at (2,0) {};
\node[vertex,label=below:$w_3$] (2) at (1,-1) {};
\node[vertex,label=below:$w_4$] (3) at (3,-1) {};

\path (0) edge[cycle1,above] node {$e_{11}$} (0);
\path (0) edge[edge,below] node {$e^1_{12}$} (1);
\path (0) edge[edge,bend left,above] node {$e^2_{12}$} (1);
\path (1) edge[cycle1,above] node {$e_{22}$} (1);
\path (1) edge[edge,right] node {$e_{24}$} (3);
\path (0) edge[edge,below left] node {$e_{13}$} (2);
\end{tikzpicture}
$$

First, we check that the inclusion
$$
\iota_1:E\to F,\quad v_1\mapsto w_1,\quad v_2\mapsto w_3,
\quad s_{11}\mapsto e_{11},\quad s_{12}\mapsto e_{13}\,,
$$
is admissible: $F^0\setminus\iota_1(E^0)=F^0\setminus\{w_1,w_3\}=\{w_2,w_4\}$
is both hereditary ($w_4$ is a sink and all paths starting at $w_2$ end at $w_2$ or $w_4$) and saturated
 ($w_3$ is a sink and $w_1$ emits a loop), and
$$
F^1\setminus t^{-1}_F(F^0\setminus\iota_1(E^0))=
t^{-1}_F(\iota_1(E^0))=
\{e_{11},e_{13}\}=\iota_1(E^1)\,.
$$
Next, we consider the inclusion
$$
\iota_2:E\to F,\quad v_1\mapsto w_2,\quad v_2\mapsto w_4,
\quad s_{11}\mapsto e_{22},\quad s_{12}\mapsto e_{24}\,,
$$
which is not admissible because $F^0\setminus\iota_2(E^0)=\{w_1,w_3\}$
is not hereditary ($w_1$ emits  $e^1_{12}$ which ends at $w_2$).
\bigskip

\item[\bf Problem 11]
Prove that the path algebra over a field $k$ of the graph  
$$
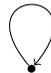
\begin{tikzpicture}[auto,swap]
\tikzstyle{vertex}=[circle,fill=black,minimum size=3pt,inner sep=0pt]
\tikzstyle{cycle1}=[draw,->,out=130, in=50, loop, distance=40pt]
   
\node[vertex] (0) at (0,0) {};

\path (0) edge[cycle1] node {} (0);

\end{tikzpicture}
$$
is isomorphic to the polynomial algebra $k[\mathbb{N}]$.
\item[\bf Solution:]
Call the above graph $E$, denote the vertex in $E$ by $v$ and the loop in $E$  by~$\alpha$ .
By definition, $\{\chi_v,\chi_\alpha,\chi_{\alpha^2},\ldots\}$ is a basis of the path algebra $kE$. Note first that
$\chi_v=1$ because $\chi_v\chi_{\alpha^i}=\chi_{\alpha^i}=\chi_{\alpha^i}\chi_v$ for all $i\in\mathbb{N}$.
Next, the multiplication is given by $\chi_{\alpha^i}\chi_{\alpha^j}=\chi_{\alpha^{i+j}}$.
Much in the same way, $\{1,x,x^2,\ldots\}$ is a basis of $k[\mathbb{N}]$. Here
$$
x^m(n):=\begin{cases}1 & n=m\\0 & n\neq m\end{cases}
$$
and the convolution product of two basis elements reads $x^i\ast x^j=x^{i+j}$.
Hence, the linear bijection determined by
$$
\varphi:k[\mathbb{N}]\longrightarrow kE,\quad 1\longmapsto \chi_v,\quad x^i\longmapsto \chi_{\alpha^i},
\; i\in \mathbb{N}\setminus\{0\},
$$
is an algebra isomorphism.
\bigskip

\item[\bf Problem 12]
Prove that the path algebra over a field $k$ of the graph 
$$
\begin{tikzpicture}[auto,swap]
\tikzstyle{vertex}=[circle,fill=black,minimum size=3pt,inner sep=0pt]
\tikzstyle{edge}=[draw,->]
   
\node[vertex] (0) at (0,0) {};
\node[vertex] (1) at (1,0) {};

\path (0) edge[edge] node {} (1);

\end{tikzpicture}
$$
is isomorphic to the algebra of upper triangular $2\times 2$ matrices over $k$.
\item[\bf Solution:]
The algebra of upper triangular $2\times 2$ matrices over $k$ admits the following basis:
$$
E_{11}:=\begin{bmatrix}
1 & 0\\
0 & 0
\end{bmatrix},\quad
E_{12}:=\begin{bmatrix}
0 & 1\\
0 & 0
\end{bmatrix},\quad
E_{22}:=\begin{bmatrix}
0 & 0\\
0 & 1
\end{bmatrix}.
$$
The multiplication is given by
\begin{center}
\begin{tabular}{c|ccc}
$\cdot$ & $E_{11}$ & $E_{12}$ & $E_{22}$\\
\hline
$E_{11}$ & $E_{11}$ & $E_{12}$ & $0$ \\
$E_{12}$ & $0$ & $0$ & $E_{12}$ \\
$E_{22}$ & $0$ & $0$ & $E_{22}$ 
\end{tabular}
\end{center}
Next, call the above graph $E$ and denote by $v$ its left vertex, by $e$ its edge, and by $w$ its right vertex. Then,
 by definition, $\{\chi_v, \chi_e, \chi_w\}$ is a basis of~$kE$. The multiplication is given by
\begin{center}
\begin{tabular}{c|ccc}
$\cdot$ & $\chi_v$ & $\chi_e$ & $\chi_w$\\
\hline
$\chi_v$ & $\chi_v$ & $\chi_e$ & $0$ \\
$\chi_e$ & $0$ & $0$ & $\chi_e$ \\
$\chi_w$ & $0$ & $0$ & $\chi_w$ 
\end{tabular}
\end{center}
Hence the linear map determined by
$$
E_{11}\mapsto \chi_v,\quad E_{12}\mapsto \chi_e,\quad E_{22}\mapsto \chi_w,
$$
is an algebra isomorphism.
\bigskip

\item[\bf Problem 13]
Let $E=(E^0,E^1,s,t)$ be a graph and $k$ be a field. Prove that the path algebra $kE$ is unital if and only if $E^0$ is finite.
\item[\bf Solution:]
Assume that $E^0$ is finite. Then
$$
1_{kE}=\sum_{v\in E^0}\chi_v.
$$
Indeed, for any $p\in FP(E)$, we have
$$
\chi_p\left(\sum_{v\in E^0}\chi_v\right)=\sum_{v\in E^0}\chi_p\chi_v=\sum_{v\in E^0}\chi_{pv}=\chi_{pt(p)}=\chi_p\,,
$$
$$
\left(\sum_{v\in E^0}\chi_v\right)\chi_p=\sum_{v\in E^0}\chi_v\chi_p=\sum_{v\in E^0}\chi_{vp}=\chi_{s(p)p}=\chi_p\,.
$$
Assume now that $kE$ is unital. Then $1_{kE}$ can be expressed as a finite linear combination of some basis 
elements: 
$$
1_{kE}=\sum_{i=1}^n\lambda_i \chi_{p_i}\,,
$$
where each $p_i$ is a path in $E$. Since $1_{kE}\chi_v=\chi_v\neq 0$ for all $v\in E^0$, we infer that
that $E^0\subseteq\{t(p_i)\}_{i=1}^n$. Consequently, $E^0$ is finite because any subset of a finite set is finite.
\bigskip

\item[\bf Problem 14]
Let $E=(E^0,E^1,s,t)$ be a graph and let $k$ be a field. Prove that the path algebra $kE$ is commutative if and only if
$E^1=\emptyset$ or each edge is a loop starting/ending at a different vertex.
\item[\bf Solution:]
Assume that $E^1=\emptyset$. Then, for all $v,w\in E^0$, $v\neq w$, $\chi_v\chi_w=0=\chi_w\chi_v$, so
$kE$ is commutative. Now let $E^1$ consist only of loops starting at different vertices.
Then, for any $p,q\in FP(E)$, with $t(p)\neq t(q)$, we have
$
\chi_p\chi_q=0=\chi_q\chi_p
$.
If $t(p)= t(q)$, then 
$
\chi_p\chi_q=\chi_{pq}=\chi_{qp}=\chi_q\chi_p
$.
Hence $kE$ is commutative.
To prove the opposite implication, we need to negate the following statement:
$$
E^1=\emptyset\text{ or } 
\big(E^1\neq\emptyset\text{ and }\forall\;e\in E^1: s(e)=t(e)\text{ and }
(e\neq f \Rightarrow s(e)\neq s(f))\big).
$$
The negation reads
$$
\exists\; e\in E^1: s(e)\neq t(e)\text{ or }(e\neq f\text{ and }s(e)=s(f)).
$$
First assume that there is an edge $e$ which is not a loop. Then
$$
\chi_{s(e)}\chi_e=\chi_e\neq 0=\chi_e\chi_{s(e)}\,,
$$
so $kE$ is not commutative. Next, if there are two different edges $e$ and $f$ such that they are loops and
$s(e)=s(f)$, then $\chi_e\chi_f\neq \chi_f\chi_e$ because they are two different paths. Hence, again, $kE$ is noncommutative.
\bigskip

\item[\bf Problem 15]
Prove that the Leavitt path algebra over a field $k$ of the graph  
$$
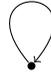
\begin{tikzpicture}[auto,swap]
\tikzstyle{vertex}=[circle,fill=black,minimum size=3pt,inner sep=0pt]
\tikzstyle{cycle1}=[draw,->,out=130, in=50, loop, distance=40pt]
   
\node[vertex] (0) at (0,0) {};

\path (0) edge[cycle1] node {} (0);

\end{tikzpicture}
$$
is isomorphic to the Laurent polynomial algebra $k[\mathbb{Z}]$.
\item[\bf Solution:]
Call the above graph $E$ and denote by $v$ the vertex of $E$ and by $\alpha$ the edge of $E$.
Then $[\chi_v]=1$ in the Leavitt path algebra $L_k(E)$, and $L_k(E)$ is spanned by  
$B:=\{1,[\chi_\alpha],[\chi_{\alpha^2}],[\chi_{\alpha^*}],[\chi_{(\alpha^*)^2}],\ldots\}$
because $[\chi_{\alpha^*\alpha}]=1=[\chi_{\alpha\alpha^*}]$.
The set $B$ is also linearly independent by Corollary~1.5.12 in \emph{Leavitt Path Algebras}, so it is 
a basis of $L_k(E)$. The multiplication of elements of $B$ is given by 
$$
[\chi_{\alpha^i}][\chi_{\alpha^j}]=[\chi_{\alpha^{i+j}}],\quad [\chi_{(\alpha^*)^i}][\chi_{(\alpha^*)^j}]=[\chi_{(\alpha^*)^{i+j}}],
$$
$$
[\chi_{\alpha^i}][\chi_{(\alpha^*)^j}]=[\chi_{(\alpha^*)^j}][\chi_{\alpha^i}]=
\begin{cases}[\chi_{\alpha^{i-j}}]&\text{for $i>j$}\\1 & \text{for $i=j$}\\
[\chi_{(\alpha^*)^{j-i}}] & \text{for $j>i$}\end{cases}
$$
for all $i,j\in\mathbb{N}$, with the convention that $[\chi_{\alpha^0}]=1=[\chi_{(\alpha^*)^0}]$.
Next, recall that $\{x^i\}_{i\in\mathbb{Z}}$ is a basis of $k[\mathbb{Z}]$, where
$$
x^j(n):=\begin{cases}1 & n=j\\0 & n\neq j\end{cases}.
$$
The convolution product for the basis elements reads 
$x^i\ast x^j=x^{i+j}$ for all $i,j\in\mathbb{N}$.
Hence, the linear map determined by
$$
x^i\longmapsto \chi_{\alpha^i} \text{ for }i\geq 0, 
\quad x^{i}\longmapsto \chi_{(\alpha^*)^{-i}}\text{ for } i\leq 0,
$$
is an algebra isomorphism.
\bigskip

\item[\bf Problem 16]
Prove that the Leavitt path algebra over a field $k$ of the graph 
$$
\begin{tikzpicture}[auto,swap]
\tikzstyle{vertex}=[circle,fill=black,minimum size=3pt,inner sep=0pt]
\tikzstyle{edge}=[draw,->]
   
\node[vertex] (0) at (0,0) {};
\node[vertex] (1) at (1,0) {};

\path (0) edge[edge] node {} (1);

\end{tikzpicture}
$$
is isomorphic to the algebra of $2\times 2$ matrices over $k$.
\item[\bf Solution:]
The algebra of $2\times 2$ matrices over $k$ admits the following basis:
$$
E_{11}:=\begin{bmatrix}
1 & 0\\
0 & 0
\end{bmatrix},\quad
E_{12}:=\begin{bmatrix}
0 & 1\\
0 & 0
\end{bmatrix},\quad
E_{21}:=\begin{bmatrix}
0 & 0\\
1 & 0
\end{bmatrix},\quad
E_{22}:=\begin{bmatrix}
0 & 0\\
0 & 1
\end{bmatrix}.
$$
The multiplication is given by
\begin{center}
\begin{tabular}{c|cccc}
$\cdot$ & $E_{11}$ & $E_{12}$ & $E_{21}$ & $E_{22}$\\
\hline
$E_{11}$ & $E_{11}$ & $E_{12}$ & $0$ & $0$ \\
$E_{12}$ & $0$ & $0$ & $E_{11}$ & $E_{12}$ \\
$E_{21}$ & $E_{21}$ & $E_{22}$ & $0$ & $0$\\
$E_{22}$ & $0$ & $0$ & $E_{21}$ & $E_{22}$
\end{tabular}
\end{center}
Next, call the above graph $E$ and denote by $v$ its left vertex, by $e$ its edge, and by $w$ its right vertex. Then
$B:=\{[\chi_v], [\chi_e], [\chi_{e^*}], [\chi_w]\}$ is a basis of the Leavitt path algebra~$L_k(E)$
by Corollary~1.5.12 in \emph{Leavitt Path Algebras}. The multiplication of elements of $B$ is given by 
\begin{center}
\begin{tabular}{c|cccc}
$\cdot$ & $[\chi_v]$ & $[\chi_e]$ & $[\chi_{e^*}]$ & $[\chi_w]$\\
\hline
$[\chi_v]$ & $[\chi_v]$ & $[\chi_e]$ & $0$ & $0$ \\
$[\chi_e]$ & $0$ & $0$ & $[\chi_v]$ & $[\chi_e]$ \\
$[\chi_{e^*}]$ & $[\chi_{e^*}]$ & $[\chi_w]$ & $0$ & $0$ \\
$[\chi_w]$ & $0$ & $0$ & $[\chi_{e^*}]$ & $[\chi_w]$
\end{tabular}
\end{center}
Hence the linear map determined by
$$
E_{11}\mapsto [\chi_v],\quad E_{12}\mapsto [\chi_e],
\quad E_{21}\mapsto [\chi_{e^*}],\quad E_{22}\mapsto [\chi_w],
$$
is an algebra isomorphism.
\bigskip

\item[\bf Problem 17]
Let $k$ be a field. Up to isomorphism, find all $6$-dimensional path  $k$-algebras  of connected graphs.
\bigskip
\item[\bf Solution:]
Since for a graph $E$ the basis of the path algebra $kE$ consists of all finite paths,
we need two find connected graphs such that the number of all their paths, including the $0$-paths (vertices), 
equals $6$.
Consider a graph with:
\begin{enumerate}
\item $1$ vertex. Then, if $E^1\neq\emptyset$,
there are no finite-dimensional path algebras for such graphs.
\item $2$ vertices. Then, to avoid creating a loop, 
the only possibility is to have $4$ edges between these vertices arranged like this:
$$
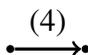
\begin{tikzpicture}[auto,swap]
\tikzstyle{vertex}=[circle,fill=black,minimum size=3pt,inner sep=0pt]
\tikzstyle{edge}=[draw,->, very thick]
\tikzstyle{cycle}=[draw,->,loop,out=130,in=50,distance=30pt]
   
\node[vertex] (0) at (0,0) {};
\node[vertex] (1) at (1,0) {};

\path (0) edge[edge] node[above] {(4)} (1);

\end{tikzpicture}
$$
\item $3$ vertices. Then we have the following graphs:
$$
\begin{tikzpicture}[auto,swap]
\tikzstyle{vertex}=[circle,fill=black,minimum size=3pt,inner sep=0pt]
\tikzstyle{edge}=[draw,->]
\tikzstyle{cycle}=[draw,->,loop,out=130,in=50,distance=30pt]
   
\node[vertex] (0) at (0,0) {};
\node[vertex] (1) at (1,0) {};
\node[vertex] (2) at (2,0) {};

\path (0) edge[edge] node {} (1);
\path (1) edge[edge] node {} (2);

\end{tikzpicture}
\qquad
\begin{tikzpicture}[auto,swap]
\tikzstyle{vertex}=[circle,fill=black,minimum size=3pt,inner sep=0pt]
\tikzstyle{edge}=[draw,->]
\tikzstyle{cycle}=[draw,->,loop,out=130,in=50,distance=30pt]
   
\node[vertex] (0) at (0,0) {};
\node[vertex] (1) at (1,0) {};
\node[vertex] (2) at (2,0) {};

\path (0) edge[edge,bend left] node {} (1);
\path (0) edge[edge] node {} (1);
\path (2) edge[edge] node {} (1);

\end{tikzpicture}
\qquad
\begin{tikzpicture}[auto,swap]
\tikzstyle{vertex}=[circle,fill=black,minimum size=3pt,inner sep=0pt]
\tikzstyle{edge}=[draw,->]
\tikzstyle{cycle}=[draw,->,loop,out=130,in=50,distance=30pt]
   
\node[vertex] (0) at (0,0) {};
\node[vertex] (1) at (1,0) {};
\node[vertex] (2) at (2,0) {};

\path (1) edge[edge,bend right] node {} (0);
\path (1) edge[edge] node {} (0);
\path (1) edge[edge] node {} (2);

\end{tikzpicture}
$$
\item $4$ or more vertices. If there are four or more vertices, then one needs more than two edges to make it connected. Hence, there no $6$-dimensional path algebra for a
 connected graph with $4$ or more vertices.
\end{enumerate}
\bigskip

\item[\bf Problem 18]
Compute the number of all paths of length two for the following graph with 11 edges:
$$
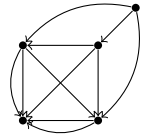
\begin{tikzpicture}[auto,swap]
\tikzstyle{vertex}=[circle,fill=black,minimum size=3pt,inner sep=0pt]
\tikzstyle{edge}=[draw,->]
   
\node[vertex] (0) at (0,0) {};
\node[vertex] (1) at (1,0) {};
\node[vertex] (2) at (0,-1) {};
\node[vertex] (3) at (1,-1) {};
\node[vertex] (4) at (1.5,0.5) {};

\path (4) edge[edge,bend right] node {} (0);
\path (4) edge[edge,bend left] node {} (3);
\path (4) edge[edge] node {} (1);
\path (0) edge[edge] node {} (2);
\path (0) edge[edge,bend right] node {} (2);
\path (0) edge[edge] node {} (3);
\path (1) edge[edge] node {} (0);
\path (1) edge[edge] node {} (2);
\path (1) edge[edge] node {} (3);
\path (3) edge[edge] node {} (2);
\path (3) edge[edge,bend left] node {} (2);

\end{tikzpicture}
$$
Is this the maximal number of paths of length two that one can obtain 
for a graph with 11 edges and no loops?
If not, find a graph that maximizes this number.
\bigskip
\item[\bf Solution:]
Call the above graph $E$. It has the following adjacency matrix:
$$
A(E)=
\begin{bmatrix}
0 & 1 & 1 & 1 & 0\\
0 & 0 & 1 & 1 & 1\\
0 & 0 & 0 & 1 & 2\\
0 & 0 & 0 & 0 & 2\\
0 & 0 & 0 & 0 & 0
\end{bmatrix}.
$$
To find out the number of all paths of length $2$, we need to compute
$$
A(E)^2=
\begin{bmatrix}
0 & 1 & 1 & 1 & 0\\
0 & 0 & 1 & 1 & 1\\
0 & 0 & 0 & 1 & 2\\
0 & 0 & 0 & 0 & 2\\
0 & 0 & 0 & 0 & 0
\end{bmatrix}^2=
\begin{bmatrix}
0 & 0 & 1 & 2 & 5\\
0 & 0 & 0 & 1 & 4\\
0 & 0 & 0 & 0 & 2\\
0 & 0 & 0 & 0 & 0\\
0 & 0 & 0 & 0 & 0
\end{bmatrix}.
$$
Hence, there are $1+2+5+1+4+2=15$ paths of length two.
The maximum number of $2$-paths for a graph with no loops and $11$
edges is $(5+1)^15^{(2-1)}=30$. A graph maximizing this number
is 
$$
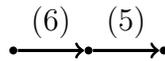
\begin{tikzpicture}[auto,swap]
\tikzstyle{vertex}=[circle,fill=black,minimum size=3pt,inner sep=0pt]
\tikzstyle{edge}=[draw,->,very thick]
   
\node[vertex] (0) at (0,0) {};
\node[vertex] (1) at (1,0) {};
\node[vertex] (2) at (2,0) {};

\path (0) edge[edge] node[above] {$(6)$} (1);
\path (1) edge[edge] node[above] {$(5)$} (2);

\end{tikzpicture}
$$
\bigskip

\item[\bf Problem 19]
Let $E$ be the following graph:
$$
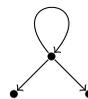
\begin{tikzpicture}[auto,swap]
\tikzstyle{vertex}=[circle,fill=black,minimum size=3pt,inner sep=0pt]
\tikzstyle{edge}=[draw,->]
\tikzstyle{cycle}=[draw,->,loop,out=130,in=50,distance=30pt]
   
\node[vertex] (0) at (0,0) {};
\node[vertex] (1) at (-0.5,-0.5) {};
\node[vertex] (2) at (0.5,-0.5) {};

\path (0) edge[cycle] node {} (0);
\path (0) edge[edge] node {} (1);
\path (0) edge[edge] node {} (2);

\end{tikzpicture}
$$
Find all admissible subgraphs of $E$ (i.e.\ all subgraphs of $E$ whose inclusion in $E$ is admissible) with proof.
\bigskip
\item[\bf Solution:]
We label vertices as follows:
$$
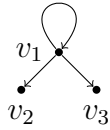
\begin{tikzpicture}[auto,swap]
\tikzstyle{vertex}=[circle,fill=black,minimum size=3pt,inner sep=0pt]
\tikzstyle{edge}=[draw,->]
\tikzstyle{cycle}=[draw,->,loop,out=130,in=50,distance=30pt]
   
\node[vertex, label=left:$v_1$] (0) at (0,0) {};
\node[vertex, label=below:$v_2$] (1) at (-0.5,-0.5) {};
\node[vertex, label=below:$v_3$] (2) at (0.5,-0.5) {};

\path (0) edge[cycle] node {} (0);
\path (0) edge[edge] node {} (1);
\path (0) edge[edge] node {} (2);
\end{tikzpicture}
$$
Of course, the empty subgraph and the whole graph are admissible.
It remains to consider all non-empty proper subsets of $E^0$:
\begin{gather*}
\{v_i\},\;i=1,2,3,\;\{v_1,v_2\},\;\{v_1,v_3\},\;\{v_2,v_3\}.
\end{gather*}
Out of these 6 subsets only the following 3 subsets are hereditary and saturated:
$$
\{v_2\},\;\{v_3\},\;\{v_2,v_3\}.
$$
Now, every admissible subgraph yields a hereditary and saturated subset of missing veritices, and given a hereditary saturated subset $H\subseteq E^0$, 
there is only one way to obtain an admissible subgraph:
$F^0=E^0\setminus H$ and $F^1=t^{-1}_E(E^0\setminus H)$.
Hence we have only the following 3 admissible non-empty proper subgraphs of~$E$:
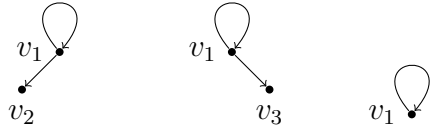
$$
\begin{tikzpicture}[auto,swap]
\tikzstyle{vertex}=[circle,fill=black,minimum size=3pt,inner sep=0pt]
\tikzstyle{edge}=[draw,->]
\tikzstyle{cycle}=[draw,->,loop,out=130,in=50,distance=30pt]
   
\node[vertex, label=left:$v_1$] (0) at (0,0) {};
\node[vertex, label=below:$v_2$] (1) at (-0.5,-0.5) {};

\path (0) edge[cycle] node {} (0);
\path (0) edge[edge] node {} (1);
\end{tikzpicture}
\qquad
\begin{tikzpicture}[auto,swap]
\tikzstyle{vertex}=[circle,fill=black,minimum size=3pt,inner sep=0pt]
\tikzstyle{edge}=[draw,->]
\tikzstyle{cycle}=[draw,->,loop,out=130,in=50,distance=30pt]
   
\node[vertex, label=left:$v_1$] (0) at (0,0) {};
\node[vertex, label=below:$v_3$] (2) at (0.5,-0.5) {};

\path (0) edge[cycle] node {} (0);
\path (0) edge[edge] node {} (2);
\end{tikzpicture}
\qquad
\begin{tikzpicture}[auto,swap]
\tikzstyle{vertex}=[circle,fill=black,minimum size=3pt,inner sep=0pt]
\tikzstyle{edge}=[draw,->]
\tikzstyle{cycle}=[draw,->,loop,out=130,in=50,distance=30pt]
   
\node[vertex, label=left:$v_1$] (0) at (0,0) {};

\path (0) edge[cycle] node[above] {} (0);

\end{tikzpicture}
$$
\bigskip

\item[\bf Problem 20]
Let $k$ be a field and let $E=(E^0,E^1,s,t)$ be a non-empty connected graph. Show that the path algebra $kE$ is 
commutative if and only if 
$$
|E^0|=1\text{ and } |E^1|\leq 1.
$$
\bigskip

\item[\bf Solution:]
Assume that $E^1=\emptyset$ and $|E^0|=1$. Then the graph $E$ is connected and its path algebra $kE\cong k$ is commutative.
Assume next that $|E^1|=1$ and $|E^0|=1$. Then the graph $E$ consists of one vertex and one loop-edge attached to it, 
so it is connected and its path algebra $kE\cong k[\mathbb{N}]$ is commutative.
Suppose now that $|E^0|>1$ and the graph $E$ is connected. Then there exists an edge $e$ that is not a loop. It follows that $kE$ is noncommutative because
$$
\chi_{s(e)}\chi_e=\chi_e\neq 0=\chi_e\chi_{s(e)}\,.
$$
Suppose next that $|E^1|>1$ and the graph $E$ is connected. If there is an edge that is not a loop, then we already know that $kE$ is noncommutative.
If all edges are loops and $E$ is connected, then there is only one vertex Hence there are at least two different loop-edges $\alpha$ and $\beta$ starting from the
same vertex. Consequently, $kE$ is noncommutative because $\chi_\alpha\chi_\beta=\chi_{\alpha\beta}\neq\chi_{\beta\alpha}=\chi_\beta\chi_\alpha$.

\bigskip

\item[\bf Problem 21]
Let $k$ be a field and $E$ be the following graph:
$$
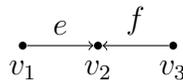
\begin{tikzpicture}[auto,swap]
\tikzstyle{vertex}=[circle,fill=black,minimum size=3pt,inner sep=0pt]
\tikzstyle{edge}=[draw,->]
   
\node[vertex,label=below:$v_1$] (0) at (0,0) {};
\node[vertex,label=below:$v_2$] (1) at (1,0) {};
\node[vertex,label=below:$v_3$] (2) at (2,0) {};

\path (0) edge[edge] node[above] {$e$} (1);
\path (2) edge[edge] node[above] {$f$} (1);
\end{tikzpicture}
$$
Show that any element in the Leavitt path algebra $L_k(E)$ is a linear combination~of 
$$
[\chi_{v_1}]\,,\quad [\chi_{v_2}]\,,\quad [\chi_{v_3}]\,,\quad  [\chi_{e}],\quad  [\chi_{f}],\quad  [\chi_{e^*}],\quad  [\chi_{f^*}],\quad  [\chi_{ef^*}],
\quad[\chi_{fe^*}]\quad\in\quad L_k(E).
$$
\bigskip

\item[\bf Solution:]
Since all elements of $L_k(E)$ corresponding to paths  of length at most one in the extended graph $\bar{E}$ are already 
listed, it suffices to check that, if $p\in FP_n(\bar{E})$ and $n>1$, then $[\chi_p]$ is a linear combination of the above elements.
If $p$ contains a subpath $x^*y$, where $x$ and $y$ are edges, then $[\chi_p]=0$, if $x\neq y$, and $[\chi_p]=[\chi_q]$, if $x=y$ and $q$ is a path obtained from $p$ 
by removing $x^*x$. Consequently, the only elements of $L_k(E)$ corresponding to paths of length at least two are of the form $[\chi_{e_1\cdots e_nf_m^*\cdots f_1^*}]$,
where $e_1\cdots e_n$ and $f_1\cdots f_m$ are paths in~$E$. As the only such paths are $e$ and $f$, the only elements of $L_k(E)$ corresponding 
to paths of length at least two are
$$
[\chi_{ef^*}],\quad[\chi_{fe^*}],\quad [\chi_{ee^*}]=[\chi_{v_1}],\quad[\chi_{ff^*}]=[\chi_{v_3}].
$$
\bigskip

\item[\bf Problem 22]
Let $k$ be a field. Up to isomorphism, find all $5$-dimensional path algebras over $k$.
\bigskip
\item[\bf Solution:]
Since for a graph $E$ the basis of the path algebra $kE$ consists of all finite paths,
we need to find all graphs such that the number of all their paths, including the $0$-paths (vertices), 
equals~$5$.
Consider a graph with:
\begin{enumerate}
\item $0$ edges. The only possibility is to have $5$ vertices.
$$
\begin{tikzpicture}[auto,swap]
\tikzstyle{vertex}=[circle,fill=black,minimum size=3pt,inner sep=0pt]
\tikzstyle{edge}=[draw,->]
\tikzstyle{cycle}=[draw,->,loop,out=130,in=50,distance=30pt]
   
\node[vertex] (0) at (0,0) {};
\node[vertex] (1) at (0.5,0) {};
\node[vertex] (2) at (1,0) {};
\node[vertex] (3) at (1.5,0) {};
\node[vertex] (4) at (2,0) {};

\end{tikzpicture}
$$
\item $1$ edge. We need to have at least two vertices, because otherwise the edge would would be a loop.
Then, the only possibility is to have two more disconnected vertices.
$$
\begin{tikzpicture}[auto,swap]
\tikzstyle{vertex}=[circle,fill=black,minimum size=3pt,inner sep=0pt]
\tikzstyle{edge}=[draw,->]
\tikzstyle{cycle}=[draw,->,loop,out=130,in=50,distance=30pt]
   
\node[vertex] (0) at (0,0) {};
\node[vertex] (1) at (0.5,0) {};
\node[vertex] (2) at (1,0) {};
\node[vertex] (3) at (1.5,0) {};

\path (0) edge[edge] node {} (1);

\end{tikzpicture}
$$
\item $2$ edges. Again, we need to have at least two vertices so that both edges are not loops.
If both edges start at the same vertex, they both can end at some other vertex or end at two different vertices.
If they both end at the same vertex, we need to add a disconnected vertex. If they start at two different vertices,
they need to end at the same vertex.
$$
\begin{tikzpicture}[auto,swap]
\tikzstyle{vertex}=[circle,fill=black,minimum size=3pt,inner sep=0pt]
\tikzstyle{edge}=[draw,->]
\tikzstyle{cycle}=[draw,->,loop,out=130,in=50,distance=30pt]
   
\node[vertex] (0) at (0,0) {};
\node[vertex] (1) at (0.5,0) {};
\node[vertex] (2) at (1,0) {};

\path (0) edge[edge,bend left] node {} (1);
\path (0) edge[edge] node {} (1);

\end{tikzpicture}
\qquad
\begin{tikzpicture}[auto,swap]
\tikzstyle{vertex}=[circle,fill=black,minimum size=3pt,inner sep=0pt]
\tikzstyle{edge}=[draw,->]
\tikzstyle{cycle}=[draw,->,loop,out=130,in=50,distance=30pt]
   
\node[vertex] (0) at (0,0) {};
\node[vertex] (1) at (0.5,0) {};
\node[vertex] (2) at (1,0) {};

\path (1) edge[edge] node {} (0);
\path (1) edge[edge] node {} (2);

\end{tikzpicture}
\qquad
\begin{tikzpicture}[auto,swap]
\tikzstyle{vertex}=[circle,fill=black,minimum size=3pt,inner sep=0pt]
\tikzstyle{edge}=[draw,->]
\tikzstyle{cycle}=[draw,->,loop,out=130,in=50,distance=30pt]
   
\node[vertex] (0) at (0,0) {};
\node[vertex] (1) at (0.5,0) {};
\node[vertex] (2) at (1,0) {};

\path (0) edge[edge] node {} (1);
\path (2) edge[edge] node {} (1);

\end{tikzpicture}
$$
\item $3$ edges. The only possibility is the following:
$$
\begin{tikzpicture}[auto,swap]
\tikzstyle{vertex}=[circle,fill=black,minimum size=3pt,inner sep=0pt]
\tikzstyle{edge}=[draw,->]
\tikzstyle{cycle}=[draw,->,loop,out=130,in=50,distance=30pt]
   
\node[vertex] (0) at (0,0) {};
\node[vertex] (1) at (0.5,0) {};

\path (0) edge[edge,bend left] node {} (1);
\path (0) edge[edge] node {} (1);
\path (0) edge[edge,bend right] node {} (1);
\end{tikzpicture}
$$
\item 
$4$ or more edges. This is impossible because then there would be only one vertex, whence edges would be loops.
\end{enumerate}
\bigskip

\item[\bf Problem 23]
Compute the number of all paths of a fixed length $k>1$ for the following graph:
$$
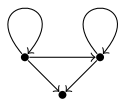
\begin{tikzpicture}[auto,swap]
\tikzstyle{vertex}=[circle,fill=black,minimum size=3pt,inner sep=0pt]
\tikzstyle{edge}=[draw,->]
\tikzstyle{cycle}=[draw,->,loop,out=130,in=50,distance=30pt]
   
\node[vertex] (0) at (-0.5,0) {};
\node[vertex] (00) at (0.5,0) {};
\node[vertex] (1) at (0,-0.5) {};

\path (0) edge[cycle] node {} (0);
\path (00) edge[cycle] node {} (00);
\path (0) edge[edge] node {} (00);
\path (0) edge[edge] node {} (1);
\path (00) edge[edge] node {} (1);

\end{tikzpicture}
$$
\item[\bf Solution:]
Call the above graph~$E$. It has the following adjacency matrix:
$$
A(E)=
\begin{bmatrix}
1 & 1 & 1\\
0 & 1 & 1\\
0 & 0 & 0
\end{bmatrix}.
$$
To obtain the number of $k$-paths we need to raise $A(E)$ to the $k$-th power.
We claim that
$$
A(E)^k=
\begin{bmatrix}
1 & k & k\\
0 & 1 & 1\\
0 & 0 & 0
\end{bmatrix},
$$
and prove it by induction. For $k=1$ the equality is satisfied and
$$
\begin{bmatrix}
1 & k & k\\
0 & 1 & 1\\
0 & 0 & 0
\end{bmatrix}
\begin{bmatrix}
1 & 1 & 1\\
0 & 1 & 1\\
0 & 0 & 0
\end{bmatrix}=
\begin{bmatrix}
1 & k+1 & k+1\\
0 & 1 & 1\\
0 & 0 & 0
\end{bmatrix}
$$
proves the inductive step. Hence, there are $2k+3$ many $k$-paths.
\bigskip

\item[\bf Problem 24]
Let $E$ be the following graph:
$$
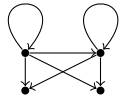
\begin{tikzpicture}[auto,swap]
\tikzstyle{vertex}=[circle,fill=black,minimum size=3pt,inner sep=0pt]
\tikzstyle{edge}=[draw,->]
\tikzstyle{cycle}=[draw,->,loop,out=130,in=50,distance=30pt]
   
\node[vertex] (0) at (-0.5,0) {};
\node[vertex] (00) at (0.5,0) {};
\node[vertex] (1) at (-0.5,-0.5) {};
\node[vertex] (2) at (0.5,-0.5) {};

\path (0) edge[cycle] node {} (0);
\path (00) edge[cycle] node {} (00);
\path (0) edge[edge] node {} (00);
\path (0) edge[edge] node {} (1);
\path (0) edge[edge] node {} (2);
\path (00) edge[edge] node {} (1);
\path (00) edge[edge] node {} (2);

\end{tikzpicture}
$$
Find all admissible subgraphs of $E$ (i.e.\ all subgraphs of $E$ whose inclusion in $E$ is admissible) with proof. 
\bigskip

\item[\bf Solution:]
We label vertices as follows:
$$
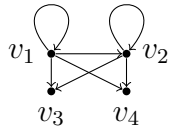
\begin{tikzpicture}[auto,swap]
\tikzstyle{vertex}=[circle,fill=black,minimum size=3pt,inner sep=0pt]
\tikzstyle{edge}=[draw,->]
\tikzstyle{cycle}=[draw,->,loop,out=130,in=50,distance=30pt]
   
\node[vertex, label=left:$v_1$] (0) at (-0.5,0) {};
\node[vertex, label=right:$v_2$] (00) at (0.5,0) {};
\node[vertex, label=below:$v_3$] (1) at (-0.5,-0.5) {};
\node[vertex, label=below:$v_4$] (2) at (0.5,-0.5) {};

\path (0) edge[cycle] node {} (0);
\path (00) edge[cycle] node {} (00);
\path (0) edge[edge] node {} (00);
\path (0) edge[edge] node {} (1);
\path (0) edge[edge] node {} (2);
\path (00) edge[edge] node {} (1);
\path (00) edge[edge] node {} (2);

\end{tikzpicture}
$$
Of course, the empty subgraph and the whole graph are admissible.
It remains to consider all non-empty proper subsets of $E^0$:
\begin{gather*}
\{v_i\},\;i=1,2,3,4,\;\{v_1,v_2\},\;\{v_1,v_3\},\;\{v_1,v_4\},\;\;\{v_2,v_3\},\;\{v_2,v_4\},\;\{v_3,v_4\},\;\\
\{v_1,v_2,v_3\},\;\{v_1,v_2,v_4\},\;\{v_1,v_3,v_4\},\;\{v_2,v_3,v_4\}.
\end{gather*}
Out of these 14 subsets only the following 4 subsets are hereditary and saturated:
$$
\{v_3\},\;\{v_4\},\;\{v_3,v_4\},\;\{v_2,v_3,v_4\}.
$$
Now, every admissible subgraph yields a hereditary and saturated subset of missing veritices, and given a hereditary saturated subset $H\subseteq E^0$, 
there is only one way to obtain an admissible subgraph:
$F^0=E^0\setminus H$ and $F^1=t^{-1}_E(E^0\setminus H)$.
Hence we have only the following 4 admissible non-empty proper subgraphs of~$E$:
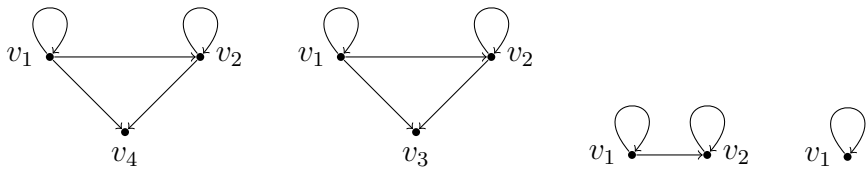
$$
\begin{tikzpicture}[auto,swap]
\tikzstyle{vertex}=[circle,fill=black,minimum size=3pt,inner sep=0pt]
\tikzstyle{edge}=[draw,->]
\tikzstyle{cycle}=[draw,->,loop,out=130,in=50,distance=30pt]
   
\node[vertex, label=left:$v_1$] (0) at (-1,0) {};
\node[vertex, label=right:$v_2$] (00) at (1,0) {};
\node[vertex, label=below:$v_4$] (1) at (0,-1) {};

\path (0) edge[cycle] node[above] {} (0);
\path (00) edge[cycle] node[above] {} (00);
\path (0) edge[edge] node[above] {} (00);
\path (0) edge[edge] node[left] {} (1);
\path (00) edge[edge] node[right] {} (1);
\end{tikzpicture}
\quad
\begin{tikzpicture}[auto,swap]
\tikzstyle{vertex}=[circle,fill=black,minimum size=3pt,inner sep=0pt]
\tikzstyle{edge}=[draw,->]
\tikzstyle{cycle}=[draw,->,loop,out=130,in=50,distance=30pt]
   
\node[vertex, label=left:$v_1$] (0) at (-1,0) {};
\node[vertex, label=right:$v_2$] (00) at (1,0) {};
\node[vertex, label=below:$v_3$] (1) at (0,-1) {};

\path (0) edge[cycle] node[above] {} (0);
\path (00) edge[cycle] node[above] {} (00);
\path (0) edge[edge] node[above] {} (00);
\path (0) edge[edge] node[left] {} (1);
\path (00) edge[edge] node[right] {} (1);
\end{tikzpicture}
\quad
\begin{tikzpicture}[auto,swap]
\tikzstyle{vertex}=[circle,fill=black,minimum size=3pt,inner sep=0pt]
\tikzstyle{edge}=[draw,->]
\tikzstyle{cycle}=[draw,->,loop,out=130,in=50,distance=30pt]
   
\node[vertex, label=left:$v_1$] (0) at (0,0) {};
\node[vertex, label=right:$v_2$] (1) at (1,0) {};

\path (0) edge[cycle] node[above] {} (0);
\path (0) edge[edge] node[below] {} (1);
\path (1) edge[cycle] node[above] {} (1);
\end{tikzpicture}
\quad
\begin{tikzpicture}[auto,swap]
\tikzstyle{vertex}=[circle,fill=black,minimum size=3pt,inner sep=0pt]
\tikzstyle{edge}=[draw,->]
\tikzstyle{cycle}=[draw,->,loop,out=130,in=50,distance=30pt]
   
\node[vertex, label=left:$v_1$] (0) at (0,0) {};
\path (0) edge[cycle] node[above] {} (0);
\end{tikzpicture}
$$
\bigskip

\item[\bf Problem 25]
Let $k$ be a field and let $E$ be the following graph:
$$
\begin{tikzpicture}[auto,swap]
\tikzstyle{vertex}=[circle,fill=black,minimum size=3pt,inner sep=0pt]
\tikzstyle{edge}=[draw,->]
   
\node[vertex] (0) at (0,0) {};
\node[vertex] (1) at (1,0) {};

\path (0) edge[edge,bend right] node {} (1);
\path (0) edge[edge,bend left] node {} (1);

\end{tikzpicture}
$$
Compute all idempotents ($x^2=x$) in the path algebra $kE$.
\bigskip

\item[\bf Solution:]
Let $v$ be the left vertex of $E$, let $w$ be the right  vertex of $E$, and let $e$ and $f$ be the two edges in $E$.
Every element of $kE$ is of the form
$$
x=\lambda_1\chi_v+\lambda_2\chi_w+\alpha_1\chi_e+\alpha_2\chi_f,\qquad \lambda_1,\lambda_2,\alpha_1,\alpha_2\in k.
$$
Therefore,
\begin{align*}
x^2&=(\lambda_1 \chi_v + \lambda_2 \chi_w + \alpha_1 \chi_e+ \alpha_2 \chi_f)(\lambda_1 \chi_v + \lambda_2 \chi_w + \alpha_1 \chi_e+ \alpha_2 \chi_f)  \\
&=\lambda^2_1 \chi_v+\lambda_1 \alpha_1 \chi_e+\lambda_1 \alpha_2 \chi_f+\lambda^2_2 \chi_w+\lambda_2\alpha_1 \chi_e+\lambda_2\alpha_2 \chi_f\,.
\end{align*}
Hence, remembering that for any finite path $p$ the element $\chi_p$ is a basis element, the idempotent equation $x^2=x$ yields
$$
\lambda^2_1=\lambda_1,\quad \lambda_2^2=\lambda_2\quad \iff \quad
\boxed{\lambda_1=0\text{ or }1, \quad \lambda_2=0\text{ or } 1}\,;
$$
$$
(\lambda_1+\lambda_2)\alpha_1=\alpha_1,\quad (\lambda_1+\lambda_2)\alpha_2=\alpha_2
$$
$$
\iff \quad \boxed{\lambda_1+\lambda_2=1\quad\text{or}\quad \alpha_1=0}\quad\text{ and }
\quad \boxed{\lambda_1+\lambda_2=1\quad\text{or}\quad \alpha_2=0}.
$$
We consider all possibilities:
\begin{enumerate}
\item 
$\lambda_1=\lambda_2=0$. Then $\alpha_1=\alpha_2=0$, and consequently $x=0$.
\item 
$\lambda_1=\lambda_2=1$. Then $\alpha_1=\alpha_2=0$, and consequently $x=\chi_v+\chi_w=1$.
\item 
$\lambda_1=1$ and $\lambda_2=0$. Then $\alpha_1$ and $\alpha_2$ are arbitrary, and
$x=\chi_v+\alpha_1\chi_e+\alpha_2\chi_f$.
\item 
$\lambda_1=0$ and $\lambda_2=1$. Then $\alpha_1$ and $\alpha_2$ are arbitrary, and
$x=\chi_w+\alpha_1\chi_e+\alpha_2\chi_f$.
\end{enumerate}
\bigskip

\item[\bf Problem 26]
Using the pullback theorem (Theorem~2.9 in the lecture notes), prove
that, for any two row-finite graphs $E$ and $F$ such that $E\cap F=\emptyset$, we have an isomorphism of algebras
$$
L_k(E\cup F)\cong L_k(E)\oplus L_k(F).
$$
If in addition both $E$ and $F$ are non-empty, show also that $E\cup F$ is \emph{not} a connected graph.
\bigskip

\item[\bf Solution:]
Since the Leavitt path algebra $L_k(\emptyset)$ of the empty graph is zero, the canonical quotient maps
$L_k(E)\stackrel{\pi_1}{\to} L_k(E\cap F)\stackrel{\pi_1}{\leftarrow} L_k(F)$ are zero. 
Furthermore, as both graphs are row finite, and the empty graph is always an admissible subgraph,
Theorem~2.9 applies, so $L_k(E\cup F)\cong P(\pi_1,\pi_2)=L_k(E)\oplus L_k(F)$.
Finally, if $E^0\neq\emptyset\neq F^0$ and $E^0\cap F^0=\emptyset$, there exist $v\in E^0$ and
$w\in F^0$ such that $v\neq w$.  Suppose that $E\cup F$ is connected. 
Then there exists an unoriented path between $v$ and $w$.  It must
contain an edge joining a vertex in $E^0$ with a vertex in $F^0$, 
but such an edge does not exist because, as $E^0\cap F^0=\emptyset$, it neither can belong to $E^1$
nor to $F^1$, and $(E\cup F)^1= E^1\cup F^1$.
\end{enumerate}

\end{document}